\documentclass[11pt]{article}
\usepackage{amsmath,amsfonts,amssymb,amsthm} 
\usepackage{bm} 
\usepackage{enumerate} 
\usepackage{indentfirst}
\usepackage{color} 
\usepackage{float}
\usepackage{graphics}
\usepackage{epsfig}
\usepackage{amsbsy}
\usepackage{mathrsfs}
\usepackage{comment}
\usepackage{empheq}
\usepackage{hyperref}
\hypersetup{colorlinks=true, 
	linkcolor=blue 
}
\usepackage{authblk}
\allowdisplaybreaks
\newtheorem{theorem}{Theorem}[section]
\newtheorem{lemma}[theorem]{Lemma}
\newtheorem{proposition}[theorem]{Proposition}
\newtheorem{corollary}[theorem]{Corollary}

\newtheorem{remark}[theorem]{Remark}

\def\R{{\mathbb{R}}}
\def\H{{I\!\!H}}

\def\CC{{\rm \kern.24em \vrule width.02em height1.4ex depth-.05ex
		\kern-.26emC}}
	
\def\Xint#1{\mathchoice
	{\XXint\displaystyle\textstyle{#1}}%
	{\XXint\textstyle\scriptstyle{#1}}%
	{\XXint\scriptstyle\scriptscriptstyle{#1}}%
	{\XXint\scriptscriptstyle\scriptscriptstyle{#1}}%
	\!\int}
\def\XXint#1#2#3{{\setbox0=\hbox{$#1{#2#3}{\int}$}
		\vcenter{\hbox{$#2#3$}}\kern-.13\wd0}}
\def\dashint{\Xint-}

\def\Xint#1{\mathchoice
	{\XXint\displaystyle\textstyle{#1}}%
	{\XXint\textstyle\scriptstyle{#1}}%
	{\XXint\scriptstyle\scriptscriptstyle{#1}}%
	{\XXint\scriptscriptstyle\scriptscriptstyle{#1}}%
	\!\int}
\def\XXint#1#2#3{{\setbox0=\hbox{$#1{#2#3}{\int}$}
		\vcenter{\hbox{$#2#3$}}\kern-.5\wd0}}

\newcommand{\pvint}{\mathop{\mathrlap{\pushpv}}\!\int}
\newcommand{\pushpv}{\mathchoice
	{\mkern10.5mu\rule[.6ex]{.5em}{.5pt}}
	{\mkern2.8mu\rule[.5ex]{.35em}{.8pt}}
	{\mkern2.5mu\rule[.50ex]{.7em}{.7pt}}
	{\mkern2mu\rule[.5ex]{.7em}{.5pt}}
}

\newcommand{\ppvint}{\mathop{\mathrlap{\pushppv}}\!\int}
\newcommand{\pushppv}{\mathchoice
	{\mkern13.9mu\rule[.6ex]{.5em}{.5pt}}
	{\mkern2.8mu\rule[.5ex]{.35em}{.8pt}}
	{\mkern2.5mu\rule[.50ex]{.7em}{.7pt}}
	{\mkern2mu\rule[.5ex]{.7em}{.5pt}}
}

\def\TagOnRight

\def\AA{{it I} \hskip-3pt{\tt A}}

\def\QQ{\rlap {\raise 0.4ex \hbox{$\scriptscriptstyle |$}} {\hskip -0.1em Q}}

\catcode`\@=11

\def\theequation{\@arabic{\c@section}.\@arabic{\c@equation}}

\renewcommand{\div}{{\mathrm{div}}}

\newcommand{\DT}{\mathbb{D}}

\newcommand{\HC}[2]{\mathcal{C}^{{#1},{#2}}}
\renewcommand{\L}[1]{L^{#1}(\Omega)}
\newcommand{\Lb}[1]{L^{#1}(\Gamma)}
\newcommand{\vL}[1]{\bm{L}^{#1}(\Omega)}
\newcommand{\vLb}[1]{\bm{L}^{#1}(\Gamma)}
\renewcommand{\H}[1]{H^{#1}(\Omega)}
\newcommand{\vH}[1]{\bm{H}^{#1}(\Omega)}

\newcommand{\vHfracb}[2]{\bm{H}^{\frac{#1}{#2}}(\Gamma)}
\newcommand{\vHfracbd}[2]{\bm{H}^{-\frac{#1}{#2}}(\Gamma)}
\newcommand{\W}[2]{W^{#1,#2}(\Omega)}
\newcommand{\vW}[2]{\bm{W}^{#1,#2}(\Omega)}
\newcommand{\Wfracb}[2]{W^{#1, #2}(\Gamma)}
\newcommand{\vWfracb}[2]{\bm{W}^{#1, #2}(\Gamma)}
\newcommand{\vE}[1]{\bm{E}^{#1}(\Omega)}

\newcommand{\vn}{\bm{n}}
\newcommand{\vt}{\bm{\tau}}
\newcommand{\vu}{\bm{u}} 

\newcommand{\KEYWORDS}[1]{\paragraph{Keywords:}#1\par}

\renewcommand{\Re}{\operatorname{Re}}
\renewcommand{\Im}{\operatorname{Im}}

\makeatletter
\newcommand{\leqnomode}{\tagsleft@true\let\veqno\@@leqno}
\newcommand{\reqnomode}{\tagsleft@false\let\veqno\@@eqno}
\makeatother

\title{Semigroup theory for the Stokes operator with Navier boundary condition on $L^p$ spaces}
\author[1]{C.~Amrouche \thanks{cherif.amrouche@univ-pau.fr}}
\author[2]{M.~Escobedo \thanks{miguel.escobedo@ehu.eus}}
\author[1,2]{A.~Ghosh \thanks{amrita.ghosh@univ-pau.fr}}

\affil[1]{LMAP, UMR CNRS 5142, Universit\'{e} de Pau et des Pays de l'Adour, France.}
\affil[2]{Departamento de Matem\'{a}ticas, Universidad del Pa\'{i}s Vasco, 48940 Lejona, Spain.}

\begin{document}

\maketitle

\begin{abstract}
	We consider the incompressible Navier-Stokes equations in a bounded domain with $\HC{1}{1}$ boundary, completed with slip boundary condition. Apart from studying the general semigroup theory related to the Stokes operator with Navier boundary condition where the slip coefficient $\alpha$ is a non-smooth scalar function, our main goal is to obtain estimate on the solutions, independent of $\alpha$. We show that for $\alpha$ large, the weak and strong solutions of both the linear and non-linear system are bounded uniformly with respect to $\alpha$. This justifies mathematically that the solution of the Navier-Stokes problem with slip condition converges in the energy space to the solution of the Navier-Stokes with no-slip boundary condition as $\alpha \to \infty$.
\end{abstract}


\section{Introduction}

In this article we prove the existence of solutions to the following problem for the Navier Stokes equation:

\begin{empheq}[left=\empheqlbrace]{align}
&\frac{\partial\boldsymbol{u}}{\partial t} - \Delta \boldsymbol{u
}+ (\boldsymbol{u} \cdot \nabla) \boldsymbol{u}+\,\nabla\pi=\boldsymbol{0} &&\textrm{in} \,\,\Omega\times (0,T) \label{lens0} \\ 
&\mathrm{div}\,\boldsymbol{u}= 0 &&\textrm{in} \,\,\Omega\times (0,T) \label{lens2}\\
&\boldsymbol{u}\cdot \boldsymbol{n}=0 &&\textrm{on}\,\, \Gamma\times (0,T)\label{lens3}\\
& 2[(\DT\boldsymbol{u})\vn]_{\vt}+\alpha\vu_{\vt} = \bm{0} &&\textrm{on} \,\,\Gamma\times (0,T)\label{lens4}\\
&\boldsymbol{u}(0)=\boldsymbol{u}_{0} &&\textrm{in} \,\,\Omega\label{lens5}
\end{empheq}
where $\Omega $ is a bounded domain of ${\mathbb{R}}^3$, not necessarily simply connected, whose boundary $\Gamma $ is of class $C^{1,1}$. The initial velocity $\boldsymbol{u}_{0}$ and the (scalar) friction coefficient $\alpha $ are given functions. The external unit vector on $\Gamma $ is denoted by $\boldsymbol{n}$, $\DT\boldsymbol{u} = \frac{1}{2}\left( \nabla \vu + \nabla \vu^T\right)$ denotes the strain tensor and the subscript ${\bm{\tau}}$ denotes the tangential component i.e. $\bm{v}_{\vt} = \bm{v}- (\bm{v}\cdot \vn)\vn$ for any vector field $\bm{v}$. The functions $\boldsymbol{u}$ and $\pi$ describe respectively the velocity and the pressure of a viscous incompressible fluid in $\Omega$ satisfying the boundary conditions (\ref{lens3}), (\ref{lens4}).

The boundary condition (\ref{lens4}) was introduced by H. Navier in \cite{Navier}, taking into account the molecular interactions with the boundary and is called \textit{Navier boundary condition}. It may be deduced from kinetic theory considerations, as  first described in \cite{Maxwell} and rigorously proved under suitable conditions in \cite{Masmoudi3}. It has been  widely studied in recent years, because of its significance in modeling and simulations of flows and fluid-solid interaction problems (cf. \cite{IS}, \cite{Masmoudi2}, \cite{Pal} and  references therein). In that context the function $\alpha $ is, up to some constant,  the inverse of the slip length. We then impose  the condition $\alpha \ge 0$ in all the remaining of this work.

The problem with Dirichlet boundary condition has deserved a lot of attention. In particular, a good semigroup theory has been developed  in a series of work by Giga (cf. \cite{Giga81PJA}, \cite{Giga86}, \cite{giga-miya}, \cite{Giga91}). 
Here we wish to establish a  similar framework for  $\alpha \not \equiv 0$. 
We will study two different types of solutions for (\ref{lens1}), (\ref{lens2})-(\ref{lens5}): \textit{strong solutions} which belong to $L^p(0, T, \boldsymbol{L}^q(\Omega))$ type spaces and \textit{weak solutions} (in a suitable sense) that may be written for a.e. $t>0$, as $\boldsymbol{u}(t)=\boldsymbol{v}(t)+\nabla w(t)$ where 
$\boldsymbol{v}\in L^p(0, T; \boldsymbol{L}^{q}(\Omega))$ and $w\in L^p(0, T; L^{q}(\Omega))$. We will also consider different hypothesis on regularity of the function $\alpha $. In particular, we collect some of the relevant results available for the Navier-Stokes problem with no-slip condition based on semigroup properties and prove them for the system (\ref{lens0})-(\ref{lens5}) for the sake of completeness, so that this paper can be used as a basis for further work.

Let us give an overview of some related works. The system (\ref{lens0})-(\ref{lens5}) has been studied in \cite{mikelic} in two-dimension where $\alpha\ge 0$ is a function in $ C^2(\Gamma)$, with the main objective to analyse vanishing viscosity limit where the existence of weak and strong solutions have been established. Also in \cite{cipriano}, the authors have studied stochastic Navier-Stokes equation with Navier boundary condition, similar to (\ref{lens0})-(\ref{lens5}) where they considered same assumption that $\alpha\ge 0$ is in $C^2(\Gamma)$ and proved existence of weak solution. Beirao Da Veiga \cite{BDV} has considered the same problem in 3D in $\HC{2}{1}$ domain with $\alpha\ge 0$ constant and first showed that the Stokes operator with Navier boundary condition $A$ is a maximal monotone, self-adjoint operator on $\bm{L}^2_{\sigma,\vt}(\Omega)$ which generates an analytic semigroup of contraction and thus obtain strong solution of Stokes problem; Also by identifying the domain of $A^{1/2}$, he obtained the global strong solution of Navier-Stokes equation under the assumption of small data as in the no-slip boundary condition. The system (\ref{lens0})-(\ref{lens5}) has also been studied by Tanaka et al \cite{itoh} in Sobolev-Slobodetskii spaces in point of view to analyze asymptotic behavior of the unsteady solution to the steady solution where they have considered $\Gamma $ belongs to $W^{\frac{5}{2}+l}_2 $ and $\alpha\in (0,1)$ is in $ W^{\frac{1}{2}+l, \frac{1}{4}+\frac{l}{2}}_2((0,\infty)\times \Gamma)$ with $l\in (\frac{1}{2},1)$ and proved existence of local in time, strong solution and global in time, strong solution for small data. Note that in this work, $\alpha$ depends on both time and space variable. 
We want to mention further the works of \cite{IS, Masmoudi} where though the main objective is again to study viscosity limit, in the first paper, Iftimie and Sueur show existence of global in time weak solution in $C([0,\infty);\bm{L}^2_{\sigma,\vt}(\Omega))\cap L^2_{loc}([0,\infty);\vH{1})$ for $ \bm{L}^2_{\sigma,\vt}(\Omega)$ initial value; by classical approach: first deriving some energy estimate and then using Galerkin method. There they considered $\alpha$ a scalar function of class $\mathcal{C}^2$, without a sign. And in the second paper, Masmoudi and Rousset worked on anistropic conormal Sobolev spaces considering smooth domain and $|\alpha|\le 1$. In the paper \cite{monniaux}, the authors aimed to prove the existence of global in time, strong solution of a similar problem assuming small data, in $\HC{1}{1}$ domain and $\alpha$ non-negative, H\"{o}lder continuous in time. Also we mention the work \cite{xiao} where Lagrangian Navier-Stokes problem (as a regularization system of classical Navier-Stokes equations) with vorticity slip boundary condition (which is close to boundary condition (\ref{lens4})) has been studied for non-negative smooth function $\alpha$ and existence of weak solution, global in time is obtained. 

Further, in \cite{benes}, Benes has established a unique weak solution, local in time for the Navier-Stokes system with mixed boundary condition: on some part of the boundary Navier condition with $\alpha =0$ is considered and on other part, Neumann type boundary condition. Similar result for Navier-Stokes problem with Navier-type boundary condition (which corresponds to $\alpha =0$) has been studied in, for example \cite{NK}.
Also for $\alpha =0$ the semigroup associated to equation (\ref{lens1}) with (\ref{lens2})-(\ref{lens5}) has been studied in \cite{Hind-Ahmed}.

In this work, we wish to study the general semigroup theory for any $p\in(1,\infty)$ for the Stokes operator with Navier boundary condition (NBC) with (possibly) minimal regularity on $\alpha$ which gives us existence, uniqueness and regularity of both strong and weak solutions of (\ref{lens1}), (\ref{lens2})-(\ref{lens5}). We start with introducing the strong and weak Stokes operators for each fixed $\alpha$ and show that they generate analytic semigroups on the respective spaces for all $p\in (1,\infty)$. The proof of this is not very complicated and mostly use the existence and estimate results studied in \cite{AG} for the steady problem. Further we study the imaginary and fractional powers of the Stokes operators. To show that the operators are of bounded imaginary power(BIP) is not straight forward; Here we did not use pseudo-differential operator theory or Fourier multiplier theory as done by Giga \cite{Giga91}, but chose a rather different approach: we show that the Stokes operator with (full) slip boundary condition (in its weak form) can be written as a lower order perturbation of the Navier-type boundary condition (cf. (\ref{20.})) for which the result (that the operator is BIP) is known (cf. \cite[Theorem 6.1]{AAE1}); And then with the help of Amann's interpolation-extrapolation theory \cite{amann}, we recover the boundedness of imaginary power of $A_{p, \alpha}$. This method has been used in \cite{pruss} to establish that the Stokes operator with NBC for $\alpha>0$ constant, possesses a bounded $\mathcal{H}^\infty$-calculus on $\bm{L}^p_{\sigma,\vt}(\Omega)$. Next we prove that the Stokes operator has maximal $L^q$-regularity and establish various types of $L^p-L^q$ estimates which helps to develop an $L^p$-theory for the Navier-Stokes equations. We have used the abstract theory by Giga \cite{{Giga86}} for semilinear parabolic equations in $L^p$ to achieve the similar existence and regularity of a local strong solution and global weak solution.

The interesting part is to show the resolvent estimate
$$
|\lambda| \|(\lambda I + A_{p,\alpha})^{-1}\bm{f}\|_{\bm{L}^p_{\sigma,\vt}(\Omega)} \le C \|\bm{f}\|_{\bm{L}^p_{\sigma,\vt}(\Omega)} \qquad \forall \,\lambda \in \mathbb{C}^* \,\,\text{ with }\,\, \Re \lambda \ge 0
$$
where the positive constant $C$ does not depend on $\alpha$ for $\alpha$ sufficiently large. Adopting the standard method, multiplying the equation by $|\vu|^{p-2}\bar{\vu}$ (for example, as followed in \cite{AAE}), one can easily obtain the above estimate but with the constant depending on $\alpha$. Therefore we need some different approach. In the Hilbert case, this follows (cf. Theorem \ref{17}) from the variational formulation as expected. But to prove it for $p\ne 2$ is much more delicate. We have used $L^p$-extrapolation theory of Shen (cf. Lemma \ref{L0}), with a suitable operator $T$. The main difficulty is to satisfy the necessary condition (\ref{extrapolation_condition}) which is a weak reverse H\"{o}lder inequality (wRHI). Unfortunately, the estimates derived on the steady problem do not help in this situation. 

It seems  very natural  to let $\alpha \to \infty$ in some sense in order  to obtain the Dirichlet boundary condition on $\Gamma $ from the condition (\ref{lens4}).
The main goal of the present article is to study the behavior of the solutions of the unsteady Stokes and Navier-Stokes equation with NBC with respect to $\alpha$, in particular what happens when $\alpha$ goes to $\infty$. This problem  is considered in  \cite{kelliher} in 2D where the author shows in Theorem 9.2 that, for $\vu_0\in\vH{3}$, when $\|1/\alpha\| _{ L^\infty (\Gamma ) }\rightarrow 0$, the solution of problem (\ref{lens0})-(\ref{lens5}) converges to the solution of the Navier-Stokes problem with Dirichlet boundary condition in suitable spaces (cf. Section \ref{9} for details). In our work, we first deduce uniform resolvent estimates with respect to some suitable norm of the function $\alpha$  which in turn provides $\alpha$ independent estimates for the solutions of non-stationary Stokes problem with NBC. This enables us to consider the case where  $\alpha $ is a constant function and $\alpha \to \infty$ in both the linear problem (\ref{lens1}), (\ref{lens2})-(\ref{lens5}) and the nonlinear problem (\ref{lens0})-(\ref{lens5}). We show that   the solutions of  the problems with NBC converge strongly in the energy space to the solutions of corresponding problem  with Dirichlet boundary condition as $\alpha$ goes to $\infty$. 

We state now our main results, for which the following notations are needed:
$$\bm{L}^p_{\sigma,\vt}(\Omega) = \lbrace \bm{v}\in\vL{p}; \div \
\bm{v} = 0 \text{ in } \Omega, \bm{v}\cdot \vn = 0 \text{ on } \Gamma\rbrace $$
equipped with the norm of $\vL{p}$ and
\begin{eqnarray*}
\mathbf{D}(A_{p,\alpha}) = \left\{ \vu\in\vW{2}{p}\cap \bm{L}^p_{\sigma,\vt}(\Omega); \,\, 2[(\DT\vu)\vn]_{\vt}+\alpha\vu_{\vt} = \bm{0} \text{ on } \Gamma \right\}.
\end{eqnarray*}
The space $\mathbf{D}(A_{p,\alpha}) $ is nothing but the domain of $A_{p,\alpha}$, the Stokes operator on $\bm{L}_{ \sigma , \vt  }^p(\Omega )$ with the boundary conditions (\ref{lens3})-(\ref{lens4}).

\begin{theorem}
Suppose that $\alpha \in \Wfracb{1-\frac{1}{r_1}}{r_1}$ for some $r_1\ge 3$ and  $\alpha \ge 0$. Then,  for every $\vu_0 \in \bm{L}^{r_2}_{\sigma,\vt}(\Omega)$ with $r_2\ge 3$, there exists a unique solution $\vu$ of (\ref{lens0})--(\ref{lens5}) defined  on a maximal time interval $[0, T_\star)$ such that
\begin{equation*}
\vu \in C([0,T_\star);\bm{L}^r_{\sigma,\vt}(\Omega)) \cap L^q(0,T_\star; \bm{L}^p_{\sigma,\vt}(\Omega))
\end{equation*}
\begin{equation*}
t^{1/q}\vu \in C([0,T_\star);\bm{L}^p_{\sigma,\vt}(\Omega)) \quad \text{ and } \quad t^{1/q}\|\vu\|_{\vL{p}} \rightarrow 0 \text{ as } t\rightarrow 0
\end{equation*}
with $r=\min (r_1, r_2)$,  $p>r$, $q>r$ and $\frac{2}{q}= \frac{3}{r}-\frac{3}{p}$.  Moreover,
$$\boldsymbol{u}\in C((0,T_{\star}),\mathbf{D}(A_{r,\alpha}))\cap C^{1}((0,T_{\star});\,\boldsymbol{L}^{r}_{\sigma,\vt}(\Omega)).$$ 

If $r>3$ and $T_\star<\infty$,
\begin{equation*}
\|u(t)\|_{\vL{r}} \geq C (T_\star -t)^{(3-r)/2r}
\end{equation*}
where $C$ is independent of $T_\star$ and $t$.

Also, there exists a constant $\varepsilon >0$ such that if $\|\vu_0\|_{\vL{3}} < \varepsilon$, then $T_\star=\infty$.
\end{theorem}
Under weaker conditions on $\Omega $ and $\alpha $, a similar Theorem holds for initial data in the space of distributions $\vu_0 =\bm{\psi}+\nabla \chi $  where  $\bm{\psi}  \in L^{r}(\Omega)$ and  $\chi \in L^{r}(\Omega)$ (denoted by $[{\bf H}_0^{r'}(\div, \Omega )]'$,  cf. Proposition \ref{prop}), with $r\ge 3$.

The proof of these results is based on a careful study of the semigroup associated to the linear equation
\begin{eqnarray}
&&\frac{\partial\boldsymbol{u}}{\partial t} - \Delta \boldsymbol{u
}\,+\,\nabla\pi=\boldsymbol{f} \label{lens1}
\end{eqnarray}
with conditions (\ref{lens2})-(\ref{lens5}). For that we first study the strong and weak Stokes operators $A_{p,\alpha}$ and $B_{p,\alpha}$ and deduce that both of them have bounded inverse on $\bm{L}^p_{\sigma,\vt}(\Omega)$ and $[{\bf H}_0^{p'}(\div, \Omega )]'$ respectively for all $p\in (1,\infty)$. Also $-A_{p,\alpha}$ and $-B_{p,\alpha}$ generate bounded analytic semigroups on their respective spaces (cf. Theorem \ref{24} and Theorem \ref{S4Thm1}) and their pure imaginary powers are uniformly bounded as well (cf. Theorem \ref{62}). We obtain the following theorems, if $\bm{ f}=\bm{0}$:

\begin{theorem}
\label{homstrong}
Let $1<p<\infty$ and $\alpha\ge 0$ be as in \eqref{alpha}. Then for $\vu_0\in \bm{L}^p_{\sigma,\vt}(\Omega)$, the problem \eqref{evolutionary Stokes} has a unique solution $\vu(t)$ satisfying
\begin{equation*}
\vu\in C([0,\infty),\bm{L}^p_{\sigma,\vt}(\Omega) )\cap C((0,\infty),\mathbf{D}(A_{p,\alpha}))\cap C^1((0,\infty),\bm{L}^p_{\sigma,\vt}(\Omega)) 
\end{equation*}
and
\begin{equation*}
\vu\in C^k((0,\infty),\mathbf{D}(A_{p,\alpha}^l)) \quad \forall \ k\in\mathbb{N}, \ \forall \ l\in\mathbb{N}\backslash\{0\} .
\end{equation*}
Also, for all $t>0$ and $q\ge p$,  $\vu(t)\in L^q(\Omega )$ and there exists $\delta >0$ independent of $t$ and $q$ such that:
\begin{equation*}
\|\vu(t)\|_{\vL{q}} \leq C(\Omega, p) \ e^{-\delta t}t^{-3/2 (1/p-1/q)} \|\vu_0\|_{\vL{p}}.
\end{equation*}
Moreover, the following estimates also hold
\begin{equation*}
\|\DT{\vu(t)}\|_{\vL{q}} \leq C(\Omega,p) \ e^{-\delta t}t^{-3/2(1/p-1/q)-1/2} \|\vu_0\|_{\vL{p}} ,
\end{equation*}
\begin{equation*}
\forall \ m,n\in\mathbb{N}, \quad \|\frac{\partial ^m}{\partial t^m}A_{p,\alpha}^n \vu(t)\|_{\vL{q}} \leq C(\Omega, p) \ e^{-\delta t}t^{-(m+n)-3/2(1/p-1/q)} \|\vu_0\|_{\vL{p}} .
\end{equation*}
\end{theorem}
For $\bm f\not \equiv \bm{0}$, and if we denote $E_q$ the following real interpolation space:
$$
E_q\equiv (D(A_{p,\alpha}), \bm{L}^p_{\sigma, \vt}(\Omega)) _{ \frac {1} {q}, q }
$$
we have the result:

\begin{theorem}
\label{nonhom}
Let $1< p,q <\infty$. Also assume that $0< T \leq \infty$ and $\alpha\ge 0$ be as in \eqref{alpha}. Then for $\bm{f}\in L^q(0,T;\bm{L}^p_{\sigma, \vt}(\Omega))$ and $\vu_0 \in E_q$, there exists a unique solution $(\vu,\pi)$ of \eqref{58} satisfying:
	\begin{equation*}
	\vu \in L^q(0,T_0;\vW{2}{p}) \ \text{ for all } \ T_0\leq T \ \text{ if } \ T<\infty \ \text{ and } \ T_0 < \infty \ \text{ if } \ T= \infty ,
	\end{equation*}
	
	\begin{equation*}
	\pi\in L^q(0,T; \W{1}{p}/\R) , \quad \frac{\partial \vu}{\partial t} \in L^q(0,T, \bm{L}^p_{\sigma, \vt}(\Omega)),
	\end{equation*}
	
	\begin{equation}
	\begin{aligned}
	\label{59}
	\int\displaylimits_{0}^{T}\left\| \frac{\partial \vu}{\partial t}\right\| ^q_{\vL{p}} \mathrm{d}t + \int\displaylimits_{0}^{T}\|\vu\|^q_{\vW{2}{p}} \mathrm{d}t + \int\displaylimits_{0}^{T}\|\pi\|^q_{\W{1}{p}} \mathrm{d}t \leq C \left( \int\displaylimits_{0}^{T} \|\bm{f}\|^q_{\vL{p}} \mathrm{d}t + \|\vu_0\|^q_{E_q} \right)
	\end{aligned}
	\end{equation}
	where $C>0$ is independent of $\bm{f}, \vu_0$ and $T$. 
\end{theorem}
Similar results hold for less regular data  (cf. Theorem \ref{37} and Theorem \ref{S7Thm1}). And the last interesting result of our work is the following limit problem which improves the result in \cite[Theorem 9.2]{kelliher}:

\begin{theorem}
	\label{S9T2}
	Let $\alpha$ be a constant and $ (\vu_\alpha, \pi_\alpha)$ be the solution of the problem (\ref{NS}) with $\vu_0 \in \bm{L}^2_{\sigma,\vt}(\Omega)$ and
	$(\vu_{\infty},\pi_\infty)\in  L^2(0,T; \bm{H}^1_0(\Omega))\times L^2(0,T;L^2(\Omega))$ a solution of the following Navier-Stokes problem with Dirichlet boundary condition
	\begin{equation}
	\label{NSD}
	\left\{
	\begin{aligned}
	\frac{\partial \vu_\infty}{\partial t} -\Delta \vu_\infty + (\vu_\infty\cdot \nabla) \vu_\infty + \nabla \pi_\infty = \bm{0} , \quad \div \ \vu_\infty = 0 \quad & \text{ in } \Omega\times (0,T) ;\\
	\vu_\infty = \bm{0} \quad & \text{ on } \Gamma\times (0,T) ;\\
	\vu_\infty(0) = \vu_0 \quad & \text{ in } \Omega
	\end{aligned}
	\right.
	\end{equation}
	(whose existence has been proved in \cite[Theorem 3.1, Chapter III]{Temam}). 
	Then for any $T<T_*$,
	\begin{equation*}
	\vu_{\alpha} \rightarrow \vu_{\infty}\quad \text{ in } \quad L^2(0,T;\vH{1}) \quad \text{ as } \quad \alpha\to \infty
	\end{equation*}
	and
	\begin{equation}
	\label{60}
	\int\displaylimits_{0}^{T}\int\displaylimits_{ \Gamma  }{ |\vu_\alpha-\vu_\infty |^2}\le \frac{C}{\alpha}.
	\end{equation}
	Moreover, if $\Gamma$ is $\HC{2}{1}$ and $\vu_0 \in \vH{2}\cap\bm{H}^1_{0,\sigma}(\Omega)$, we also have
	\begin{equation}
	\label{61}
	\sup_{t\in[0,T]}\int\displaylimits_{ \Omega  }{|\vu_\alpha-\vu_\infty |^2}+\int\displaylimits_{0}^{T}\int\displaylimits_{\Omega}{|\DT(\vu_\alpha-\vu_\infty)|^2}+	\int\displaylimits_{0}^{T}\int\displaylimits_{ \Gamma  }{ |\vu_\alpha-\vu_\infty |^2}\le \frac{C}{\alpha}.
	\end{equation}
\end{theorem}

Similar result holds for the linear problem as well as for initial data in $L^p$-spaces (cf. Section \ref{9}).

\section{Preliminaries}
\setcounter{equation}{0}
First, we review some basic notations and functional framework, we will use in the study. Throughout the work, if we do not specify otherwise, $\Omega$ is an open bounded set in $\R^{3}$ with boundary $\Gamma$ of class $\mathcal{C}^{1,1}$ possibly multiply connected. Also if we do not precise otherwise, we will always assume
$$
\alpha \ge 0 \quad \text{ on } \Gamma \quad \text{ and }
\alpha> 0 \quad \text{ on some } \Gamma_0\subset \Gamma \quad \text{ with } |\Gamma_0|>0.
$$ We follow the convention that $C$ is an unspecified positive constant that may vary from expression to expression, even across an inequality (but not across an equality); Also $C$ depends on $\Omega$ and $p$ generally and the dependence on other parameters will be specified in the parenthesis when necessary. 

The vector-valued Laplace operator of a vector field $\bm{v}$ can be equivalently defined by
$$\Delta \bm{v} = 2 \ \div \ \DT\bm{v} - \textbf{grad} \ (\div \ \bm{v}).$$
We will denote by $\pmb{\mathscr{D}}({\Omega})$ the set of smooth functions (infinitely differentiable) with compact support in $\Omega$. Define
$$\pmb{\mathscr{D}}_\sigma({\Omega}) := \left\{ \bm{v}\in \pmb{\mathscr{D}}({\Omega}) : \div \ \bm{v} = 0 \text{ in } \Omega \right\}$$
and
$$
L^p_0(\Omega) := \left\{ v\in \L{p} : \int\displaylimits_{ \Omega  }{v}=0 \right\}.
$$
For $p\in [1,\infty)$, $p'$ denotes the conjugate exponent of $p$ i.e. $\frac{1}{p}+\frac{1}{p'}=1$. We also introduce the following space:
$$\bm{H}^{p}(\div, \Omega) := \left\{ \bm{v}\in \bm{L}^p(\Omega) : \div \ \bm{v}\in L^p(\Omega)\right\} $$
equipped with the norm
$$\|\bm{v}\|_{\bm{H}^{p}(\div, \Omega)} = \|\bm{v}\|_{\bm{L}^p(\Omega)} + \|\div \ \bm{v}\|_{L^p(\Omega)}.$$
It can be shown that $\pmb{\mathscr{D}}({\overline \Omega})$ is dense in $\bm{H}^{p}(\div, \Omega)$ for all $p\in [1,\infty)$. The closure of $\pmb{\mathscr{D}}({\Omega})$ in $\bm{H}^{p}(\div, \Omega)$ is denoted by $\bm{H}^{p}_0(\div, \Omega)$ and it can be characterized by
\begin{equation}
\label{S2E1}
\bm{H}^{p}_0(\div, \Omega) = \left\{\bm{v}\in \bm{H}^{p}(\div, \Omega) : \bm{v}\cdot \bm{n}=0 \ \text{ on } \ \Gamma\right\}.
\end{equation}
For $p\in (1,\infty)$, we denote by $[\bm{H}^{p}_0(\div, \Omega)]^\prime$, the dual space of $\bm{H}^{p}_0(\div, \Omega)$, which can be characterized as (for details, see \cite[Proposition 1.0.4]{nour}):

\begin{proposition}
\label{prop}
	A distribution $\bm{f}$ belongs to $[\bm{H}^{p}_0(\div, \Omega)]^\prime$ iff there exists $\bm{\psi} \in \bm{L}^{p^\prime}(\Omega)$ and $\chi\in L^{p^\prime}(\Omega)$ such that $\bm{f} = \bm{\psi}+ \nabla \chi$. Moreover, we have the estimate :
	$$\|\bm{f}\|_{[\bm{H}^{p}_0(\div, \Omega)]^\prime} = \underset{\bm{f} = \bm{\psi}+ \nabla \chi}{\inf} \ \max\{\|\bm{\psi}\|_{\bm{L}^{p^\prime}(\Omega)}, \|\chi\|_{L^{p^\prime}(\Omega)}\} .$$
\end{proposition}

\noindent Next we introduce the spaces:
$$\bm{L}^p_{\sigma,\vt}(\Omega) := \lbrace \bm{v}\in\vL{p} : \div \
\bm{v} = 0 \text{ in } \Omega, \bm{v}\cdot \vn = 0 \text{ on } \Gamma\rbrace $$
equipped with the norm of $\vL{p}$;
$${\bf W} ^{1, p} _{ \sigma , \vt }(\Omega ):=\lbrace \bm{v}\in\vW{1}{p} : \div \ \bm{v} = 0 \text{ in } \Omega, \bm{v}\cdot \vn = 0 \text{ on }\Gamma\rbrace $$
equipped with the norm of $\vW{1}{p}$ and ${\bf H} ^{1} _{ \sigma , \vt  }(\Omega ):= {\bf W} ^{1, 2} _{ \sigma , \vt }(\Omega )$. Also let us define
\begin{equation*}
{\bf E}^p(\Omega ):=\left\{(\vu , \pi )\in {\bf W}^{1, p}(\Omega )\times L^p(\Omega ); -\Delta \vu +\nabla \pi \in \bm{L}^{r(p)}(\Omega) \right\}
\end{equation*}
where
\begin{eqnarray*}
\begin{cases}
r(p)= \max\left\lbrace 1,\frac {3p} {p+3}\right\rbrace \,\,\,&\hbox{if}\,\,\,p\ne \frac{3}{2}\\
r(p)>1\,\,\,&\hbox{if}\,\,\,p= \frac{3}{2} .
\end{cases}
\end{eqnarray*}

Let us now introduce some notations to describe the boundary. Consider any point $P$ on $\Gamma$ and choose an open neighborhood $W$ of $P$ in $\Gamma$, small enough to allow the existence of 2 families of $\mathcal{C}^2$ curves on $W$ with the following properties: a curve of each family passes through every point of $W$ and the unit tangent vectors to these curves form an orthogonal system (which we assume to have the direct orientation) at every point of $W$. The lengths $s_1, s_2$ along each family of curves, respectively, are a possible system of coordinates in $W$. We denote by $\vt_1, \vt_2$ the unit tangent vectors to each family of curves.

With these notations, we have $\bm{v}=\sum_{k=1}^{2}v_k \vt_k + (\bm{v}\cdot \vn) \vn$ where $ \vt_k = (\tau_{k1},\tau_{k2},\tau_{k3})$ and $v_k = \bm{v}\cdot \vt_k$. For simplicity of notations, we will denote,
\begin{equation}
\label{lambda}
\Lambda \bm{v} = \sum_{k=1}^{2} \left( \bm{v}_{\vt}\cdot \frac{\partial \vn}{\partial s_k}\right) \vt_k \ .
\end{equation} 

Here we state a relation between the Navier boundary condition and another type of boundary condition involving $\mathbf{curl}$ (often called as 'Navier type boundary condition') which will be used in later work. For proof, see \cite[Appendix A]{AR}.
\begin{lemma}
For any $\bm{v} \in \vW{2}{p}$, we have the following equalities:
$$2\left[(\DT\bm{v})\vn\right]_{\vt} = \nabla_{\vt}(\bm{v}\cdot \vn) + \left( \frac{\partial \bm{v}}{\partial \vn}\right) _{\vt} - \bm{\Lambda}\bm{v} $$
and
$$\mathbf{curl} \ \bm{v} \times \vn = - \nabla_{\vt} (\bm{v}\cdot \vn)+ \left( \frac{\partial \bm{v}}{\partial \vn}\right) _{\vt} + \bm{\Lambda}\bm{v} .$$
\end{lemma}

\begin{remark}
\rm{In the particular case $\bm{v}\cdot \vn = 0$ on $\Gamma$, we have the following equalities for all $\bm{v} \in \vW{2}{p}$,
\begin{equation*}
2\left[(\DT\bm{v})\vn\right]_{\vt} = \left( \frac{\partial \bm{v}}{\partial \vn}\right) _{\vt} - \bm{\Lambda}\bm{v} \quad \text{ and } \quad \mathbf{curl} \ \bm{v} \times \vn = \left( \frac{\partial \bm{v}}{\partial \vn}\right) _{\vt} + \bm{\Lambda}\bm{v}
\end{equation*}
which implies that 
\begin{equation}
\label{19.}
2\left[(\DT\bm{v})\vn\right]_{\vt} = \mathbf{curl} \ \bm{v} \times \vn - 2\bm{\Lambda}\bm{v} \ .
\end{equation}
Note that on a flat boundary, $\bm{\Lambda} = \bm{0}$ and $2\left[(\DT\bm{v})\vn\right]_{\vt}$ is actually equal to $\mathbf{curl} \ \bm{v} \times \vn$.
}
\end{remark}

Let us recall the Green's formula that plays an important role in this work, which is proved in \cite[Lemma 3.5]{AG}
\begin{theorem}
\label{Green}
Let $\Omega \subset \R^3$ be a $C^{0, 1}$ bounded domain. Then,

(i) $\pmb{\mathscr{D}}({\overline\Omega})\times \mathscr{D}({\overline{\Omega}})$ is dense in ${\bf E}^{p}(\Omega )$.

(ii)  The linear mapping $\left( \bm{v},\pi\right) \mapsto\left[(\DT\bm{v})\vn\right]_{\vt}$, defined on $\pmb{\mathscr{D}}({\overline\Omega})\times \mathscr{D}({\overline{\Omega}})$ can be extended to a linear, continuous map from $\vE{p}$ to $\vWfracb{-\frac{1}{p}}{p}$. Moreover, we have the following relation: for all $\left( \bm{v},\pi\right) \in\vE{p}$ and $\bm{\varphi}\in {\bf W} _{ \sigma , \vt  }^{1, p'}(\Omega )$,
\begin{equation}
\label{green}
  \int\displaylimits_{\Omega}{\left( -\Delta\bm{v} + \nabla\pi\right) \cdot\bm{\varphi}} =2\int\displaylimits_{\Omega}{\DT\bm{v}:\DT\bm{\varphi}}-2
  \left\langle \left[(\DT\bm{v})\vn\right]_{\vt},{\bm{\varphi}}\right\rangle_{\vWfracb{-\frac{1}{p}}{p}\times \vWfracb{\frac{1}{p}}{p'}} .
  \end{equation}
\end{theorem}

Finally, we recall that the infinitesimal generator of an analytic semigroup can be characterized by the following theorem \cite[Theorem 3.2, Chapter I]{Barbu}:

\begin{theorem}
\label{analyticity}
	Let $A$ be a densely defined linear operator in a Banach space $\mathscr{E}$. Then $A$ generates an analytic semigroup on $\mathscr{E}$ if and only if there exists $M>0$ such that
	$$\|\left( \lambda I -A\right) ^{-1}\|_{\mathcal{L}(\mathscr{E})} \leq \frac{M}{|\lambda|}$$
	for all $\lambda \in \mathbb{C}$ with $\Re \lambda >w $.
\end{theorem}

\section{Stokes operator on $L^p_{\sigma,\tau}(\Omega)$: Strong solutions}
\label{66}
\setcounter{equation}{0}
It is known that  the closure of $\pmb{\mathscr{D}} _{ \sigma  }({\Omega})$ in $\vL{p}$ is the Banach space $\bm{L}^p_{\sigma,\vt}(\Omega)$ \cite[Theorem 1.4]{Temam}. We introduce now the unbounded operator $(A _{ p, \alpha  }, \mathbf{D}(A _{ p, \alpha  }))$ on $\bm{L}^p_{\sigma,\vt}(\Omega)$ whose definition depends on the regularity of the function $\alpha$.\\

\noindent
1. If $\alpha \in\Lb{t(p)}$ with
\begin{equation}
\label{22}
t(p)= \begin{cases}
\frac{2}{3}p^\prime + \rho & \text{if\quad}1< p<\frac{3}{2}\\
2+\rho & \text{if\quad}\frac{3}{2}\leq p\leq 3, p\neq 2\\
2 & \text{if \quad}p=2\\
\frac{2}{3}p + \rho & \text{if \quad}p>3
\end{cases}
\end{equation}
with $\rho >0$ arbitrarily small, we define the Stokes operator $A_{p,\alpha}$ on $\bm{L}^p_{\sigma,\vt}(\Omega)$ as:

\begin{empheq}[left=\empheqlbrace]{align}
&\mathbf{D}(A _{ p, \alpha  }) = \left\{ \vu\in\bm{W}^{1,p}_{\sigma,\vt}(\Omega):\Delta\vu \in \vL{p}, 2[(\DT\vu)\vn]_{\vt}+\alpha\vu_{\vt} = \bm{0} \text{ on } \Gamma \right\} \label{S3E1}\\
&A _{ p, \alpha  }(\vu) = -P(\Delta \vu) \quad \text{ for } \vu\in\mathbf{D}(A _{ p, \alpha  }) \label{S3E10}
\end{empheq}
where $P :\vL{p} \rightarrow \bm{L}^p_{\sigma,\vt}(\Omega)$ is the orthogonal projection on $\bm{L}^p_{\sigma,\vt}(\Omega)$.  More precisely, for all ${\bm \psi }\in \vL{p}$, $P({\bm \psi })={\bm \psi }-\nabla \pi $ where $\pi\in \W{1}{p}$ is a solution of
\begin{equation}
\label{pression}
\begin{cases}
&\div (\nabla \pi - {\bm \psi } )=0 \,\,\, \hbox{in}\,\,\,\Omega \\
&(\nabla \pi - {\bm \psi })\cdot \vn =0\,\,\,\hbox{on}\,\,\Gamma .
\end{cases}
\end{equation}
Notice that, when $\vu \in \mathbf{D}(A _{ p, \alpha  })$ but $\vu \not \in \vW{2}{p}$, the boundary term $[(\DT\vu)\vn]_{\vt}$ is still well defined as shown in \cite[Lemma 2.4]{AR}.
\\

\noindent
2. If  $\alpha$ is such that
\begin{equation}
\label{alpha}
\alpha \in \begin{cases}
\Wfracb{1-\frac{1}{\frac{3}{2}+\rho }}{\frac{3}{2}+\rho  } & \text{if\quad}1< p\leq\frac{3}{2}\\
\Wfracb{1-\frac{1}{p}}{p} & \text{if\quad}p > \frac{3}{2}
\end{cases}
\end{equation}
with $\rho >0$ arbitrarily small, then we define the Stokes operator $A _{ p, \alpha  }$ on $\bm{L}^p_{\sigma,\vt}(\Omega)$ as
\begin{empheq}[left=\empheqlbrace]{align}
&\mathbf{D}(A _{ p, \alpha  }) = \left\{ \vu\in\vW{2}{p}\cap \bm{L}^p_{\sigma,\vt}(\Omega), 2[(\DT\vu)\vn]_{\vt}+\alpha\vu_{\vt} = \bm{0} \text{ on } \Gamma \right\rbrace \label{S3E11} \\
&A _{ p, \alpha  }(\vu) = -P(\Delta \vu) \quad \text{ for } \vu\in\mathbf{D}(A _{ p, \alpha  }). \label{S3E12}
\end{empheq}

\begin{remark}
	\label{S3Ralpha}
	\rm{If $\alpha$ satisfies (\ref{alpha}), then $\alpha \in L^{t(p)}(\Gamma)$ as well.}
\end{remark}

\begin{remark}
\label{OperatorA}
\rm{When $\alpha $ satisfies  (\ref{alpha}) and $\vu \in {\bf D}( A _{ p, \alpha  })$, $A _{ p, \alpha  }\vu=A \vu$ where $A$ is the following operator defined in \cite[just before Proposition 5.6]{AG}:
\begin{eqnarray*}
&&A\in \mathcal L\left({\bf W}^{1, p} _{ \sigma , \vt }(\Omega ), \left({\bf W}^{1, p'} _{ \sigma , \vt  }(\Omega )\right)' \right)\\
\text{ defined by } &&\langle A \vu, {\bm v} \rangle= a(\vu, {\bm v})\\
\text{where } && a(\vu, {\bm v})=2\int\displaylimits_{\Omega}{\DT\vu:\DT\bar{\bm{v}}} + \int\displaylimits_{\Gamma}{\alpha \vu_{\vt}\cdot \bar{\bm{v}}_{\vt}}.
\end{eqnarray*}
More precisely, if $\vu\in {\bf D}( A _{ p, \alpha  })$, then using Green's formula (\ref{green}) and the relation $2[(\DT\vu)\vn]_{\vt}=-\alpha\vu_{\vt}$ on $\Gamma$, we can show $\langle A \vu, {\bm v} \rangle =\langle A _{ p, \alpha  } \vu, {\bm v} \rangle $ for any ${\bm v}\in {\bf W}^{1, p'} _{ \sigma , \vt}(\Omega )$.}
\end{remark}

\subsection{Analyticity }
\noindent In order to show that $({\bf D}(A_{p,\alpha}), A_{p,\alpha})$ generates an analytic semi group on $\bm{L}^p_{\sigma,\vt}(\Omega)$, $\lambda\in \mathbb{C}$,  some estimate on the resolvent $(\lambda I+A_p)^{-1}$ is needed.  

Suppose that  $\alpha \in L^{t(p)}(\Gamma )$ and ${\bm f} \in \bm{L}^p_{\sigma,\vt}(\Omega)$. Then by (\ref{S3E1})-(\ref{S3E10}), $\vu \in \mathbf{D}(A_{p,\alpha})$ and $(\lambda I+A_{p,\alpha})\vu = \bm{f}$  is equivalent to $\vu \in \vW{1}{p}$ satisfying
\begin{empheq}[left=\empheqlbrace]{align}
&\lambda \vu -\Delta \vu + \nabla \pi = \bm{f}\quad & \text{ in } \Omega  \label{S3.ER1}\\
& \div \ \vu = 0 \quad & \text{ in } \Omega \label{S3.ER2}\\
&\vu\cdot \vn = 0 \quad & \text{ on } \Gamma \label{S3.ER03}\\
&2[(\DT\vu)\vn]_{\vt}+\alpha \vu_{\vt} = \bm{0} \quad & \text{ on } \Gamma \label{S3.ER3}
\end{empheq}
for some $\pi \in L^p(\Omega )$.

If on the other hand, $\alpha $ satisfies (\ref{alpha}), $\vu \in \mathbf{D}(A_{p,\alpha})$ and $(\lambda I+A_{p,\alpha})\vu = \bm{f}$ for ${\bm f}\in \bm{L}^p_{\sigma,\vt}(\Omega)$,  is equivalent to $(\vu, \pi )\in \vW{2}{p} \times \W{1}{p}$ satisfying (\ref{S3.ER1})-(\ref{S3.ER3}).

\begin{proposition}
\label{S3P1}
Suppose that $\alpha \in L^{t(p)}(\Gamma)$,  $ {\bm f}\in [\bm{H}_0^{p'}(\div ,\Omega)]'$ with $p\in(1,\infty)$ and $\lambda\in \mathbb{C}$. Then,
$\vu\in \vW{1}{p}$ solves (\ref{S3.ER1})-(\ref{S3.ER3}) for some $\pi \in L^p(\Omega)$ is equivalent to $\vu \in {\bf W}^{1,p}_{ \sigma, \vt}(\Omega )$ satisfies:
\begin{equation}\label{variational formulation}
\left\{
\begin{aligned}
&a_\lambda  (\vu,\bm{\varphi}) = \langle {\bm{f}, \bar{\bm{\varphi}}} \rangle \quad \forall \ \bm{\varphi}\in {\bf W}^{1,p'}_{ \sigma, \vt}(\Omega )\\
&\hbox{where:}\\
&a_\lambda (\vu,\bm{\varphi}) = \lambda \int\displaylimits_{\Omega}{\vu\cdot \bar{\bm{\varphi}}} + 2\int\displaylimits_{\Omega}{\DT\vu:\DT\bar{\bm{\varphi}}} + \left\langle \alpha \vu_{\vt} , \bar{\bm{\varphi}}_{\vt}\right\rangle _\Gamma
\end{aligned}
\right.
\end{equation}
and $\left\langle , \right\rangle _{\Gamma}$ denotes the duality product between $\vWfracb{-\frac{1}{p}}{p}$ and $\vWfracb{\frac{1}{p}}{p'}$.
\end{proposition}

\begin{proof}
By \cite[Lemma 5.1]{AG}, for all $\vu\in {\bf W} ^{1,p} _{ \sigma , \vt }\Omega )$ and  $\bm{\varphi}\in {\bf W} ^{1,p'} _{ \sigma , \vt }(\Omega )$ we have  $\alpha \vu_{\vt}\cdot \bar{\bm{\varphi}}_{\vt}\in L^1(\Gamma )$ and
$$
\int\displaylimits_{\Gamma}{\alpha \vu_{\vt}\cdot \bar{\bm{\varphi}}_{\vt}}\le C \|\alpha \| _{L^{t(p)}(\Gamma ) }\|\vu\| _{{\bf W} ^{1,p} _{ \sigma , \vt }(\Omega )}\|{\bm{\varphi}}\| _{ {\bf W} ^{1,p'} _{ \sigma , \vt }(\Omega ) } .
$$
It easily then follows that
$a_\lambda (\cdot, \cdot)$ is a continuous sesqui-linear form on ${\bf W} ^{1,p} _{ \sigma , \vt }(\Omega )\times {\bf W} ^{1,p'} _{ \sigma , \vt }(\Omega )$. When $\lambda =0$, it is the sesqui-linear form $a (\cdot, \cdot)$ introduced in \cite{AG}.  The proof of this proposition then follows exactly the same steps and uses the same arguments as in  \cite[Proposition 5.2]{AG}.
\hfill
\end{proof}

The following theorem gives the existence of a unique solution of the resolvent problem and also the resolvent estimate.

\begin{theorem}
\label{17}
For any $\varepsilon \in (0,\pi)$, let $\lambda \in \Sigma_\varepsilon := \lbrace \lambda\in\mathbb{C} : |\arg \lambda| \leq \pi - \varepsilon\rbrace$, $\bm{f}\in \vL{2}$ and $\alpha\in\Lb{2}$. Then,\\
{\bf 1.} the problem (\ref{S3.ER1})-(\ref{S3.ER3}) has a unique solution $(\vu,\pi)\in \vH{1}\times L^2_0(\Omega)$.\\
{\bf 2.} there exists a constant $C_\varepsilon >0$, independent of $\bm{f}$, $\alpha $ and $\lambda$, such that the solution $\vu$ satisfies the following estimates, for $\lambda \neq 0$:
\begin{equation}
\label{26}
\|\vu\|_{\vL{2}} \leq \frac{1}{C_\varepsilon |\lambda|} \ \|\bm{f}\|_{\vL{2}}
\end{equation}

\begin{equation}
\label{27}
\|\DT\vu\|_{\vL{2}}\leq \frac{1}{C_\varepsilon \sqrt{|\lambda|}} \ \|\bm{f}\|_{\vL{2}}
\end{equation}
and
\begin{equation}
\label{270}
\|\pi\|_{\L{2}} \le C(\Omega) \left[1+\frac{1}{C_\varepsilon} \left(1+\frac{1}{\sqrt{|\lambda|}}\right)\right] \|\bm{f}\|_{\vL{2}} .
\end{equation}
{\bf 3.} moreover, if either (i) $\Omega$ is not axisymmetric or (ii) $\Omega$ is axisymmetric and $\alpha\ge \alpha_*>0$, then
\begin{equation}
\label{28}
\|\vu\|_{\vH{1}} + \|\pi\|_{\L{2}} \le C(\Omega) \left( 1+\frac{1}{C_\varepsilon}\right) \|\bm{f}\|_{\vL{2}} 
\end{equation}
and if $\alpha$ is a constant, then
\begin{equation}
\label{29.}
\|\vu\|_{\vH{2}}+ \|\pi\|_{\H1} \le C(\Omega) \left( 1+\frac{1}{C_\varepsilon}\right) \|\bm{f}\|_{\vL{2}} 
\end{equation}
where $C(\Omega)$ does not depend on $\alpha$.
\end{theorem}

\begin{remark}
\rm{Note that the estimates (\ref{26}) and (\ref{27}) give better decay for $\lambda$ large and enables us having a good semigroup theory, in general. On the other hand, estimate (\ref{28}) gives uniform bound on the solution, especially when $\lambda$ is small.}
\end{remark}

\begin{proof}
{\bf 1.} In view of Proposition \ref{S3P1}, it is enough to prove the existence and  uniqueness of a solution to (\ref{variational formulation}). Also $\lambda = 0$ case corresponds to the existence result for the stationary Stokes problem \cite[Theorem 4.1]{AG}. So we may consider $\lambda \ne 0$.

By Korn inequality, $\|\vu\|_{\vL{2}} + \|\DT\vu\|_{\vL{2}}$ is an equivalent norm on ${\bf H}^{1}(\Omega )$.
Then using the following inequality (cf. \cite[Lemma 4.1]{AAE1}):
$$\forall \ \lambda \in \Sigma_\varepsilon, \forall \ a,b >0, \qquad |\lambda a + b| \geq C_\varepsilon \left( |\lambda|a + b\right) \ \text{for some constant } C_\varepsilon >0$$
 and $\alpha \ge 0$, we get,
\begin{align*}
|a_\lambda(\vu,\vu)|
&\geq C_\varepsilon \left( |\lambda| \|\vu\|^2_{\vL{2}} + 2\|\DT\vu\|^2_{\vL{2}} + \int\displaylimits_{\Gamma}{\alpha |\vu_{\vt}|^2}\right) \\
&\geq C_\varepsilon \ \min{(|\lambda|, 2)} \left( \|\vu\|^2_{\vL{2}} + \|\DT\vu\|^2_{\vL{2}}\right) \\
& \geq C_\varepsilon \ \min{(|\lambda|, 2)} \|\vu\|^2_{\vH{1}} .
\end{align*}
Hence, for all $\lambda\in\Sigma_\varepsilon$, $a_\lambda$ is coercive on ${\bf H} ^{1} _{ \sigma , \vt }(\Omega )$ and therefore, by Lax-Milgram lemma, we get a unique solution in ${\bf H} ^{1} _{ \sigma , \vt}(\Omega )$ of the problem \eqref{variational formulation} which proves {\bf 1.}\\

{\bf 2.} From the variational formulation, we have, $$a_\lambda(\vu,\vu) = \int\displaylimits_{\Omega}{\bm{f}\cdot \bar{\vu}}$$
which gives
 $$|a_\lambda(\vu,\vu)| \leq \|\bm{f}\|_{\vL{2}} \|\vu\|_{\vL{2}} .$$
But we also have, $$|a_\lambda(\vu,\vu)| \geq C_\varepsilon \left( |\lambda|\|\vu\|^2_{\vL{2}}+2\|\DT\vu\|^2_{\vL{2}}\right) . $$
Thus
 $$\|\vu\|_{\vL{2}} \leq \frac{1}{C_\varepsilon |\lambda|}\|\bm{f}\|_{\vL{2}} $$
and then
$$\|\DT\vu\|^2_{\vL{2}} \ \leq \ \frac{1}{2 \ C_\varepsilon} \|\bm{f}\|_{\vL{2}}\|\vu\|_{\vL{2}} \ \leq \ \frac{1}{2\ |\lambda| \ C^2_\varepsilon} \|\bm{f}\|^2_{\vL{2}}$$
prove the inequalities (\ref{26}) and \eqref{27}.

From the equation (\ref{S3.ER1}), we may write, using (\ref{26}) and (\ref{27}),
\begin{equation*}
\begin{aligned}
\|\pi\|_{\L{2}} \le \|\nabla \pi\|_{\vH{-1}}&\le C(\Omega) \left(\|\bm{f}\|_{\vL{2}} + \|\Delta \vu\|_{\vH{-1}} + |\lambda|\|\vu\|_{\vL{2}} \right)\\
& \le C(\Omega) \left(\|\bm{f}\|_{\vL{2}} + \| \DT\vu\|_{\vL{2}} + |\lambda|\|\vu\|_{\vL{2}} \right)\\
&\le C(\Omega) \left[1+\frac{1}{C_\varepsilon} \left(1+\frac{1}{\sqrt{2|\lambda|}}\right)\right] \|\bm{f}\|_{\vL{2}}
\end{aligned}
\end{equation*}
which gives (\ref{270}).

{\bf 3.} Moreover, writing (\ref{S3.ER1}) as $-\Delta \vu+\nabla \pi ={\bm f}-\lambda \vu$, we deduce from the Stokes estimate \cite[Proposition 4.3]{AG} in the case either (i) $\Omega$ is not axisymmetric or (ii) $\Omega$ is axisymmetric and $\alpha\ge \alpha_*>0$, the existence of a constant $C>0$ which depends only on $\Omega$ such that
$$
\|\vu\| _{ {\bm H}^1(\Omega ) }+\|\pi \| _{ L^2(\Omega ) }\le C(\Omega) \left(\|{\bm f}\| _{ {\bm L}^2(\Omega ) }+|\lambda | \|\vu\| _{ {\bm L}^2(\Omega ) }  \right) \le C(\Omega) \left(1+\frac{1}{C_\varepsilon}\right) \|\bm{f}\|_{\vL{2}}.
$$
This provides the better bound (\ref{28}) on $\vu$ and $\pi$ when $\lambda$ is small.

Similarly, for constant $\alpha$, the $H^2$ estimate of the Stokes problem \cite[Theorem 4.5]{AG} yields
\begin{equation*}
\|\vu\|_{\vH{2}}+\|\pi\|_{\H1}\le C(\Omega) \left(1+\frac{1}{C_\varepsilon}\right) \|\bm{f}\|_{\vL{2}}.
\end{equation*}
\hfill
\end{proof}

In the next theorem we prove the analyticity of the semigroup generated by the Stokes operator with Navier boundary condition on $\bm{L}^2_{\sigma,\vt}(\Omega)$.

\begin{theorem}
For any $\alpha\in\Lb{2}$, the operator $-A_{2,\alpha}$, defined in (\ref{S3E1})-(\ref{S3E10}) with $p=2$,  generates a bounded analytic semigroup on $\bm{L}^2_{\sigma,\vt}(\Omega)$.
\end{theorem}

\begin{proof}
Obviously $\mathbf{D}(A_{2,\alpha})$ is dense in $\bm{L}^2_{\sigma,\vt}(\Omega)$. Therefore, according to Theorem \ref{analyticity}, it is enough to prove the resolvent estimate. Now, by definition and from the previous theorem, we have,
$$\|(\lambda I + A_{2,\alpha})^{-1}\| = \sup_{\underset{\bm{f}\neq \bm{0}}{\bm{f}\in\bm{L}^2_{\sigma,\vt}(\Omega)}}\frac{\|(\lambda I + A_{2,\alpha})^{-1}\bm{f}\|_{\vL{2}}}{\|\bm{f}\|_{\vL{2}}} = \sup_{\bm{f}\in\bm{L}^2_{\sigma,\vt}(\Omega)}\frac{\|\vu\|_{\vL{2}}}{\|\bm{f}\|_{\vL{2}}} \leq \frac{1}{C_\varepsilon |\lambda|} .$$
Hence, the result.
\hfill
\end{proof}

Next we extend the results of Theorem \ref{17} for all $p\in (1,\infty)$.

\begin{theorem}
\label{18}
Suppose that $p\in (1,\infty)$, $\alpha $ satisfies (\ref{alpha}) and $\lambda \in \Sigma_\varepsilon$ where $\Sigma_\varepsilon$ is defined as in Theorem \ref{17}. Then for every  $\bm{f}\in\vL{p}$, the problem (\ref{S3.ER1})-(\ref{S3.ER3}) has a unique solution $(\vu,\pi)\in\vW{2}{p}\times \left( \W{1}{p}\cap L^p_0(\Omega)\right) $. 
\end{theorem}

\begin{proof}
\textbf{case (i):} $p> 2$. Since from the assumption, $\bm{f}\in {\bm L}^2(\Omega )$ and $\alpha \in L^2(\Gamma )$, there exists a unique solution  $(\vu,\pi)\in \vH{1}\times L^2_0(\Omega)$ of (\ref{S3.ER1})-(\ref{S3.ER3}) by Theorem \ref{17}. Now writing the equation (\ref{S3.ER1}) as
$-\Delta \vu + \nabla \pi= \bm{f} -\lambda \vu$ and since $\vu\in\vH{1}\hookrightarrow \vL{6}$, we have $\bm{f}-\lambda\vu\in \vL{p}$ for all $p\leq 6$. Thus, using the regularity result in \cite[Theorem 5.10]{AG}, we obtain $\vu\in\vW{2}{p}$ for all $p\leq 6$.

Now for $p>6$, we have $\vu\in\vW{2}{6}\hookrightarrow \vL{\infty}$. Hence $\bm{f}-\lambda\vu \in\vL{p}$ and by the same regularity result, we get $\vu\in\vW{2}{p}$ for all $p>6$.\\

\noindent \textbf{case (ii):} $1<p< 2$. We first claim that 
 $(\lambda I + A)$ is an isomorphism from ${\bm W}^{1, q} _{ \sigma , \vt }(\Omega )$ to $({\bm W}^{1, q'} _{ \sigma , \vt }(\Omega ))^\prime$ for all $q\geq 2$. Then the adjoint operator,  $\lambda I +A^*$ is also an isomorphism from ${\bm W}^{1, q'} _{ \sigma , \vt }(\Omega )$ to $({\bm W}^{1, q} _{ \sigma , \vt}(\Omega ))^\prime$ with $q^\prime \leq 2$. Then, 
 for any ${\bm f}\in {\bf L}^p _{ \sigma , \vt}(\Omega )\subset ({\bm W}^{1, p'} _{ \sigma , \vt }(\Omega ))^\prime$, there exists a unique $\vu \in {\bm W}^{1, p} _{ \sigma , \vt}(\Omega )$ such that $(\lambda I+A^*)\vu={\bm f}$. Our second claim is that, since ${\bm f}\in {\bm L}^p _{ \sigma , \vt }(\Omega )$ it follows that $\vu \in {\bf D}(A_{p,\alpha})$ and $A^* \vu=A_{p,\alpha} \vu$. This finally implies that $(\lambda I+A_{p,\alpha})\vu=\bm f$.
 
First claim: For $q>2$ and ${\bm \ell} \in ({\bm W}^{1, q'} _{ \sigma , \vt }(\Omega ))^\prime\subset ({\bm H}^{1} _{ \sigma , \vt }(\Omega ))^\prime$, from Lax Milgam lemma there is a unique
 $\vu \in {\bm H}^{1} _{ \sigma , \vt }(\Omega )$ such that  $a_ \lambda  (\vu, {\bm \varphi} )=\langle {\bm \ell}, {\bm \varphi} \rangle$ for all ${\bm \varphi}\in {\bm H}^{1} _{ \sigma , \vt }(\Omega )$. Then,  $a (\vu, {\bm \varphi} )=\langle {\bm \ell}-\lambda \vu, {\bm \varphi} \rangle$ with 
${\bm \ell}-\lambda \vu\in ({\bm W}^{1, q'} _{ \sigma , \vt }(\Omega ))^\prime$. 
On the other hand, by \cite[Theorem 5.5]{AG} there exists a unique ${\bm w} \in {\bm W}^{1, q} _{ \sigma , \vt }(\Omega )\subset {\bm W}^{1, 2} _{ \sigma , \vt }(\Omega )$ such that
$$
a ({\bm w}, {\bm \varphi} )=\langle {\bm \ell}-\lambda \vu, {\bm \varphi} \rangle\,\,\,\forall {\bm \varphi}\in {\bm W}^{1, q'} _{ \sigma , \vt }(\Omega ).
$$
It then follows that
$$
a ({\bm w}-\vu, {\bm \varphi} )=0\,\,\,\forall {\bm \varphi}\in {\bm H}^{1} _{ \sigma , \vt}(\Omega )
$$
and by the uniqueness result \cite[Theorem 5.5]{AG}, $\vu={\bm w}$ in  ${\bm H}^{1} _{ \sigma , \vt }(\Omega )$ and thus $\vu\in {\bm W}^{1, q} _{ \sigma , \vt }(\Omega )$.
 
Second claim: If ${\bm f}\in {\bm L}^p _{ \sigma , \vt }(\Omega )$ with $1<p<2$ and $(\lambda I+A^*)\vu={\bm f}$ with $\vu \in {\bm W}^{1, p} _{ \sigma , \vt}(\Omega )$,
then $A^* \vu={\bm f}-\lambda \vu=:{\bm g}\in {\bm L}^p _{ \sigma , \vt }(\Omega )$ which means $a(\vu, {\bm \varphi} )=({\bm g}, {\bm \varphi} )$ for all ${\bm \varphi }\in {\bm W}^{1, p'} _{ \sigma , \vt}(\Omega )$. It then follows  by \cite[Theorem 5.10]{AG} that  $(\vu, \pi )\in {\bm W}^{2, p}(\Omega )\times W^{1, p}(\Omega )$ with $\pi $ defined by  (\ref{pression}) for ${\bm \psi }=\Delta \vu$ and $\vu$ satisfies the boundary condition. In particular $\vu \in {\bf D}({ A}_{p,\alpha})$. Then, using  Remark \ref{OperatorA}, $A_p \vu=A^* \vu$.
\hfill
\end{proof}

\begin{remark}
	\rm{
Notice that though the two boundary conditions
\begin{equation}
\label{20.}
\mathbf{curl} \ \vu \times \vn = \bm{0} \ \text{ on }\Gamma
\end{equation}
and
\begin{equation}
\label{21}
\ 2[(\DT\vu)\vn]_{\vt}+\alpha \vu_{\vt} = \bm{0} \ \text{ on }\Gamma
\end{equation}
are very much similar, as described in \eqref{19.}, but in the case of the Stokes problem with Navier type boundary condition \eqref{20.} on $\bm{L}^p_{\sigma,\vt}(\Omega)$ the pressure is constant and hence does not appear in the operator (see \cite[Proposition 3.1]{AAE}). On the contrary, the pressure term does appear in the Stokes operator with Navier boundary condition \eqref{21}.}
\end{remark}

Next we deduce ${\bm L}^p$-resolvent estimate for all $p\in (1,\infty)$.
\begin{theorem}
\label{T1}
Let $\lambda \in \mathbb{C}^*$ with $\Re \lambda \ge 0$ and $p\in (1,\infty)$. Then for $\alpha \in \Lb{t(p)}$ and for any $\bm{f}\in \vL{p}$, the unique solution $\vu\in \vW{1}{p}$ of (\ref{S3.ER1})-(\ref{S3.ER3}) satisfies the estimate:
	\begin{equation*}
	\|\vu\|_{\vL{p}}\le \frac{C}{|\lambda|} \|\bm{f}\|_{\vL{p}}
	\end{equation*}
	where $C$ depends at most on $ p$ and $\Omega$.
\end{theorem}

To prove the above theorem, we need to establish the following weak Reverse H\"{o}lder estimate (as in \cite[Lemma 6.2]{shen}):

\begin{lemma}
	\label{LRHI}
	Let $x_0\in \overline{\Omega}$ and $0< r<c \ \text{diam}(\Omega)$. Let $(\vu,\pi)\in \bm{H}^1(B(x_0, 2r)\cap \Omega)\times L^2(B(x_0, 2r)\cap \Omega)$ satisfies the Stokes system
	\begin{equation}
	\label{S_lambda}
	\begin{cases}
	\lambda \vu-\Delta\vu +\nabla \pi=\bm{0},\quad \div\;\vu=0\ &\text{ in $B(x_0, 2r)\cap \Omega$}\\
	\vu\cdot\vn=0, \quad 2\left[(\DT\vu)\vn \right]_{\vt}+\alpha\vu_{\vt}=\bm{0} \ &\text{ on $B(x_0,2r)\cap \Gamma$}
	\end{cases}
	\end{equation}
where $\lambda\in \mathbb{C}^*$ with $\Re \lambda \ge 0$ and $\alpha \in \Lb{t(p)}$. Then, for any $p\ge 2$,
	\begin{equation}
	\label{RHI}
	\left( \frac{1}{r^3} \int\displaylimits_{\Omega\cap B(x_0, r)}{|\vu|^p}\right) ^{1/p} \le C \left( \frac{1}{r^3} \int\displaylimits_{\Omega\cap B(x_0, 2r)}{|\vu|^2}\right) ^{1/2}
	\end{equation}
	where $C>0$ depends only on $p$ and $\Omega$.
\end{lemma}

Note that if $B(x_0, 2r)\cap \Gamma = \emptyset$, then we do not consider any boundary condition.

\begin{proof}
	By a geometric consideration, we only need to establish the estimate in two cases : (i) $x_0\in \Omega$ and $B(x_0, 3r)\subset \Omega$; (ii) $x_0 \in \Gamma$.
	
	The first case (i) follows from interior estimate (cf. \cite[estimate 5.22]{shen}).
	The second case (ii) concerns a boundary estimate. Here onwards, we denote the ball $B(x_0,R)$ by $B_R$ and $\dashint_\omega {f}= \frac{1}{|\omega|}\int\displaylimits_{ \omega  } {f}$.\\

1. Let $r\le s < t\le 2r$ and consider a cut-off unction $\eta\in C_c^{\infty}(B_t)$ such that
	\begin{equation}
	\label{eta}
	\eta \equiv 1 \ \text{ on } B_s, \quad 0\le \eta \le 1 \quad \text{ and } \quad |\nabla \eta| \le \frac{C}{t-s}.
	\end{equation}
Multiplying (\ref{S_lambda}) by $\eta ^2 \vu$ and integrating by parts, we get,
	\begin{equation*}
0= \lambda \int\displaylimits_{B_{2r} \cap\Omega}{\eta^2 |\vu|^2}+2\int\displaylimits_{B_{2r} \cap\Omega}{\DT\vu: \DT(\eta^2 \vu)} + \int\displaylimits_{B_{2r} \cap\Gamma}{\alpha \eta^2 |\vu_{\vt}|^2} - \int\displaylimits_{B_{2r} \cap\Omega}{(\pi-\pi_0) \ \div (\eta^2 \vu)}
	\end{equation*}
where $\pi_0 = \dashint\displaylimits_{\Omega}\pi$. This gives, using the fact that $\div \ \vu =0$ in $\Omega$,
\begin{equation*}
\lambda \int\displaylimits_{B_{2r} \cap\Omega}{\eta^2 |\vu|^2}+ 2\int\displaylimits_{B_{2r} \cap\Omega}{\eta^2 |\DT\vu|^2} +\int\displaylimits_{B_{2r} \cap\Gamma}{\alpha \eta^2 |\vu_{\vt}|^2}
= -4\int\displaylimits_{B_{2r} \cap\Omega}{\DT\vu:\eta \nabla \eta \bar{\vu}} + 2\int\displaylimits_{B_{2r} \cap\Omega}{(\pi-\pi_0) \eta \nabla \eta \bar{\vu}}
\end{equation*}
where $	\nabla \eta \bar{\vu}$	is the matrix $\nabla \eta\otimes \bar{\vu}$. Equating real and imaginary parts and using Cauchy's inequality, this implies
\begin{align}
& \quad \Re \lambda \int\displaylimits_{B_{2r} \cap\Omega}{\eta^2 |\vu|^2}+ 2\int\displaylimits_{B_{2r} \cap\Omega}{\eta^2 |\DT\vu|^2} +\int\displaylimits_{B_{2r} \cap\Gamma}{\alpha \eta^2 |\vu_{\vt}|^2}\nonumber\\
& = -4\Re \int\displaylimits_{B_{2r} \cap\Omega}{\DT\vu:\eta \nabla \eta \bar{\vu}} + 2\Re \int\displaylimits_{B_{2r} \cap\Omega}{(\pi-\pi_0) \eta \nabla \eta \bar{\vu}}\nonumber\\
&\le \varepsilon \int\displaylimits_{ \Omega \cap B_{2r}}{\eta^2 |\DT \vu|^2} + C_\varepsilon \int\displaylimits_{ \Omega \cap B_{2r}}{|\nabla \eta|^2 |\vu|^2} + \varepsilon \int\displaylimits_{ \Omega \cap B_{2r}}{\eta^2 |\pi- \pi_0|^2} + C_\varepsilon \int\displaylimits_{ \Omega \cap B_{2r}}{|\nabla \eta |^2 |\vu|^2} \label{0}
\end{align}
and
\begin{align}
& \quad \Im \lambda \int\displaylimits_{B_{2r} \cap\Omega}{\eta^2 |\vu|^2}\nonumber\\
& = -4\Im \int\displaylimits_{B_{2r} \cap\Omega}{\DT\vu:\eta \nabla \eta \bar{\vu}} + 2\Im \int\displaylimits_{B_{2r} \cap\Omega}{(\pi-\pi_0) \eta \nabla \eta \bar{\vu}}\nonumber\\
&\le \varepsilon \int\displaylimits_{ \Omega \cap B_{2r}}{\eta^2 |\DT \vu|^2} + C_\varepsilon \int\displaylimits_{ \Omega \cap B_{2r}}{|\nabla \eta|^2 |\vu|^2} + \varepsilon \int\displaylimits_{ \Omega \cap B_{2r}}{\eta^2 |\pi-\pi_0|^2} + C_\varepsilon \int\displaylimits_{ \Omega \cap B_{2r}}{|\nabla \eta |^2 |\vu|^2} .\label{1}
\end{align}
Now adding (\ref{0}) and (\ref{1}) gives, since $\alpha>0$,
\begin{equation*}
\begin{aligned}
& \quad (\Re \lambda + |\Im \lambda |) \int\displaylimits_{B_{2r} \cap\Omega}{\eta^2 |\vu|^2}+ 2\int\displaylimits_{B_{2r} \cap\Omega}{\eta^2 |\DT\vu|^2}\nonumber\\
& \le \varepsilon \int\displaylimits_{ \Omega \cap B_{2r}}{\eta^2 |\DT \vu|^2} + C_\varepsilon \int\displaylimits_{ \Omega \cap B_{2r}}{|\nabla \eta|^2 |\vu|^2} + \varepsilon \int\displaylimits_{ \Omega \cap B_{2r}}{\eta^2 |\pi- \pi_0|^2}.
\end{aligned}
\end{equation*}
Incorporating the properties of $\eta$, we obtain
\begin{equation}
\label{3}
|\lambda| \int\displaylimits_{B_{s} \cap\Omega}{ |\vu|^2}+ \int\displaylimits_{B_{s} \cap\Omega}{ |\DT\vu|^2} \le \frac{C(\Omega)}{(t-s)^2} \int\displaylimits_{ \Omega \cap B_{t}}{ |\vu|^2} + \varepsilon \int\displaylimits_{ \Omega \cap B_{t}}{ |\pi - \pi_0|^2}.
\end{equation}

2. Next to estimate the pressure term, as $\pi-\pi_0\in L^2(B_t\cap\Omega)$, there exists a unique $\bm{\psi}\in \bm{H}^1(B_t\cap\Omega)$ such that
\begin{equation*}
\begin{cases}
\div \ \bm{\psi}=\pi-\pi_0 &\text{ in } B_t\cap\Omega\\
\bm{\psi} = \bm{0} &\text{ on } \partial(B_t\cap\Omega)
\end{cases}
\end{equation*}
satisfying
\begin{equation}
\label{psi}
\|\bm{\psi}\|_{\bm{H}^1(B_t\cap\Omega)} \le C(\Omega) \ \|\pi-\pi_0\|_{L^2(B_t\cap\Omega)}. 
\end{equation}
Multiplying (\ref{S_lambda}) by $\bm{\psi}$ yields (replacing $\pi$ by $(\pi-\pi_0)$ and extending $\bm{\psi}$ by $0$ outside $ \Omega \cap B_t$),
\begin{equation*}
\lambda \int\displaylimits_{ \Omega \cap B_t}{\vu \cdot \bm{\psi}} + 2 \int\displaylimits_{ \Omega \cap B_t}{\DT \vu : \DT\bm{\psi}}= \int\displaylimits_{ \Omega \cap B_t}{\left( \pi- \pi_0\right) \ \div \ \bm{\psi}} = \int\displaylimits_{ \Omega \cap B_t}{|\pi-\pi_0|^2}
\end{equation*}
which gives, using (\ref{psi}),
\begin{equation*}
\label{2}
\|\pi-\pi_0\|_{L^2(\Omega \cap B_t)} \le C(\Omega) \left( |\lambda| \|\vu\|_{L^2(\Omega \cap B_t)} + \|\DT\vu \|_{L^2(\Omega\cap B_t)}\right) .
\end{equation*}
Plugging this estimate in (\ref{3}), we obtain
\begin{equation*}
| \lambda | \int\displaylimits_{B_s \cap\Omega}{|\vu|^2}+ \int\displaylimits_{B_s \cap\Omega}{|\DT\vu|^2}\nonumber\\
\le \frac{C(\Omega)}{(t-s)^2}\int\displaylimits_{ \Omega \cap B_t}{|\vu|^2} + \varepsilon \int\displaylimits_{ \Omega \cap B_t}{ |\DT \vu|^2} + \varepsilon |\lambda | \int\displaylimits_{ \Omega \cap B_t}{|\vu|^2}.
\end{equation*}
From here, we deduce the following Caccioppoli inequality for Stokes system
\begin{equation}
\label{11}
\int\displaylimits_{B_s\cap\Omega}{|\DT\vu|^2} \le \frac{C(\Omega)}{(t-s)^2}\int\displaylimits_{B_t\cap\Omega }{|\vu|^2}
\end{equation}
which follows from the general result \cite[Lemma 0.5]{modica} setting $f = |\lambda| |\vu|^2 + |\DT\vu |^2, g = |\vu|^2, h= 0$ and $\alpha =2$.

3. Therefore \cite[Lemma 6.7]{mitrea} enables us to have the following reverse H\"{o}lder inequality,
\begin{equation}
\label{10}
\left( \ \ppvint\displaylimits_{ \Omega \cap B_r}{|\vu|^2}\right) ^{1/2} \le C(q) \left( \,\,\,\,\,\,\,\, \ppvint\displaylimits_{ \Omega \cap B_{2r}}{|\vu|^q}\right) ^{1/q} \quad \text{ for any } q>0.
\end{equation}
Finally we claim that the above inequality implies (\ref{RHI}). To prove that, let us define an operator
\begin{equation*}
\begin{aligned}
T : {\bm L}^{p'}(\Omega\cap B_{2r}) &\rightarrow {\bm L}^{2}(\Omega\cap B_r)\\
{\bm v} & \mapsto {\bm v} .
\end{aligned}
\end{equation*}
Now for the adjoint map
$$
T^* : {\bm L}^{2}(\Omega\cap B_r) \rightarrow {\bm L}^{p}(\Omega\cap B_{2r})
$$
by definition, we can write,
\begin{equation*}
\begin{aligned}
\left\langle T^*\vu , \bm{f}\right\rangle_{{\bm L}^{p}(\Omega\cap B_{2r}) \times {\bm L}^{p'}(\Omega\cap B_{2r})}& = \left\langle \vu, T \bm{f} \right\rangle_{{\bm L}^{2}(\Omega\cap B_r) \times {\bm L}^{2}(\Omega\cap B_r)}\\
& \le C(p) r^{3(\frac{1}{2}- \frac{1}{p'})} \|{\bm f}\|_{{\bm L}^{p'}(\Omega\cap B_{2r })} \|\vu\|_{{\bm L}^{2}(\Omega\cap B_r)}
\end{aligned}
\end{equation*}
where the last inequality comes from (\ref{10}). This shows
\begin{equation*}
\|T^* \vu\|_{{\bm L}^{p}(\Omega\cap B_{2r}) } \le C(p) r^{3(\frac{1}{2}- \frac{1}{p'})} \|\vu\|_{{\bm L}^{2}(\Omega\cap B_r)}.
\end{equation*}
Since $T^*\vu = \vu$ on $ {\bm L}^p(\Omega\cap B_r)$, we then obtain
\begin{equation*}
\left( \ \int\displaylimits_{ \Omega \cap B_r}{|\vu|^p}\right) ^{1/p} \le C(p) r^{3(\frac{1}{2}- \frac{1}{p'})} \left( \ \int\displaylimits_{ \Omega \cap B_r}{|\vu|^2}\right) ^{1/2}
\end{equation*}
which concludes (\ref{RHI}) upon using $\|\vu\|_{{\bm L }^2(\Omega\cap B_r)}\le \|\vu\|_{{\bm L }^2(\Omega\cap B_{2r})}$.
\hfill
\end{proof}

The following lemma, due to Shen \cite[Lemma 6.3]{shen}, contains the real variable argument needed to complete the proof of Theorem \ref{T1}.

\begin{lemma}
	\label{L0}
	Let $p>2$ and $\Omega$ be a bounded Lipschitz domain in $\R^{d}$. Suppose that:\\
	(1) $T$ is a bounded sublinear operator in $L^2(\Omega;\mathbb{C}^m)$ and $\|T\|_{L^2\rightarrow L^2} \le C_0$;\\
	(2) There exist constants $0< \beta <1$ and $N>1$ such that for any bounded measurable $f$ with $\text{supp} f \subset \Omega \setminus 3B$,
	\begin{equation}
	\label{extrapolation_condition}
	\left( \ \pvint\displaylimits_{\Omega\cap B}{|Tf|^p}\right) ^{1/p} \le N\left\lbrace \left( \ \ppvint\displaylimits_{\Omega\cap 2B}{|Tf|^2}\right)^{1/2} + \sup_{B'\supset B} \left( \ \dashint\displaylimits_{B'}{|f|^2}\right) ^{1/2} \right\rbrace 
	\end{equation}
	where $B=B(x,r)$ is a ball with $x\in \overline{\Omega}$ and $0<r<\beta \ \text{diam} (\Omega)$. Then $T$ is bounded on $L^q(\Omega;\mathbb{C}^m)$ for any $2< q <p$. Moreover, $\|T\|_{L^q \rightarrow L^q}$ is bounded by a constant depending on at most $d, m, \beta, N, C_0, p, q$ and the Lipschitz character of $\Omega$.
\end{lemma}

\begin{proof}[{ \bf Proof of Theorem \ref{T1}}]
	By rescaling, we may assume that $\text{diam}(\Omega)=1$. Then for $\lambda \in \mathbb{C}^*$ with $\Re \lambda \ge 0$ and $\bm{f} \in \vL{2}$, there exists a unique $(\vu,\pi) \in \vH{1}\times L^2_0(\Omega)$ satisfying (\ref{S3.ER1})-(\ref{S3.ER3}). Also
	\begin{equation*}
	|\lambda| \int\displaylimits_{\Omega}{|\vu|^2} + 2 \int\displaylimits_{\Omega} {|\DT \vu|^2} + \int\displaylimits_{\Gamma}{\alpha |\vu_{\vt}|^2} \le C \int\displaylimits_{\Omega} {|\bm{f}| |\vu|}
	\end{equation*}
	where $C$ depends only on $\Omega$. By H\"{o}lder inequality, this implies
	\begin{equation*}
	|\lambda| \|\vu\|_{\vL{2}}\le C_0 \|\bm{f}\|_{\vL{2}}
	\end{equation*}
	where $C_0$ depends only on $\Omega$. Let us now define the operator $T_\lambda$ by $T_\lambda (\bm{f}) =|\lambda| \vu$. Then $T_\lambda$ is a bounded linear operator on $\vL{2}$ and $\|T_\lambda \|_{\bm{L}^2\rightarrow \bm{L}^2} \le C_0$ which is the assumption (1) in Lemma \ref{L0}. We will show $\|T_\lambda \|_{\bm{L}^q\rightarrow \bm{L}^q} \le C$ for $2<q<p $ using Lemma \ref{L0}.
	
	To verify the assumption (2) in Lemma \ref{L0}, let $B=B(x_0,r)$ where $x_0\in \overline{\Omega}$ and $ 0<r<c$. Let $\bm{f}\in \vL{2}$ with $\text{supp} \bm{f}\subset \Omega \setminus 3B$. Since
	\begin{equation*}
	\begin{cases}
	\lambda \vu - \Delta \vu + \nabla \pi = \bm{0}, \,\,\, \div \ \vu =0 \ & \text{ in } \ \Omega \cap 3B\\
		\vu\cdot\vn=0, \quad 2\left[(\DT\vu)\vn \right]_{\vt}+\alpha\vu_{\vt}=\bm{0} \ &\text{ on $\Gamma\cap 3B$}
	\end{cases}
	\end{equation*}
	we may then apply Lemma \ref{LRHI} to obtain,
	\begin{equation*}
	\left( \frac{1}{r^3}\int\displaylimits_{ \Omega \cap B}{|\vu|^p}\right) ^{1/p} \le C \left( \frac{1}{r^3}\int\displaylimits_{ \Omega \cap 2B}{|\vu|^2}\right) ^{1/2}.
	\end{equation*}
	It follows that
	\begin{equation*}
	\left( \ \pvint\displaylimits_{\Omega \cap B}{|T_\lambda (\bm{f})|^p}\right) ^{1/p} \le C \left( \ \ppvint\displaylimits_{\Omega \cap 2B}{|T_\lambda (\bm{f})|^2}\right) ^{1/2}
	\end{equation*}
	where $C$ depends only on $ p$ and $\Omega$. Therefore by Lemma \ref{L0}, we conclude that the operator $T_\lambda$ is bounded on $\bm{L}^q(\Omega)$ for any $2<q<p$ and that $\|T_\lambda \|_{\bm{L}^q\rightarrow \bm{L}^q}\le C_q$ where $C_q$ depends on at most $q$ and $\Omega$. Thus in view of the definition of $T_\lambda$, we have shown that for any $q>2$,
	\begin{equation*}
	\|\vu\|_{\vL{q}} \le \frac{C_q}{|\lambda|} \|\bm{f}\|_{\vL{q}}.
	\end{equation*}
	By duality, the estimate also holds for $q\in (1,2)$.
	\hfill
\end{proof}

As a conclusion, we have the following theorem:

\begin{theorem}
\label{51}
Let $p\in (1,\infty)$ and  $\alpha$ as in \eqref{alpha}. Then for all $\lambda \in \mathbb{C}^*$ with $\Re \lambda \geq 0$ and $\bm{f}\in\vL{p}$, there exists a constant $C>0$ depending on at most $p$ and $\Omega$ such that the unique solution $(\vu, \pi )\in {\bf D}(A_{p,\alpha})\times \left( W^{1, p}(\Omega )\cap L^p_0(\Omega)\right) $ of \eqref{S3.ER1}-\eqref{S3.ER3} satisfies:
\begin{equation}
\label{13}
\|\vu\|_{\vL{p}} \leq \frac{C}{|\lambda|}\|\bm{f}\|_{\vL{p}}.
\end{equation}
If moreover $\alpha$ is a constant and either (i) $\Omega$ is not axisymmetric or (ii) $\Omega$ is axisymmetric and $\alpha\ge \alpha_*>0$, then
\begin{equation}
\label{14}
\|\DT\vu\|_{\vL{p}} \leq \frac{C}{\sqrt{|\lambda|}}\|\bm{f}\|_{\vL{p}}
\end{equation}
and
\begin{equation}
\label{15}
\|\vu\|_{\vW{2}{p}} \leq C \ \|\bm{f}\|_{\vL{p}}.
\end{equation}
\end{theorem}

\begin{proof}
Estimate \eqref{13} follows from Theorem \ref{T1}.

Let us now prove the estimate (\ref{14}). From Gagliardo-Nirenberg inequality \cite[Chapter IV, Theorem 4.14, Theorem 4.17]{Adams} and regularity estimate for the stationary Stokes problem \cite[Theorem 6.11]{AG}, we have,
\begin{align*}
\|\DT\vu\|_{\vL{p}} \leq C \ \|\vu\|_{\vW{1}{p}} &\leq C \ \|\vu\|^{1/2}_{\vW{2}{p}}\|\vu\|^{1/2}_{\vL{p}}\\
& \leq C \, \|\bm{f}-\lambda \vu\|^{1/2}_{\vL{p}} \|\vu\|^{1/2}_{\vL{p}}\\
& \leq C \left( \|\bm{f}\|_{\vL{p}} + |\lambda| \|\vu\|_{\vL{p}}\right) ^{1/2} \|\vu\|^{1/2}_{\vL{p}}.
\end{align*}
Thus estimate (\ref{14}) follows using \eqref{13}.

To prove (\ref{15}) we again use \cite[Theorem 6.11]{AG} and (\ref{13}) to obtain
$$\|\vu\|_{\vW{2}{p}} \leq C \ \|\bm{f} - \lambda \vu \|_{\vL{p}} \leq C \ \|\bm{f}\|_{\vL{p}} .$$
\hfill
\end{proof}

Finally we obtain our first main result:

\begin{theorem}
\label{24}
Let $\alpha$ be as in \eqref{alpha}. The operator $-A_{p,\alpha}$ generates a bounded analytic semigroup on $\bm{L}^p_{\sigma,\vt}(\Omega)$ for all $1<p<\infty$.
\end{theorem}

\begin{proof}
In view of Theorem \ref{51}, to apply Theorem \ref{analyticity} it remains to check that ${\bf D}(A_{p,\alpha})$ is dense in $\bm{L}^p_{\sigma,\vt}(\Omega)$. But this is immediate since
$\pmb{\mathscr{D}}_\sigma({\Omega})\hookrightarrow {\bf D}(A_{p,\alpha}) \hookrightarrow \bm{L}^p_{\sigma,\vt}(\Omega)$ and by definition $\pmb{\mathscr{D}}_\sigma({\Omega})$ is dense in 
$\bm{L}^p_{\sigma,\vt}(\Omega)$.
\hfill
\end{proof}

\section{Stokes operator on $[H_0^{p^\prime}(\div ,\Omega)]^\prime_{\sigma, \tau}$}
\label{67}
\setcounter{equation}{0}

We first recall that if $\bm{f}\in [\bm{H}_0^{p^\prime}(\div ,\Omega)]^\prime$ (defined in (\ref{S2E1})) and is such that $\div \bm{f}\in \vL{p}$ for some $p\in (1, \infty)$, then its normal trace 
$(\bm{f}\cdot \vn) _{ |\Gamma  }$ is well defined and belongs to  $\Wfracb{-1-\frac{1}{p}}{p}$ \cite[Corollary 3.7]{AAE1}.

 Let $\mathscr{B}$ be the closure of $\pmb{\mathscr{D}} _{ \sigma  }({\Omega})$ in $[\bm{H}_0^{p^\prime}(\div ,\Omega)]^\prime$. Then, it can be shown that
$$\mathscr{B} = [\bm{H}_0^{p^\prime}(\div ,\Omega)]^\prime _{\sigma, \vt} := \left\lbrace \bm{f}\in [\bm{H}_0^{p^\prime}(\div ,\Omega)]^\prime : \div \ \bm{f}=0 \text{ in } \Omega, \bm{f}\cdot \vn = 0 \text{ on } \Gamma\right\rbrace $$
which is a Banach space with the norm  of $[\bm{H}_0^{p^\prime}(\div ,\Omega)]^\prime$ \cite[Proposition 3.9]{AAE1}. Let $Q : [\bm{H}_0^{p^\prime}(\div ,\Omega)]^\prime\rightarrow \mathscr{B}$ be the orthogonal projection on $\mathscr{B}$. We define the Stokes operator $B_{p,\alpha}$ on $\mathscr{B}$, as
\begin{equation*}
\begin{cases}
&\mathbf{D}(B_{p,\alpha}) = \left\lbrace \vu\in \vW{1}{p}\cap \mathscr{B} : \Delta\vu \in [\bm{H}_0^{p^\prime}(\div ,\Omega)]^\prime, 2[(\DT\vu)\vn]_{\vt}+\alpha\vu_{\vt} = \bm{0} \text{ on } \Gamma \right\} ;\\
&B_{p,\alpha}(\vu) = -Q(\Delta \vu) \quad \text{ for } \vu\in\mathbf{D}(B_{p,\alpha}) 
\end{cases}
\end{equation*}
with $\alpha \in\Lb{t(p)}$ where $t(p)$ is defined in \eqref{22}.

\subsection{Analyticity}

As in the previous section, we will now discuss the analyticity of the semigroup generated by the Stokes operator $B_{p,\alpha}$ on $[\bm{H}_0^{p^\prime}(\div ,\Omega)]^\prime _{\sigma, \vt}$.

\begin{theorem}
\label{23}
Let $p\in (1,\infty)$ and  $\alpha \in \Lb{t(p)}$. Then, for all  $\lambda \in \mathbb{C}^*$ such that $\Re \lambda \geq 0$, and all 
$\bm{f}\in [\bm{H}_0^{p^\prime}(\div ,\Omega)]^\prime$, the problem (\ref{S3.ER1})-(\ref{S3.ER3}) has a unique solution $(\vu, \pi) \in \vW{1}{p} \times L^{p}_0(\Omega)$ satisfying
\begin{equation}\label{56}
\|\vu\|_{[\bm{H}_0^{p^\prime}(\div ,\Omega)]^\prime}  \leq \frac{C}{|\lambda|} \ \|\bm{f}\|_{[\bm{H}_0^{p^\prime}(\div ,\Omega)]^\prime}
\end{equation}
for some constant $C$ independent of $\lambda$ and $\alpha$.
\end{theorem}

\begin{proof}
{\bf 1.	Existence:} The proof of the existence and uniqueness of the solution follows similar arguments as in the proof of Theorem \ref{17} and Theorem \ref{18}.
For $p=2$ the existence and uniqueness of solution comes from Lax-Milgram lemma and De Rham theorem for the pressure.  

When $p>2$, since $\bm{f}\in [\bm{H}_0^{p^\prime}(\div ,\Omega)]^\prime \subset [\bm{H}_0^{2}(\div ,\Omega)]^\prime$ and $\alpha\in \Lb{t(p)} \subset \Lb{2}$, we have the existence of the unique solution $(\vu, \pi) \in \vH{1} \times L^{2}_0(\Omega)$. We can now apply \cite[Corollary 5.8]{AG}, since $\bm{f}-\lambda \bm{u} \in [\bm{H}_0^{p^\prime}(\div ,\Omega)]^\prime$ to obtain $\vu\in\vW{1}{p}$. 

For $p<2$, the proof follows in the same way as in \cite[Corollary 5.8]{AG}. \\

{\bf 2. Estimate:} Now to prove the estimate \eqref{56}, consider the problem,
\begin{equation*}
\left\{
\begin{aligned}
\lambda \bm{v} -\Delta \bm{v} + \nabla \theta = \bm{F} , \quad \div \ \bm{v} = 0 \quad & \text{ in } \Omega \\
\bm{v}\cdot \vn = 0, \quad 2[(\DT\bm{v})\vn]_{\vt}+\tilde{\alpha} \bm{v}_{\vt} = \bm{0} \quad & \text{ on } \Gamma
\end{aligned}
\right.
\end{equation*}
where $\bm{F}\in \bm{H}_0^{p^\prime}(\div ,\Omega)$ and $\tilde{\alpha}$ as in \eqref{alpha}. Thanks to Theorem \ref{18} and the estimate \eqref{13}, there exists unique $(\bm{v}, \theta)\in \vW{2}{p^\prime}\times (\W{1}{p^\prime}\cap L^{p'}_0(\Omega))$ with the estimate
$$\|\bm{v}\|_{\vL{p^\prime}} \leq \frac{C}{|\lambda|} \ \|\bm{F}\|_{\vL{p^\prime}} \ .$$
As a result, we get,
$$\|\bm{v}\|_{\bm{H}_0^{p^\prime}(\div ,\Omega)} \leq \frac{C}{|\lambda|} \ \|\bm{F}\|_{\bm{H}_0^{p^\prime}(\div ,\Omega)} .$$
Now, for the solution $(\vu,\pi)\in \vW{1}{p}\times \L{p}$ of the problem (\ref{S3.ER1})-(\ref{S3.ER3}), we get,
\begin{align*}
\|\vu\|_{[\bm{H}_0^{p^\prime}(\div ,\Omega)]^\prime} &= \underset{\underset{\bm{F}\neq \bm{0}}{\bm{F}\in \bm{H}_0^{p^\prime}(\div ,\Omega)} }{\sup}\frac{|\left\langle \vu,\bm{F}\right\rangle| }{\|\bm{F}\|_{\bm{H}_0^{p^\prime}(\div ,\Omega)}}
= \underset{\underset{\bm{F}\neq \bm{0}}{\bm{F}\in \bm{H}_0^{p^\prime}(\div ,\Omega)} }{\sup}\frac{|\left\langle \vu,\lambda \bm{v} -\Delta \bm{v} + \nabla \theta \right\rangle| }{\|\bm{F}\|_{\bm{H}_0^{p^\prime}(\div ,\Omega)}}\\
&= \underset{\underset{\bm{F}\neq \bm{0}}{\bm{F}\in \bm{H}_0^{p^\prime}(\div ,\Omega)} }{\sup}\frac{|\left\langle \lambda \vu - \Delta \vu + \nabla \pi,\bm{v} \right\rangle| }{\|\bm{F}\|_{\bm{H}_0^{p^\prime}(\div ,\Omega)}}= \underset{\underset{\bm{F}\neq \bm{0}}{\bm{F}\in \bm{H}_0^{p^\prime}(\div ,\Omega)} }{\sup}\frac{|\left\langle \bm{f},\bm{v} \right\rangle| }{\|\bm{F}\|_{\bm{H}_0^{p^\prime}(\div ,\Omega)}}\\
& \leq \ \frac{C}{|\lambda|} \ \|\bm{f}\|_{[\bm{H}_0^{p^\prime}(\div ,\Omega)]^\prime} 
\end{align*}
which is the required estimate.
\hfill
\end{proof}

\begin{theorem}
\label{S455}
Let $ p\in (1,\infty)$ and  $\alpha \in \Lb{t(p)}$ with $t(p)$ defined in \eqref{22}. Then, for $\bm{f}\in [\bm{H}_0^{p^\prime}(\div ,\Omega)]^\prime$, the  unique solution $(\vu, \pi )$ of (\ref{S3.ER1})-(\ref{S3.ER3}) in ${\bf D}(B_{p,\alpha})\times L^p_0(\Omega )$ is such that the pressure $\pi$ satisfies the following estimate in the case when (i) either $\Omega$ is not axisymmetric or (ii) $\Omega$ is axisymmetric and $\alpha\ge \alpha _*>0$:
\begin{equation*}
\|\pi\|_{\L{p}} \leq C \ \|\bm{f}\|_{[\bm{H}_0^{p^\prime}(\div ,\Omega)]^\prime} 
\end{equation*}
for some constant $C$ independent of $\lambda$ and $\alpha$.
\end{theorem}

\begin{proof}
By the regularity result of the stationary Stokes problem \cite[Theorem 6.7]{AG}, we can write
\begin{eqnarray*}
\|\vu\| _{ {\bf W}^{1, p}(\Omega ) }+\|\pi \| _{ L^p(\Omega ) }\le C\|{\bm f}-\lambda \vu\|_{[\bm{H}_0^{p^\prime}(\div ,\Omega)]^\prime} 
\end{eqnarray*}
and the result follows using estimate (\ref{56}).
\hfill
\end{proof}

The analyticity of the semigroup generated by the Stokes operator $B_{p,\alpha}$ on $[\bm{H}_0^{p^\prime}(\div ,\Omega)]^\prime _{\sigma, \vt}$ is now easily deduced from Theorem \ref{23}.
\begin{theorem}
	\label{S4Thm1}
Let $\alpha \in \Lb{t(p)}$ with $t(p)$ defined in \eqref{22}. The operator $-B_{p,\alpha}$ generates a bounded analytic semigroup on $[\bm{H}_0^{p^\prime}(\div ,\Omega)]^\prime _{\sigma, \vt}$ for all $1<p<\infty$.
\end{theorem}

\begin{proof}
In view of Theorem \ref{23}, to apply Theorem \ref{analyticity} it remains to check that ${\bf D}(B_{p,\alpha})$ is dense in $[\bm{H}_0^{p^\prime}(\div ,\Omega)]^\prime _{\sigma, \vt}$. But this is immediate since
$\pmb{\mathscr{D}}_\sigma({\Omega})\hookrightarrow {\bf D}(B_{p,\alpha}) \hookrightarrow [\bm{H}_0^{p^\prime}(\div ,\Omega)]^\prime _{\sigma, \vt}$ and by definition $\pmb{\mathscr{D}}_\sigma({\Omega})$ is dense in 
$[\bm{H}_0^{p^\prime}(\div ,\Omega)]^\prime _{\sigma, \vt}$.
\hfill
\end{proof}

\section{Imaginary and Fractional powers}
\label{frac power}
\setcounter{equation}{0}

\subsection{Imaginary powers}
Our main purpose in this section is to prove local bounds on pure imaginary powers $A_{p,\alpha}^{is}$ and $B_{p,\alpha}^{is}$ of the Stokes operators defined in Section \ref{66} and Section \ref{67} respectively. A complete theory of fractional powers of an operator (bounded or unbounded) can be found in Komatsu \cite{Komatsu}. 

Since these operators are non-negative operators, it then follows from the results in \cite{Komatsu} and in \cite{Triebel} that their powers are well, densely defined and closed linear operators on $\boldsymbol{L}^{p}_{\sigma,\vt}(\Omega)$ and $[\bm{H}_0^{p^\prime}(\div ,\Omega)]^\prime$ with domain $\mathbf{D}(A^{is}_{p,\alpha})$ and $\mathbf{D}(B^{is}_{p,\alpha})$ respectively.

Notice that in \cite{AAE}, it was comparatively straight forward to obtain the bounds on pure imaginary powers, since with Navier type boundary condition, the Stokes operator actually reduces to Laplace operator and thus they could borrow the well-established theory for elliptic operators, which is not our case. Therefore we use the theory of interpolation-extrapolation to make use of the established theory for similar operators and implement a perturbation argument.

\begin{theorem}
\label{62}
Let $\alpha$ be as in (\ref{alpha}) and if $p\in (1, 3]$ suppose also that  $\alpha\in\Lb{\infty}$. Then there exists an angle $0 <\theta < \pi/2$ and a constant $C>0$ such that for any $s\in \mathbb{R}$,
\begin{equation}
\label{53}
\|A_{p,\alpha}^{is}\| \leq C \ e^{|s| \theta} .
\end{equation}
Similarly, for $\alpha \in \Lb{\infty}$, there exists an angle $0 <\theta' < \pi/2$ and a constant $C'>0$, such that for any $s\in \mathbb{R}$,
\begin{equation}
\label{53.}
\|B_{p,\alpha}^{is}\| \leq C' \ e^{|s| \theta'} .
\end{equation}
\end{theorem}

\begin{proof}
Since the proof of (\ref{53.}) is exactly similar to that of (\ref{53}), we only show (\ref{53}). The proof of (\ref{53}) is based on the theory of interpolation-extrapolation scales from \cite{amann}. A similar approach has been followed in \cite{pruss}, considering the perturbation of a different operator than ours and for $\alpha$ constant.

{\bf 1.} Let us define $X_0 := \bm{L}^p_{\sigma,\vt}(\Omega)$ and $A_0 := \lambda I +A_{NT}$, for all  $\lambda >0$ and where $A_{NT}$ is the Stokes operator with Navier-type boundary condition
$$\vu\cdot \vn = 0, \quad \mathbf{curl} \ \vu \times \vn =\bm{0} \quad \text{ on } \Gamma .$$
i.e.
\begin{equation*}
\begin{cases}
&\mathbf{D}(A_{NT}) = \left\{ \vu\in\vW{2}{p}\cap \bm{L}^p_{\sigma,\vt}(\Omega),\mathbf{curl} \ \vu \times \vn = \bm{0} \text{ on } \Gamma \right\rbrace \\
&A_{NT}(\vu) = -P(\Delta \vu) \quad \text{ for } \vu\in\mathbf{D}(A_{NT}).
\end{cases}
\end{equation*}
Note that from \cite[Theorem 6.4]{Komatsu}, it follows that for $\lambda\ge 0$, the domain $\mathbf{D}(\lambda I +A_p)$ does not depend on $\lambda$ and
$$
\mathbf{D}(\lambda I +A_p) = \mathbf{D}(A_p) \quad \text{ for } \lambda>0 .
$$
As indicated in the Introduction to this Section, the powers $A_0^a$ of the operator $A_0$ are  well, densely defined and  closed linear operators on $\boldsymbol{L}^{p}_{\sigma,\vt}(\Omega)$ with domain $\mathbf{D}(A^{a}_{0})$.  

Now by \cite[Theorems V.1.5.1 and V.1.5.4]{amann}, $(X_0, A_0)$ generates an interpolation-extrapolation scale $(X_a,A_a), a\in \R$ with respect to the complex interpolation functor since $A_0$ is a closed operator on $X_0$ with bounded inverse (cf. \cite[Theorem 4.8]{AAE}). More precisely, for every $a\in \R$, $X_a$ is a Banach space,  
$X_a\hookrightarrow X _{ a-1 }$ and $A_a$ is an unbounded linear operator on $X_a$ with domain $X _{ a+1 }$ and for $a>0$:
\begin{eqnarray*}
	&&(i)\,\,\,X_a = \left( \mathbf{D}(A_0^a), \|A_0^a \cdot \|\right)\\
	&&(ii)\,\,\,A_a\,\,\,\hbox{is the restriction of}\,\,A_0\,\,\hbox{on}\,\,\,X_a.
\end{eqnarray*}
Moreover, for any $b\in (a, a+1)$,
$$
X_b=[X_a, X _{ a+1 }] _{ \theta }\,\,\,\hbox{ where }\,\,\frac {1} {b}=\frac {1-\theta} {a}+\frac {\theta} {a+1}.
$$

Similarly, let $X_0^{\sharp} := (X_0)' = \bm{L}^{p'}_{\sigma,\vt}(\Omega)$, $A_0^{\sharp} := (A_0)'$. Then $(X_0^{\sharp}, A_0^{\sharp})$ generates another interpolation-extrapolation scale $(X_a^{\sharp}, A_a^{\sharp})$, the dual scale by \cite[Theorem V.1.5.12]{amann} and
$$
(X_a)' = X_{-a}^{\sharp} \quad \text{ and } \quad (A_a)' = A_{-a}^{\sharp} \quad \text{ for } a\in \R
$$
where $A'$ denotes the dual of $A$. 
In the particular case $a=-1/2$, we obtain by definition, an operator $A_{-1/2}: X_{-1/2} \rightarrow X_{-1/2}$ with 
\begin{equation}
\label{SAE2}
\mathbf{D}(A_{-1/2}) = X_{1/2} = [X_0, X_1]_{1/2} .
\end{equation}
We now claim that
\begin{equation}
\label{SAE1}
[X_0, X_1]_{1/2}= \bm{W}^{1,p}_{\sigma,\vt}(\Omega)
\end{equation}
and then,
\begin{equation}
\label{SAE3}
X_{-1/2} = \left[ \bm{W}^{1,p'}_{\sigma,\vt}(\Omega)\right]'
\end{equation}
will follow.

To prove (\ref{SAE1}), 
one  inclusion is obvious. Indeed,
$$
[X_0, X_1]_{1/2} \subset [{L}^{p}_{\sigma,\vt}(\Omega),\bm{W}^{2,p}_{\sigma,\vt}(\Omega)]_{1/2} = \bm{W}^{1,p}_{\sigma,\vt}(\Omega) .
$$
And for the other inclusion, by (\ref{SAE2}) it is enough to prove that  $\bm{W}^{1,p}_{\sigma,\vt}(\Omega)\subset \mathbf{D}(A_0^{1/2})$. To this end,  first consider the operator $A_0^{1/2}$ on ${\bm L}^{p'} _{ \sigma , \vt  }(\Omega )$. Since $A_0$ has a bounded inverse, $A_0^{1/2}$ is an isomorphism from $\mathbf{D}(A_0^{1/2})$ to $\bm{L}^{p'}_{\sigma,\vt}(\Omega)$ \cite[Theorem 1.15.2, part(e)]{Triebel} and thus, for any $\mathbf{F}\in \bm{L}^{p^\prime}_{\sigma,\vt}(\Omega)$, there exists a unique $\bm{v}\in \mathbf{D}(A_0^{1/2})$ such that $A_0^{1/2} \bm{v} = \mathbf{F}$. So, for all $\vu \in \mathbf{D}(A_0)$,
\begin{align}
\label{SA3}
\|A_0^{1/2}\vu\|_{\bm{L}^p_{\sigma,\vt}(\Omega)} = \underset{\underset{ \mathbf{F}\neq \mathbf{0}}{\mathbf{F}\in \bm{L}^{p^\prime}_{\sigma,\vt}(\Omega)}}{\sup}\frac{\left| \left\langle A_0^{1/2}\vu, \mathbf{F}\right\rangle \right| }{\|\mathbf{F}\|_{\bm{L}^{p^\prime}_{\sigma,\vt}(\Omega)}}
& = \underset{\underset{\bm{v}\neq \mathbf{0}}{\bm{v}\in \mathbf{D}(A_{0}^{1/2})}}{\sup}\frac{\left| \left\langle A_0^{1/2}\vu, A_0^{\frac{1}{2}}\bm{v}\right\rangle \right| }{\|A_0^{\frac{1}{2}}\bm{v}\|_{\bm{L}^{p^\prime}_{\sigma,\vt}(\Omega)}}\notag \\
& = \underset{\underset{\bm{v}\neq \mathbf{0}}{\bm{v}\in \mathbf{D}(A_0^{1/2})}}{\sup}\frac{\left| \left\langle A_0\vu, \bm{v}\right\rangle \right| }{\|A_0^{1/2}\bm{v}\|_{\bm{L}^{p^\prime}_{\sigma,\vt}(\Omega)}}\notag \\
& = \underset{\underset{\bm{v}\neq \mathbf{0}}{\bm{v}\in \mathbf{D}(A_0^{1/2})}}{\sup}\frac{\left|\int\displaylimits_{\Omega}{\lambda \vu\cdot \bm{v}+ \mathbf{curl} \ \vu\cdot \mathbf{curl} \ \bm{v}} \right| }{\|A_0^{1/2}\bm{v}\|_{\bm{L}^{p^\prime}_{\sigma,\vt}(\Omega)}}\notag \\
& \leq C \|\vu\|_{\vW{1}{p}} .
\end{align}
Now as $\mathbf{D}(A_0)$ is dense in $\bm{W}^{1,p}_{\sigma,\vt}(\Omega)$, we get the inequality \eqref{SA3} for all $\vu\in\bm{W}^{1,p}_{\sigma,\vt}(\Omega)$ which gives the required embedding.

Now from \cite[Theorem 6.1]{AAE1}, we know that there exist constants $M>0$ and $\theta\in (0,\frac{\pi}{2})$ such that
$$\forall s\in \R,\,\,\, \| A_{0}^{is}\|_{\mathcal{L}(X_0)} \le M e^{|s|\theta }.$$
It then follows from \cite[Theorem V.1.5.5 (ii)]{amann} that
$$\forall s\in \R,\,\,\, \|\left( A_{-1/2}\right)  ^{is}\|_{\mathcal{L}(X_{-1/2})} \le M e^{|s|\theta } .$$
We call the operator $A_{-1/2}$ the weak Stokes operator subject to Navier-type boundary condition. Since $A_{-1/2}$ is the closure of $A_0$ in $X_{-1/2}$ and $X_1 \hookrightarrow X _{ 1/2 }$, it follows that $A_{-1/2}\vu = A_0 \vu$ for $\vu\in X_1$ and thus, for all $\bm{v}\in \bm{W}^{1,p'}_{\sigma,\vt}(\Omega)$,
$$
\left\langle \bm{v}, A_{-1/2}\vu\right\rangle _{ (X_{-1/2})'\times X_{-1/2}} = \left \langle \bm{v}, A_0 \vu\right\rangle = \lambda \int\displaylimits_{\Omega}{\vu\cdot \bm{v}} + \int\displaylimits_{\Omega}{\mathbf{curl} \ \vu \cdot \mathbf{curl} \ \bm{v}}
$$
where we only used integration by parts. Now using the density of $X_1$ in $X_{1/2}$, we obtain the relation, for all $\left( \vu,\bm{v}\right) \in \bm{W}^{1,p}_{\sigma,\vt}(\Omega) \times \bm{W}^{1,p'}_{\sigma,\vt}(\Omega)$,
\begin{eqnarray} 
\left\langle A_{-1/2}\vu,\bm{v}\right\rangle = \lambda \int\displaylimits_{\Omega}{\vu\cdot \bm{v}} +\int\displaylimits_{\Omega}{\mathbf{curl} \ \vu \cdot \mathbf{curl} \ \bm{v}} .\label{SA1}
\end{eqnarray}

{\bf 2.} Next let us define an unbounded operator $A_{N,w}$ on $X_{-1/2}$, with domain $X_{1/2}$, as, for all $\left( \vu,\bm{v}\right) \in \bm{W}^{1,p}_{\sigma,\vt}(\Omega) \times \bm{W}^{1,p'}_{\sigma,\vt}(\Omega)$,
\begin{eqnarray}
\left\langle A_{N,w}\vu,\bm{v}\right\rangle = \int\displaylimits_{\Omega}{\mathbf{curl} \ \vu \cdot \mathbf{curl} \ \bm{v}} + \left\langle \Lambda \vu,\bm{v}\right\rangle_{\Gamma} + \int\displaylimits_{\Gamma}{\alpha \vu\cdot \bm{v}}\label{SA2}
\end{eqnarray}
where $\Lambda$ is defined in (\ref{lambda}). We call the operator $A_{N,w}$ the weak Stokes operator subject to Navier boundary conditions. Comparing (\ref{SA2}) with (\ref{SA1}) implies 
$$
\left\langle \left( \lambda I + A_{N,w}\right) \vu,\bm{v}\right\rangle = \left\langle A_{-1/2} \vu,\bm{v}\right\rangle + \left\langle \Lambda_\alpha \vu,\bm{v}\right\rangle_{\Gamma}
$$
where the linear operator $ \Lambda_\alpha : X_{-1/2}\rightarrow X_{-1/2}$, given by,
$$
\left\langle \Lambda_\alpha \vu,\bm{v}\right\rangle_\Gamma = \left\langle \Lambda \vu,\bm{v}\right\rangle_{\Gamma} + \int\displaylimits_{\Gamma}{\alpha \vu\cdot \bm{v}}
$$
is a lower order perturbation of $A_{-1/2}$. Therefore, as $\alpha\in \Lb{\infty}$, it follows from \cite[Proposition 3.3.9]{pruss},
$$\forall s\in \R:\,\,\,\, \ \|\left[ (\lambda I + A_{N,w})\right] ^{is}\|_{\mathcal{L}(X_{-1/2})} \le M e^{|s|\theta_A } $$
for some constant  $\theta_A \in (0,\pi/2)$. Since, from \cite[Theorem 5.8]{AG}, $A_{N,w}$ has a bounded inverse it follows from \cite[Proposition 3.3.9]{pruss} again, that
$$ \| A_{N,w}^{is}\|_{\mathcal{L}(X_{-1/2})} \le M e^{|s|\theta_A } .$$

{\bf 3.} Now we want to transfer this 'bounded imaginary power' property to the strong Stokes operator $A_p$ with Navier boundary condition, defined in (\ref{S3E11})-(\ref{S3E12}) on $\bm{L}^p_{\sigma,\vt}(\Omega)$. For that we will apply again Amann's theory of interpolation-extrapolation scales. Let $X_0^w := \left[ \bm{W}^{1,p'}_{\sigma,\vt}(\Omega)\right] '$, $A^w_0 :=A_{N,w}$ and $X_1^w := \bm{W}^{1,p}_{\sigma,\vt}(\Omega)$. By \cite[Theorems V.1.5.1 and V.1.5.4]{amann}, the pair $(X_0^w, A^w_0)$ generates an interpolation-extrapolation scale $(X_a^w,A_a^w), a\in \R$ with respect to the complex interpolation functor and by \cite[Theorem V.1.5.5 (ii)]{amann}, for any $a\in \R$,

$$\forall s\in \R,\,\,\, \ \|( A^w_a) ^{is}\|_{\mathcal{L}(X_{a}^w)} \le M e^{|s|\theta_A } .$$

We will show in the remaining part  of this proof that the operator $A_{1/2}^w: X_{3/2}^w\subset X_{1/2}^w\rightarrow X_{1/2}^w$ coincides with $ A_p$ where the strong Stokes operator $A_p: \mathbf{D}(A_p)\subset \bm{L}^p_{\sigma,\vt}(\Omega) \rightarrow \bm{L}^p_{\sigma,\vt}(\Omega)$ is defined in (\ref{S3E11})-(\ref{S3E12}).
Observe that, by (\ref{SAE2}), (\ref{SAE3}),
$$
X_0^w = X_{-1/2} \quad \text{ and } \quad X_1^w = X_{1/2} .
$$
Therefore,
$$
X_{1/2}^w = \left[ X_0^w, X_1^w\right] _{1/2} = \left[ X_{-1/2}, X_{1/2}\right] _{1/2} = X_0 = \bm{L}^p_{\sigma,\vt}(\Omega) 
$$
and the operator $A_{1/2}^w$ is the restriction of $A_0^w$ on $X_{1/2}^w$. Hence, $A_{1/2}^w \vu = A_0^w \vu = A_{N,w} \vu$ for any $\vu \in \mathbf{D}(A_{1/2}^w) = X_{3/2}^w$ and then, for any $\bm{\varphi}\in \bm{W}^{1,p'}_{\sigma,\vt}(\Omega)$,
\begin{eqnarray}
&&\hskip -2cm \left\langle \bm{\varphi}, A_{1/2}^w \vu\right\rangle_{  \left( X_{1/2}^w\right) '\times X_{1/2}^w}  = 
\left\langle \bm{\varphi}, A_{N,w} \vu \right\rangle_{  \left( X_{1/2}^w\right) '\times X_{1/2}^w} 
\label{SAE4}\\
&& = \int\displaylimits_{\Omega}{\mathbf{curl} \ \vu \cdot \mathbf{curl} \ \bm{\varphi}} + \left\langle \Lambda \vu,\bm{\varphi}\right\rangle_{\Gamma} + \int\displaylimits_{\Gamma}{\alpha \vu\cdot \bm{\varphi}} .\label{SAE5}
\end{eqnarray}
On the other hand, for any $(\bm{v},\bm{\varphi})\in \mathbf{D}(A_p)\times \bm{W}^{1,p'}_{\sigma,\vt}(\Omega)$,  it follows from integration by parts that
\begin{eqnarray}
\label{SAE6}
\left\langle \bm{\varphi} , A_{p} \bm{v}\right\rangle_{ \left( X_{1/2}^w\right) '\times  X_{1/2}^w}   = \int\displaylimits_{\Omega}{\mathbf{curl} \ \bm{v} \cdot \mathbf{curl} \ \bm{\varphi}}+ \left\langle \Lambda \bm{v},\bm{\varphi}\right\rangle_{\Gamma} + \int\displaylimits_{\Gamma}{\alpha \bm{v} \cdot \bm{\varphi}}.
\end{eqnarray}
Now for any given $\vu \in \mathbf{D}(A_{1/2}^w)$, $A_{1/2}^w\vu \in \bm{L}^p_{\sigma, \vt}(\Omega)$ and then there exists a unique $\bm{v}\in \mathbf{D}(A_p)$ such that
$$A_p \bm{v} = A_{1/2}^w \vu $$
since $A_p$ is onto. Thus it follows from (\ref{SAE4}) that for any $\bm{\varphi }\in \bm{W}^{1,p'}_{\sigma,\vt}(\Omega)$,
$$
\left\langle \bm{\varphi },  A_p \bm{v} \right\rangle_{\left( X_{1/2}^w\right) '\times X_{1/2}^w}  =\left\langle \bm{\varphi }, A_{N,w}\vu \right\rangle_{ \left( X_{1/2}^w\right) '\times X_{1/2}^w}
$$
This in turn implies by (\ref{SAE5}) and (\ref{SAE6}) that
$$\left\langle \bm{\varphi }, A_{N,w}\vu \right\rangle_{ \left( X_{1/2}^w\right) '\times X_{1/2}^w} =\left\langle \bm{\varphi }, A_{N,w} \bm{v} \right\rangle_{ \left( X_{1/2}^w\right) '\times X_{1/2}^w}. $$
Hence, $\bm{v} = \vu$ by injectivity of $A_{N,w}$. Similarly, if $\bm{v}\in \mathbf{D}(A_p)$ is given, then there exists a unique $\vu \in \mathbf{D}(A_{1/2}^w)$ such that $ A_{1/2}^w\vu = A_p \bm{v}$ since $A_p \bm{v} \in \bm{L}^p_{\sigma, \vt}(\Omega)$ and $A_{1/2}^w$ is onto. By the same argument as above, we obtain $\vu = \bm{v}$ showing that $\mathbf{D}(A_p) = \mathbf{D}(A_{1/2}^w)$ and $A_p = A_{1/2}^w$.
Thus finally we get that,
\begin{equation*}
\forall s\in \R,\,\,\,\,\| A_p^{is}\|_{\mathcal{L}(\bm{L}^p_{\sigma,\vt}(\Omega))} \le M e^{|s|\theta_A} .
\end{equation*}
\hfill
\end{proof}

\subsection{Fractional  powers}
The above result allows us to study the domains of $A^\beta_p$, $\beta\in \R$. It can be shown that $\bm{D}(A^\beta_{p,\alpha})$ is a Banach space with the graph norm which is equivalent to the norm $\|A^\beta_{p,\alpha} \cdot\|_{\vL{p}}$, since $A_{p,\alpha}$ has bounded inverse. Note that for any $\beta \in\R$, the map $\vu \rightarrow \|A^\beta_{p,\alpha}\vu\|_{\vL{p}}$ defines a norm on $\bm{D}(A^\beta_{p,\alpha})$ due to the injectivity of $A^\beta_{p,\alpha}$.

\begin{theorem}
\label{45}
	For all $p\in (1,\infty)$, $\mathbf{D}(A_{p,\alpha}^{1/2}) = \bm{W}^{1,p}_{\sigma,\vt}(\Omega)$ with equivalent norms.
\end{theorem}

\begin{proof}
Since the pure imaginary power of $A_{p,\alpha}$ is bounded and satisfies estimate \eqref{53}, using the result \cite[Theorem 1.15.3]{Triebel}, we get that
	\begin{equation}
	\label{41}
\mathbf{D}(A_{p,\alpha}^{1/2})= [\bm{L}^p_{\sigma,\vt}(\Omega), \mathbf{D}(A_{p,\alpha})]_{\frac{1}{2}} .
	\end{equation}
Then it is enough to show that
$$
[\bm{L}^p_{\sigma,\vt}(\Omega), \mathbf{D}(A_{p,\alpha})]_{\frac{1}{2}} = \mathbf{W}^{1,p}_{\sigma,\vt}(\Omega)
$$
with equivalent norms, which is already proved in (\ref{SAE1}). 
\hfill 
\end{proof}

\begin{remark}
\rm{	If $\Omega$ is not obtained by rotation around an axis i.e. if $\Omega$ is not axisymmetric, the norms $\|\vu\|_{\vW{1}{p}}$ and $\|\DT\vu\|_{\vL{p}}$ are equivalent for $\vu\in\vW{1}{p}$ with $\vu\cdot \vn = 0$ on $\Gamma$, as shown in \cite[Proposition 3.7]{AG}. As a result we have the following equivalence for all $\vu\in \mathbf{D}(A_{p,\alpha}^{1/2})$:
	$$\|\DT\vu\|_{\vL{p}} \simeq \|A_{p,\alpha}^{1/2}\vu\|_{\vL{p}} .$$}
\end{remark}

Our next result is an embedding theorem of Sobolev type for domains of fractional powers which will be applied to deduce the so-called $L^p-L^q$ estimates for the solution of the evolutionary Stokes equation.
\begin{theorem}\label{44}
	For all $1<p<\infty$ and for all $\beta \in \mathbb{R}$ such that $0<\beta <\frac{3}{2p}$, the following embedding holds :
	\begin{equation*}
	\mathbf{D}(A_{p,\alpha}^\beta)\hookrightarrow \vL{q} \quad \text{ where } \ \frac{1}{q} = \frac{1}{p}-\frac{2\beta}{3} .
	\end{equation*}
\end{theorem}

\begin{proof}
	First observe that for $0\leq \theta\leq 1$, by the result \cite[Theorem 1.15.3]{Triebel} and the estimate \eqref{53}, we can write
	\begin{equation}\label{43}
	\mathbf{D}(A_{p,\alpha}^{\theta}) = [\bm{L}^p_{\sigma,\vt}(\Omega), \mathbf{D}(\lambda I+A_{p,\alpha})]_{\theta}
	\hookrightarrow [\vL{p}, \vW{2}{p}]_\theta \hookrightarrow \vW{2\theta}{p} \hookrightarrow \vL{q}
	\end{equation}
	where
	$$\frac{1}{q} = \frac{1}{p}-\frac{2\theta}{3} \quad \text{ when } \ p < \frac{3}{2\theta} .$$
	Now let $\beta = \theta +k$ where $0\leq \theta <1$ and $k\in \mathbb{N}\cup \{0\}$. Consider $m$ large so that $\mathbf{D}(A^m_{p,\alpha})\subset \mathbf{D}(A^\beta_{q,\alpha})$ where $\frac{1}{q} = \frac{1}{p}-\frac{2\beta}{3}$. Also, by the definition of $q$, it is obvious that $\mathbf{D}(A^\beta_{q,\alpha}) \subset \mathbf{D}(A^\beta_{p,\alpha})$. 
	If we set
	$$\frac{1}{q_0} = \frac{1}{p}- \frac{2\theta}{3} \quad \text{ and } \quad \frac{1}{q_j} = \frac{1}{q_0}-\frac{2j}{\theta} \quad \text{ for } \ 0\leq j\leq k ,$$
	then we have $\frac{1}{q_j} = \frac{1}{q_{j-1}} - \frac{2}{3}$ for $1\leq j\leq k$ and $q_k = q$. Moreover, $q_{j-1}<\frac{3}{2}$ for $1\leq j\leq k$ by assumptions on $p$ and $\beta$. Hence as the consequence of the embedding \eqref{43}, we get that
	$$
	\mathbf{D}(A_{p,\alpha}^\theta) \hookrightarrow \vL{q_0}
	$$
	and
	$$\mathbf{D}(A_{q_{j-1},\alpha}) \hookrightarrow \vL{q_j} \ \text{ for } \ 1\leq j\leq k .
	$$
	Thus it follows that for all $\vu\in \mathbf{D}(A^m_p)$,
	$$
	\|\vu\|_{\vL{q}} \leq C \|A_{q_{k-1},\alpha}\vu\|_{\vL{q_{k-1}}} \leq ... \leq C \|A^k_{q_0,\alpha}\vu\|_{\vL{q_0}} \leq C \|A^\beta _{p,\alpha} \vu\|_{\vL{p}} .
	$$
	By density of $\mathbf{D}(A^m_{p,\alpha})$ in $\mathbf{D}(A^\beta_{p,\alpha})$, we get the final result.
	\hfill
\end{proof}

\section{The homogeneous Stokes problem}
\setcounter{equation}{0}

In this section, with the help of the semigroup theory, we solve the homogeneous time dependent Stokes problem:
\begin{equation}
\label{evolutionary Stokes}
\left\{
\begin{aligned}
\frac{\partial \vu}{\partial t} -\Delta \vu + \nabla \pi = \bm{0} , \quad \div \ \vu = 0 \quad & \text{ in } \Omega\times (0,T) ,\\
\vu\cdot \vn = 0, \quad 2[(\DT\vu)\vn]_{\vt}+\alpha \vu_{\vt} = \bm{0} \quad & \text{ on } \Gamma\times (0,T) ,\\
\vu(0) = \vu_0 \quad & \text{ in } \Omega
\end{aligned}
\right.
\end{equation}
for which the analyticity of the semigroups, considered before give a unique solution satisfying the usual regularity.

\subsection{Strong solution}
We start with the strong solution of the problem \eqref{evolutionary Stokes}.
\begin{theorem}
\label{35}
Let $ p\in (1,\infty)$ and $\alpha$ be as in \eqref{alpha}. Then for $\vu_0\in \bm{L}^p_{\sigma,\vt}(\Omega)$, the problem \eqref{evolutionary Stokes} has a unique solution $\vu(t)$ satisfying
\begin{equation}
\label{32}
\vu\in C([0,\infty),\bm{L}^p_{\sigma,\vt}(\Omega) )\cap C((0,\infty),\mathbf{D}(A_{p,\alpha}))\cap C^1((0,\infty),\bm{L}^p_{\sigma,\vt}(\Omega)) 
\end{equation}
and
\begin{equation}
\label{33}
\vu\in C^k((0,\infty),\mathbf{D}(A_{p,\alpha}^l)) \quad \forall \ k\in\mathbb{N}, \ \forall \ l\in\mathbb{N}\backslash\{0\} .
\end{equation}
Also we have the estimates, for some constant $C>0$ independent of $\alpha$,
\begin{equation}
\label{30}
\|\vu(t)\|_{\vL{p}} \leq C \|\vu_0\|_{\vL{p}}
\end{equation}
and
\begin{equation}
\label{31}
\left\| \frac{\partial \vu(t)}{\partial t}\right\| _{\vL{p}} \leq \frac{C}{t} \|\vu_0\|_{\vL{p}}.
\end{equation}
Moreover, if $\alpha$ is a constant and either (i) $\Omega$ is not axisymmetric or (ii) $\Omega$ is axisymmetric and $\alpha\ge \alpha_*>0$, then
\begin{equation}
\label{34}
\left\| \DT \vu(t)\right\| _{\vL{p}} \leq \frac{C}{\sqrt{t}} \ \|\vu_0\|_{\vL{p}}
\end{equation}

\begin{equation}
\label{46}
\|\vu(t)\|_{\vW{2}{p}} \leq \frac{C}{t} \|\vu_0\|_{\vL{p}}
\end{equation}
and
\begin{equation}
\label{12}
 \|\nabla \pi\|_{\L{p}} \leq \frac{C}{t} \|\vu_0\|_{\vL{p}} .
\end{equation}
\end{theorem}

\begin{proof}
Since $-A_{p,\alpha}$ generates an analytic semigroup for every $\vu_0\in\bm{L}^p_{\sigma,\vt}(\Omega)$, the initial value problem \eqref{evolutionary Stokes} has a unique solution $\vu(t) = T(t)\vu_0$, by \cite[Corollary 1.5, Chapter 4, page 104]{Pazy}. Also, from \cite[Theorem 7.7, Chapter 1, Page 30]{Pazy}, we get that
$$\|T(t)\| \leq C \quad \text{ for some constant } C>0, \text{ independent of } \alpha .$$ 
As a result, we obtain the estimate \eqref{30}. Also, with the help of \cite[point (d), Theorem 5.2, Chapter 2]{Pazy}, we get the estimate \eqref{31}. To prove the estimate \eqref{34}, we need to proceed as in the proof of \eqref{14}, hence we skip it. The estimate \eqref{46} follows from \eqref{31} and using the fact that $\|\bm{v}\|_{\vW{2}{p}} \simeq \|A_{p,\alpha} \bm{v}\|_{\vL{p}}$ for $\bm{v}\in\mathbf{D}(A_{p,\alpha})$. 

Further, using the usual regularity properties of semi group and by  \cite[Lemma 4.2, chapter 2]{Pazy}, we can deduce the regularity \eqref{32} and \eqref{33}. The estimate on the pressure term (\ref{12}) can be deduced from the equation using (\ref{31}) and (\ref{46}).
\hfill
\end{proof}

The estimates \eqref{30}, \eqref{31} and \eqref{34} allow us to deduce the following regularity result.

\begin{corollary}
\label{36}
Let $p\in ( 1,\infty)$ and $\alpha$ be as in \eqref{alpha}. Moreover, $\vu_0 \in \bm{L}^p_{\sigma,\vt}(\Omega), 0<T<\infty$ and $(\vu,\pi)$ be the unique solution of problem \eqref{evolutionary Stokes} given by theorem \ref{35}. Then, for all $1\leq q<2$, we have,
$$\vu\in L^q(0,T;\vW{1}{p}), \pi\in L^q(0,T;L^p_0(\Omega)) \text{ and } \frac{\partial \vu}{\partial t} \in L^q(0,T;[\bm{H}_0^{p^\prime}(\div ,\Omega)]^\prime) .$$
\end{corollary}

\begin{proof}
Since we have the Korn inequality
$$\|\vu(t)\|_{\vW{1}{p}} \leq \|\vu(t)\|_{\vL{p}} + \|\DT\vu(t)\|_{\vL{p}} $$
and $\vu(t)$ satisfies the estimates \eqref{30} and \eqref{34}, we get,
$$\|\vu(t)\|^q_{\vW{1}{p}} \leq  C(1+t^{-q/2}) \|\vu_0\|_{\vL{p}}$$
which implies $\vu\in L^q(0,T;\vW{1}{p})$ only for $1\leq q <2$ and for all $0<T<\infty$.

Moreover, as the operator $B_{p,\alpha} : \bm{D}(B_{p,\alpha})\rightarrow [\bm{H}_0^{p^\prime}(\div ,\Omega)]^\prime$ is an isomorphism, we have the equivalence of norm, for any $\bm{v}\in \bm{D}(B_{p,\alpha})$, $\|B_{p,\alpha}\bm{v} \|_{[\bm{H}_0^{p^\prime}(\div ,\Omega)]^\prime}\simeq \|\bm{v}\|_{\bm{D}(B_{p,\alpha})}$ and here $B_{p,\alpha} \bm{\vu} = \frac{\partial \vu}{\partial t}$. Thus $\frac{\partial \vu}{\partial t} \in L^q(0,T;[\bm{H}_0^{p^\prime}(\div ,\Omega)]^\prime)$.

Finally from the equation $\nabla \pi = \Delta \vu - \frac{\partial \vu}{\partial t}$, the regularity of $\pi$ follows.
\hfill
\end{proof}

\begin{theorem}
\label{S6Thm2}	
 Let $\alpha $ satisfy (\ref{alpha}). Then for all $p\leq q <\infty$ and $\vu_0\in \bm{L}^p_{\sigma,\vt}(\Omega)$, there exists $\delta>0$ such that the unique solution $\vu(t)$ of the problem (\ref{evolutionary Stokes}) belongs to $\vL{q}$ and satisfies, for all $t>0$ :
\begin{equation}
\label{38}
\|\vu(t)\|_{\vL{q}} \leq C(\Omega, p) \ e^{-\delta t}t^{-3/2 (1/p-1/q)} \|\vu_0\|_{\vL{p}} .
\end{equation}
Moreover, the following estimates also hold
\begin{equation}
\label{39}
\|\DT{\vu(t)}\|_{\vL{q}} \leq C(\Omega,p) \ e^{-\delta t}t^{-3/2(1/p-1/q)-1/2} \|\vu_0\|_{\vL{p}} ,
\end{equation}
\begin{equation}
\label{40}
\forall \ m,n\in\mathbb{N}, \quad \|\frac{\partial ^m}{\partial t^m}A_{p,\alpha}^n \vu(t)\|_{\vL{q}} \leq C(\Omega, p) \ e^{-\delta t}t^{-(m+n)-3/2(1/p-1/q)} \|\vu_0\|_{\vL{p}} .
\end{equation}
Note that all the above constants $C(\Omega,p)$ are independent of $\alpha$.
\end{theorem}

\begin{proof}
First observe that in the case of $p=q$, the estimates \eqref{38}, \eqref{39} and \eqref{40} follow from the classical semi group theory and the result that $\|T(t)\| \leq M e^{-\delta t}$ (\cite[Theorem 6.13, Chapter 2]{Pazy}).
	
Suppose that $p\neq q$. Let $s\in\mathbb{R}$ such that $\frac{3}{2}(\frac{1}{p}-\frac{1}{q})<s<\frac{3}{2p}$ and set $\frac{1}{p_{0}}=\frac{1}{p}-\frac{2s}{3}$. It is clear that $p<q<p_{0}$. Since for all $t>0$ and for all $l\in \R^+$, $\boldsymbol{u}(t)\in\mathbf{D}(A_{p,\alpha}^{l})$, thanks to Theorem \ref{44}, $\boldsymbol{u}(t)\in\mathbf{D}(A_{p,\alpha}^{s})\hookrightarrow\boldsymbol{L}^{p_{0}}(\Omega)$. Now $\frac{1}{q}=\frac{\alpha}{p_{0}}+\frac{1-\alpha}{p}$ for $\alpha=\frac{1/p-1/q}{1/p-1/p_{0}}\in( 0,1) $. Thus $\boldsymbol{u}(t)\in\boldsymbol{L}^{q}(\Omega)$ and
 \begin{eqnarray*}
 \Vert\boldsymbol{u}(t)\Vert_{\boldsymbol{L}^{q}(\Omega)} \leq C \Vert\boldsymbol{u}(t)\Vert^{\alpha}_{\boldsymbol{L}^{p_{0}}(\Omega)}
 \Vert\boldsymbol{u}(t)\Vert^{1-\alpha}_{\boldsymbol{L}^{p}(\Omega)}
 &\leq &C \Vert A_{p,\alpha}^{s}T(t)\boldsymbol{u}_{0}\Vert^{\alpha}_{\boldsymbol{L}^{p}(\Omega)} \Vert T(t)\boldsymbol{u}_{0}\Vert^{1-\alpha}_{\boldsymbol{L}^{p}(\Omega)}\nonumber\\
 &\leq& C\,e^{-\delta t}t^{-\alpha s} \Vert\boldsymbol{u}_{0}\Vert_{\boldsymbol{L}^{p}(\Omega)}
\end{eqnarray*}
where the last estimate follows from \cite[Chapter 2, Theorem 6.13]{Pazy}.

In order to prove (\ref{40}) we first obtain from \eqref{32}: $\frac{\partial ^m}{\partial t^m}A_{p,\alpha}^n \vu(t) \in \vL{q}$ for any $m,n \in \mathbb{N}$ and then
\begin{align*}
\|\frac{\partial ^m}{\partial t^m}A_{p,\alpha}^n \vu(t)\|_{\vL{q}} = \|A_{p,\alpha}^{(m+n)}T(t)\vu_0\|_{\vL{q}} \leq C e^{-\delta t}t^{-(m+n)-3/2(1/p-1/q)} \|\vu_0\|_{\vL{p}} .
\end{align*}

To prove estimate (\ref{39}), we first deduce from  (\ref{40}) for $m=1$ and $n=0$:
\begin{equation}
\label{S5E1}
\|A_{p,\alpha}\vu(t)\|_{\vL{q}} \le Ce^{-\delta t}t^{-1-\frac {3}{2}\left(\frac {1} {p}-\frac {1} {q} \right)}.
\end{equation}
Then, from Gagliardo Nirenberg's inequality, we obtain
\begin{align*}
\|\DT\vu(t)\|_{\vL{q}} &\leq C \ \|\vu(t)\|_{\vW{1}{q}} \leq C \ \|\vu(t)\|^{1/2}_{\vW{2}{q}}\|\vu(t)\|^{1/2}_{\vL{q}}\\
&\leq C \ \|\vu(t)\|^{1/2}_{\mathbf{D}(A_{q,\alpha})   }\|\vu(t)\|^{1/2}_{\vL{q}}\leq C \ \| A_{q,\alpha}\vu(t)\|^{1/2}_{\vL{q}} \|\vu(t)\|^{1/2}_{\vL{q}}.
\end{align*}
Thus (\ref{39}) follows from (\ref{38}) and (\ref{S5E1}).
\hfill
\end{proof}

\begin{proof}[\bf Proof of Theorem \ref{homstrong}]
This essentially follows from Theorem \ref{35} and Theorem \ref{S6Thm2}. 
\hfill	
\end{proof}

\subsection{Weak solution}

The following result says that if the initial data is in $[\bm{H}_0^{p^\prime}(\div ,\Omega)]^\prime$, we have the weak solution for the homogeneous problem \eqref{evolutionary Stokes}. Here, as in Theorem \ref{35}, we use the analyticity of the semi group generated by the operator $B_{p,\alpha}$ and the fact that $\|T(t)\| \leq M e^{-\delta t}$. 

\begin{theorem}\label{37}
Let $ 1<p<\infty$ and $\alpha \in\Lb{t(p)}$ where $t(p)$ defined in \eqref{22}. Then, for all $\vu_0\in [\bm{H}_0^{p^\prime}(\div ,\Omega)]^\prime$, the problem \eqref{evolutionary Stokes} has a unique solution $\vu(t)$ with the regularity
\begin{equation*}
\vu\in C([0,\infty),[\bm{H}_0^{p^\prime}(\div ,\Omega)]^\prime)\cap C((0,\infty),\mathbf{D}(B_{p,\alpha}))\cap C^1((0,\infty),[\bm{H}_0^{p^\prime}(\div ,\Omega)]^\prime) 
\end{equation*}
and
\begin{equation*}
\vu\in C^k((0,\infty),\mathbf{D}(B_{p,\alpha}^l)) \quad \forall \ k\in\mathbb{N}, \ \forall \ l\in\mathbb{N}\backslash\{0\} .
\end{equation*}
Also there exists constants $C>0$, independent of $\alpha$ and $\delta >0$ such that for all $t>0$,
\begin{equation*}
\|\vu(t)\|_{[\bm{H}_0^{p^\prime}(\div ,\Omega)]^\prime} \leq C e^{-\delta t} \|\vu_0\|_{[\bm{H}_0^{p^\prime}(\div ,\Omega)]^\prime}
\end{equation*}
and
\begin{equation*}
\left\| \frac{\partial \vu(t)}{\partial t}\right\| _{[\bm{H}_0^{p^\prime}(\div ,\Omega)]^\prime} \leq C \ \frac{e^{-\delta t}}{t} \|\vu_0\|_{[\bm{H}_0^{p^\prime}(\div ,\Omega)]^\prime}.
\end{equation*}
Moreover, if (i) either $\Omega$ is not axisymmetric or (ii) $\Omega$ is axisymmetric and $\alpha\ge \alpha_*>0$, then
\begin{equation*}
\label{47}
\|\vu\|_{\vW{1}{p}} \leq C \ \frac{e^{-\delta t}}{t} \|\vu_0\|_{[\bm{H}_0^{p^\prime}(\div ,\Omega)]^\prime}
\end{equation*}
and 
\begin{equation*}
\|\nabla \pi\|_{[\bm{H}_0^{p^\prime}(\div ,\Omega)]^\prime} \leq C \ \frac{e^{-\delta t}}{t} \|\vu_0\|_{[\bm{H}_0^{p^\prime}(\div ,\Omega)]^\prime} .
\end{equation*}
\end{theorem}

In the same way as we deduced in Corollary \ref{36}, we can have the following regularity result from Theorem \ref{37}.
\begin{corollary}
Let $ p\in (1,\infty)$ and $\alpha \in \Lb{t(p)}$ where $t(p)$ defined in \eqref{22}. Moreover, suppose $\vu_0 \in [\bm{H}_0^{p^\prime}(\div ,\Omega)]^\prime, 0<T<\infty$ and $(\vu,\pi)$ be the unique solution of problem \eqref{evolutionary Stokes} given by theorem \ref{37}. Then, for all $1\leq q<2$, we have
$$\vu\in L^q(0,T;\vL{p}) .$$
\end{corollary}

\begin{proof}
We know the interpolation inequality
$$\|\vu(t)\|_{\vL{p}} \leq \|\vu(t)\|_{\vW{1}{p}}^{1/2} \|\vu(t)\|_{\vW{-1}{p}}^{1/2} .$$
Now using the estimates in Theorem \ref{47} and the fact that $[\bm{H}_0^{p^\prime}(\div ,\Omega)]^\prime \hookrightarrow \vW{-1}{p}$ in the above inequality, we get the result.
\hfill
\end{proof}

\section{The non-homogeneous Stokes problem}
\setcounter{equation}{0}

Here we discuss the non-homogeneous Stokes problem:
\begin{equation}
\label{58}
\left\{
\begin{aligned}
\frac{\partial \vu}{\partial t} -\Delta \vu + \nabla \pi = \bm{f} , \quad \div \ \vu = 0 \quad & \text{ in } \Omega\times (0,T) ,\\
\vu\cdot \vn = 0, \quad 2[(\DT\vu)\vn]_{\vt}+\alpha \vu_{\vt} = \bm{0} \quad & \text{ on } \Gamma\times (0,T) ,\\
\vu(0) = \vu_0 \quad & \text{ in } \Omega .
\end{aligned}
\right.
\end{equation}
It is known that if $-\mathcal{A}$ generates a bounded analytic semigroup on a Banach space $X$, then we can construct a strong solution of
\begin{equation}
\label{57}
u^\prime + \mathcal{A}u = f \ \text{ for a.e. } t\in (0,T), \quad u(0)=a
\end{equation}
if $f$ is H\"{o}lder continuous in time with values in $X$. But the analyticity of $e^{-t\mathcal{A}}$ is not sufficient to deduce the existence of solutions of \eqref{57} for general $f\in L^p(0,T;X)$ unless $X$ is a Hilbert space. Therefore, we use the result on abstract Cauchy problem by Giga and Sohr \cite[Theorem 2.3]{Giga91} which used the notion of $\zeta$-convexity. For completeness, we recall the definition of $\zeta$-convexity (see \cite{Bourgain}, also refer to \cite{rubio}):

A Banach space $X$ is said to be $\zeta$-convex if there exists a symmetric biconvex function $\zeta$ on $X\times X$ such that $\zeta(0,0)>0$ and
$$\zeta(x,y) \leq \|x+y\| \quad \text{ if } \quad \|x\| \leq 1\leq \|y\| .$$
The concept of $\zeta$-convexity is stronger than that of reflexivity. For application purpose, it is important to recall \cite{Giga91} that $X$ is $\zeta$-convex iff for some $1<s<\infty$, the truncated Hilbert transform
$$(H_\varepsilon f)(t) = \frac{1}{\pi}\int_{|\tau| >\varepsilon}^{} \frac{f(t-\tau)}{\tau} d\tau , \quad f\in L^s(\mathbb{R},X)$$
converges as $\varepsilon \rightarrow 0$ for almost all $t\in\mathbb{R}$ and there is a constant $C=C(s,x)$ independent of $f$ such that
$$\|Hf\|_{L^s(\mathbb{R},X)} \leq C \|f\|_{L^s(\mathbb{R},X)} $$
where $(Hf)(t) = \lim\limits_{\varepsilon \rightarrow 0} (H_\varepsilon f)(t) $.

The result in \cite[Theorem 2.3]{Giga91} is useful in two senses : (i) it can be used even when $\mathcal{A}$ does not have a bounded inverse (though in our case, both $A_p$ and $B_p$ have bounded inverse) and (ii) the constant in the estimate is independent of time $T$, hence gives global in time results.

Here we introduce the notation for the space, for any $1<p,q<\infty$,
$$D^{\frac{1}{q},p}_{\mathcal{A}} = \left\lbrace v \in X : \|v\|_{D^{\frac{1}{q},p}_{\mathcal{A}} } = \|v\|_X +\left(   \int_{0}^{\infty}\|t^{1-\frac{1}{q}}\mathcal{A}e^{-t\mathcal{A}}v\|^p_X \frac{dt}{t} \right)^{1/p} < \infty\right\rbrace $$
which actually agrees with the real interpolation space $\left( D(\mathcal{A}), X \right) _{1-1/q,p}$ when $e^{-t\mathcal{A}}$ is an analytic semigroup. First we deduce the strong solution of the Stokes system \eqref{58} and obtain $L^p-L^q$ estimates.

\begin{proof}[\bf Proof of Theorem \ref{nonhom}]
Since $\bm{L}^p_{\sigma, \vt}(\Omega)$ is $\zeta$-convex \cite[page 81]{Giga91} and $A_{p,\alpha}$ satisfies the estimate \eqref{53}, all the assumptions of \cite[Theorem 2.3]{Giga91} are fulfilled with $\mathcal{A}= A_{p,\alpha}$ and $X= \bm{L}^p_{\sigma,\vt}(\Omega)$. As a result, the regularity of $\vu$ and $ \frac{\partial \vu}{\partial t}$ follow. The regularity of $\pi$ comes from the fact that 
\begin{equation}
\label{65}
\nabla \pi = \bm{f} - \frac{\partial \vu}{\partial t} + \Delta \vu .
\end{equation}
Also we get the estimate
$$\int_{0}^{T}\left\| \frac{\partial \vu}{\partial t}\right\| ^q_{\vL{p}} dt + \int_{0}^{T}\|A_{p,\alpha} \vu\|^q_{\vL{p}} dt \leq C \left( \int_{0}^{T} \|\bm{f}(t)\|^q_{\vL{p}} dt + \|\vu_0\|^q_{D^{1-\frac{1}{q},q}_{A_{p,\alpha}}} \right) $$
which yields \eqref{59} using the fact that $A_{p,\alpha} \vu = -\Delta \vu + \nabla \pi$.
\hfill
\end{proof}

In the same way, using $\mathcal{A} = B_{p,\alpha}$ and $X = [\bm{H}_0^{p^\prime}(\div ,\Omega)]^\prime _{\sigma, \vt} $ in \cite[Theorem 2.3]{Giga91}, since $[\bm{H}_0^{p^\prime}(\div ,\Omega)]^\prime _{\sigma, \vt}$ is $\zeta$-convex \cite[Proposition 2.16]{AAE1} and $B_{p,\alpha}$ satisfies the estimate on the pure imaginary power (\ref{53.}), we get the weak solution of the problem \eqref{58} with corresponding estimates as follows:

\begin{theorem}
\label{S7Thm1}
	Let $0< T \leq \infty, 1< p,q <\infty$ and $\alpha\in\Lb{t(p)}$ where $t(p)$ be defined in \eqref{22}. Then for every $\bm{f}\in L^q(0,T;[\bm{H}_0^{p^\prime}(\div ,\Omega)]^\prime _{\sigma, \vt})$ and $\vu_0 \in D^{1-\frac{1}{q},q}_{B_{p,\alpha}}$ there exists a unique solution $(\vu,\pi)$ of \eqref{58} satisfying the properties :
	\begin{equation*}
	\begin{aligned}
	\vu \in L^q(0,T_0;\vW{1}{p}) \ \text{ for all } \ T_0\leq T \ \text{ if } \ T<\infty \ \text{ and } \ T_0 < \infty \ \text{ if } \ T= \infty ,
	\end{aligned}
	\end{equation*}
	
	\begin{equation*}
	\begin{aligned}
	\pi\in L^q(0,T; L^p_0(\Omega)) , \quad \frac{\partial \vu}{\partial t} \in L^q(0,T, [\bm{H}_0^{p^\prime}(\div ,\Omega)]^\prime _{\sigma, \vt}) ,
	\end{aligned}
	\end{equation*}
	
	\begin{equation*}
	\begin{aligned}
	&\int_{0}^{T}\left\| \frac{\partial \vu}{\partial t}\right\| ^q_{[\bm{H}_0^{p^\prime}(\div ,\Omega)]^\prime} dt + \int_{0}^{T}\| \vu\|^q_{\vW{1}{p}} dt + \int_{0}^{T} \|\pi\|^q_{\L{p}/\R} dt\\
	& \qquad \leq C \left( \int_{0}^{T} \|\bm{f}\|^q_{[\bm{H}_0^{p^\prime}(\div ,\Omega)]^\prime} dt + \|\vu_0\|^q_{D^{1-\frac{1}{q},q}_{B_{p,\alpha}}}\right)  .
	\end{aligned}
	\end{equation*}
\end{theorem}

\section{Nonlinear problem}
\setcounter{equation}{0}

In this section, we consider the initial value problem for the Navier-Stokes system with Navier boundary condition :
\begin{equation}
\label{NS}
\left\{
\begin{aligned}
\frac{\partial \vu}{\partial t} -\Delta \vu + (\vu\cdot \nabla) \vu + \nabla \pi = \bm{0} , \quad \div \ \vu = 0 \quad & \text{ in } \Omega\times (0,T) ,\\
\vu\cdot \vn = 0, \quad 2[(\DT\vu)\vn]_{\vt}+\alpha \vu_{\vt} = \bm{0} \quad & \text{ on } \Gamma\times (0,T) ,\\
\vu(0) = \vu_0 \quad & \text{ in } \Omega .
\end{aligned}
\right.
\end{equation}
The semi group theory formulated in Section \ref{66}, \ref{67} and \ref{frac power} for the Stokes operator provides us the necessary properties with which we can obtain some existence, uniqueness and regularity result for the non-linear problem as well. Here we want to  employ the results of \cite{Giga86} for the abstract semi linear parabolic equation of the form
\begin{equation}
\label{S8E1}
u_t + \mathcal{A}u = F u, \quad u(0) = a
\end{equation}
where $F u$ represents the nonlinear part and $\mathcal{A}$ is an elliptic operator. This abstract theory gives the existence of a local solution $u(t)$ for certain class of $F u$. The solution can be extended globally also, provided norm of the initial data is sufficiently small. Moreover, this solution belongs to $L^q(0,T; L^p)$ with suitably chosen $p,q$. Since, $u\in L^q(0,T; L^p)$ is equivalent of saying $\| u(t)\|_{L^p(\Omega)} \in L^q(0,T)$, this gives the asymptotic behaviour of $\|u(t)\|_{L^p(\Omega)}$ as $t\rightarrow 0$ and $t\rightarrow \infty$.

To apply \cite[Theorem 1 and Theorem 2]{Giga86}, we need to verify the hypothesis therein, which we state below for convenience:

For a closed subspace $E^p$ of $\L{p}$, let $P: \L{p}\rightarrow E^p$ be a continuous projection for $p\in (1,\infty)$ such that the restriction of $P$ on $C_c(\Omega)$, the space of continuous functions with compact support, is independent of $p$ and $C_c(\Omega) \cap E^p$ be dense in $E^p$. Let $e^{-t\mathcal{A}}$ be a strongly continuous operator on $E^p$ for all $p\in (1, \infty)$. Also, there exists constants $n,m\ge 1$ such that for a fixed $T\in (0,\infty)$, the estimate
\leqnomode
\begin{equation}
\tag{\bf A}
\|e^{-t\mathcal{A}}f\|_{\L{p}} \le M \|f\|_{\L{s}}/t^\sigma ,\quad f\in E^s,\quad t\in(0,T)
\end{equation}
holds with $\sigma = (\frac{1}{s} - \frac{1}{p})\frac{n}{m}$ for $p\ge s>1$ and constant $M$ depending only on $p,s,T$.

Moreover, let $Fu$ be written as
$$Fu = LGu$$
where $L$ is a closed, linear operator, densely defined from $\L{p}$ to $E^q$ for some $q>1$ such that for some $\gamma, 0\le \gamma<m$, the estimate
\begin{equation}
\tag{\bf N1}
\|e^{-t\mathcal{A}}Lf\|_{\L{p}} \le N_1 \|f\|_{\L{p}}/t^{\gamma/m}, \quad f\in E^p, \quad t\in(0,T)
\end{equation}
holds with $N_1$ depending only on $ T$ and $p$, for all $p\in (1,\infty)$ and $G$ is a nonlinear mapping from $E^p$ to $\L{h}$ such that for some $\beta>0$, the estimate
\begin{equation}
\tag{\bf N2}
\|Gv-Gw\|_{\L{h}}\le N_2 \|v-w\|_{\L{p}} \left(  \|v\|^\beta_{\L{p}}+\|w\|^\beta_{\L{p}}\right), \quad G(0)=0 
\end{equation}
holds with $1\le h=p/(1+\beta)$ and $N_2$ depending only on $p$, for all $p\in (1,\infty)$.

With these assumptions, the next Theorem follows directly from \cite[Theorem 1 and Theorem 2]{Giga86}.

\begin{theorem}
\label{S8T1}
For $\vu_0 \in \bm{L}^r_{\sigma,\vt}(\Omega)$ and $\alpha\in \Wfracb{1-\frac{1}{r}}{r}, r\geq 3$, $\alpha \ge 0$, there exists $T_0 > 0$ and a unique solution $\vu(t)$ of \eqref{NS} on $[0,T_0)$ such that
\begin{equation}
\label{S8E0}
\vu \in C([0,T_0);\bm{L}^r_{\sigma,\vt}(\Omega)) \cap L^q(0,T_0; \bm{L}^p_{\sigma,\vt}(\Omega))
\end{equation}
\begin{equation*}
t^{1/q}\vu \in C([0,T_0);\bm{L}^p_{\sigma,\vt}(\Omega)) \quad \text{ and } \quad t^{1/q}\|\vu\|_{\vL{p}} \rightarrow 0 \text{ as } t\rightarrow 0
\end{equation*}
with $\frac{2}{q}= \frac{3}{r}-\frac{3}{p}, p,q >r$. Moreover, there exists a constant $\varepsilon >0$ such that if $\|\vu_0\|_{\vL{r}} < \varepsilon$, then $T_0$ can be taken as infinity for $r=3$.

Let $(0,T_\star)$ be the maximal interval such that $\vu$ solves \eqref{NS} in $C((0,T_\star);\bm{L}^r_{\sigma,\vt}(\Omega)), r>3$. Then
\begin{equation*}
\|u(t)\|_{\vL{r}} \geq C (T_\star -t)^{(3-r)/2r}
\end{equation*}
where $C$ is independent of $T_\star$ and $t$.
\end{theorem}

\begin{proof}
As our Stokes operator has all the same properties and estimates satisfied by the Stokes operator with Dirichlet boundary condition, we are exactly in the same set up as in \cite{Giga86} and hence the proof goes similar to that. However, we briefly review it for completeness.

Let $E^p$ be $\bm{L}^p_{\sigma,\vt}(\Omega)$ and $P: \vL{p}\rightarrow \bm{L}^p_{\sigma, \vt}(\Omega)$ be the Helmholtz projection, defined in (\ref{pression}). It is trivial to see that $P$ is independent of $p\in (1,\infty)$ on $C_c(\Omega)$ and $C_c(\Omega)\cap \bm{L}^p_{\sigma,\vt}(\Omega)$ is dense in $\bm{L}^p_{\sigma,\vt}(\Omega)$. The Stokes operator $A_{p,\alpha}$ on $\bm{L}^p_{\sigma,\vt}(\Omega)$ is defined in (\ref{S3E11})-(\ref{S3E12}) with dense domain and $-A_{p,\alpha}$ generates bounded analytic semigroup on $\bm{L}^p_{\sigma,\vt}(\Omega)$ for all $p\in (1,\infty)$ also (cf. Theorem \ref{24}). Applying $P$ on both sides of the Navier-Stokes system (\ref{NS}) gives
	$$
	\vu_t + A_{p,\alpha} \vu = -P(\vu\cdot \nabla )\vu, \qquad \vu(0) = \vu_0
	$$
which is obviously in the form (\ref{S8E1}) with $F u =-P(\vu\cdot \nabla )\vu $.
	
We now need to verify the assumptions ({\bf A}), ({\bf N1}) and ({\bf N2}). Since $-A_{p,\alpha}$ generates a bounded analytic semigroup with bounded inverse, we have (cf. \cite[Chapter 2, Theorem 6.13]{Pazy})
$$
\forall \bm{f}\in \bm{L}^s_{\sigma,\vt}(\Omega), \qquad \|A_{s,\alpha}^\sigma e^{-tA_{s,\alpha}}\bm{f}\|_{\vL{s}}\le M \|\bm{f}\|_{\vL{s}} /t^\sigma .
$$
As $\bm{D}(A_{s,\alpha}^\sigma)$ is continuously embedded in $\bm{W}^{2\sigma,s}(\Omega)$, this together with the Sobolev embedding theorem yields ({\bf A}) with $m=2, n=3$.

Next we want to write the nonlinear term $Fu $. Since $\div \ \vu = 0$, we have $(\vu\cdot \nabla )\vu_i = \sum_{j=1}^3 \nabla_j (u_j u_i)$. If we define $g: \R^3 \rightarrow \R^{9}$ by
$$
(g(x))_{ij} = -x_i x_j
$$
and $G:\bm{L}^p_{\sigma,\vt}(\Omega) \rightarrow (\L{p})^{9}$ by
$$
G \vu(x) = g(\vu(x))
$$
and $L : (\L{p})^9 \rightarrow \bm{L}^q_{\sigma,\vt}(\Omega)$ by
$$
L g_{ij} = \sum_{j=1}^3 P \nabla_j g_{ij}
$$
which is a linear operator, it implies $F u = L G \vu$. Also it is easy to see from H\"{o}lder inequality that
$$
|g(y) - g(z)| \le N_2 |y-z| (|y| + |z|), \quad g(0)=0
$$
which gives in turn ({\bf N2}) with $\beta =1$.

Finally
$$
\|e^{-tA_{p,\alpha}}L\bm{f}\|_{\vL{p}} = \|A_{p,\alpha}^{1/2} e^{-tA_p}A_{p,\alpha}^{-1/2}L\bm{f}\|_{\vL{p}} \le \frac{C}{t^{1/2}} \|A_{p,\alpha}^{-1/2}L \bm{f}\|_{\vL{p}}
$$
and since $A_{p,\alpha}^{-1/2}L$ is bounded in $\vL{p}$ (cf. \cite[Lemma 2.1]{giga-miya}), the assumption ({\bf N1}) is verified for $\gamma =1$. This completes the proof.
\hfill
\end{proof}

Next we show that the solution of (\ref{NS}) given in the above Theorem in the integral form is actually regular enough and satisfies (\ref{NS}).
\begin{theorem}
\label{S8T2}
Let $\vu_0\in\bm{L}^r_{\sigma,\vt}(\Omega), r\geq 3$ and $\vu(t)$ be the unique solution of (\ref{NS}) given by Theorem \ref{S8T1}. Then
\begin{equation*}
\vu\in C((0,T_*], {\bf D}(A_{p,\alpha})) \cap C^1((0,T_*];\bm{L}^r_{\sigma,\vt}(\Omega) )
\end{equation*}
\end{theorem}

\begin{proof}
As in the previous Theorem, the proof follows exactly the same way as in the case of the Dirichlet boundary condition in \cite{giga-miya}. Since the Stokes operator with Navier boundary condition, defined in (\ref{S3E11})-(\ref{S3E12}), have all the same properties as for the Stokes operator with Dirichlet boundary condition, \cite[Theorem 2.5]{giga-miya} gives (with $f=0$) that $\vu\in C((0,T_*], {\bf D}(A_{p,\alpha}))$. And $\vu\in C^1((0,T_*];\bm{L}^r_{\sigma,\vt}(\Omega) )$ follows from \cite[Lemma 2.14]{kato} (with $f=0$).
\hfill
\end{proof}

Next we show that regular solutions satisfy energy inequality provided the initial condition is in $ \bm{L}^2_{\sigma,\vt}(\Omega)$.

\begin{proposition}
\label{energy_est_NS}
Let $\vu$ be the regular solution of (\ref{NS}) on $ (0,T_0) (T_0<\infty)$ satisfying (\ref{S8E0}). Suppose $\vu_0\in \bm{L}^2_{\sigma,\vt}(\Omega)$. Then
\begin{equation*}
\vu\in L^\infty(0,T_0; \bm{L}^2_{\sigma,\vt}(\Omega) ) \cap L^2(0,T_0;\vH{1})
\end{equation*}
and satisfies the energy equality
\begin{equation*}
	\frac{1}{2}\int\displaylimits_{ \Omega  }{|\vu(t)|^2} + 2 \int\displaylimits_0^t \int\displaylimits_{ \Omega  }{|\DT \vu|^2} + \int\displaylimits_0^t \int\displaylimits_{ \Gamma  }{\alpha |\vu_{\vt}|^2}= \frac{1}{2}\int\displaylimits_{ \Omega  }{| \vu_0|^2}.
\end{equation*}
\end{proposition}

\begin{proof}
The proof follows the same reasoning as in \cite[Proposition 1, Section 5]{Giga86}.
\hfill
\end{proof}

\section{The limit as $\alpha \to \infty$}
\label{9}
\setcounter{equation}{0}

\noindent
Let us denote now  $u _{ \alpha  }$ the solutions  of the unsteady Stokes or Navier-Stokes equation with NBC for a given slip coefficient $\alpha\ge 0 $ and a fixed initial data $u_0$.  A very formal argument suggests that when $\alpha \to \infty$, we may expect that $u _{ \alpha  }\to u _{ \infty }$ in some sense, where $u _{ \infty }$ is the solution of the same equation, with the same initial data, but with Dirichlet boundary condition. The existence of solutions $u_\infty$ for such a problem, for a suitable set of initial data has been proved  for example in \cite[Theorem 1.1, Chapter III]{Temam} for the Stokes equation and \cite[Theorem 3.1, Chapter III]{Temam} for the Navier Stokes equation.

This question has already been  considered in \cite{kelliher} for $\Omega $ a two dimensional domain and $\frac{1}{\alpha} \in L^\infty (\Gamma )$. The author proves in Theorem 9.2 that when $\|\frac{1}{\alpha}\| _{ \Lb{\infty}}\rightarrow 0$ and $\vu_0\in\vH{3}\cap \bm{H}^1_{0}(\Omega)\cap \bm{L}^2_{\sigma,\vt}(\Omega)$, the solution of problem (\ref{lens0})-(\ref{lens5}) converges to the solution of the Navier-Stokes problem with Dirichlet boundary condition in $L^\infty (0,T; L^2(\Omega ))\cap L^2(0,T; \dot{H}^1(\Omega ))\cap L^2(0, T; L^2(\Gamma ))$ (\cite[Theorem 9.2]{kelliher}).  

Our results on this problem  are based on the uniform estimates of  the solutions with NBC with respect to the parameter $\alpha$ proved in the previous Sections. Then, we may only consider the case where the function $\alpha$ is a non negative constant. But, on the other hand, our convergence result in the Hilbert case, with the same rate of convergence as in  \cite{kelliher} only needs the initial data to satisfy $\vu_0\in \bm{L}^2_{\sigma,\vt}(\Omega)$ and we also obtain convergence results for the non Hilbert cases.

In the first two results of this Section, we prove that when $\alpha $ is a constant and the initial data is such that $\vu_0\in\mathbf{D}(A_{p, \alpha })$ for all $\alpha $ sufficiently large, the solutions of the Stokes equation  with Navier boundary conditions (\ref{evolutionary Stokes}) converge in the energy space to the solutions of the Stokes equation with Dirichlet boundary condition obtained in \cite{giga83}. Moreover, we also obtain estimates on the rates of convergence.

\begin{theorem}
\label{S9T0}
Let $\vu_0 \in \bm{L}^2_{\sigma,\vt}(\Omega)$, $\alpha$ be a  constant and $T_\alpha(t):\bm{L}^p_{\sigma,\vt}(\Omega) \rightarrow \bm{L}^p_{\sigma,\vt}(\Omega) $ the semigroup generated by the Stokes operator $A_{p, \alpha}$, defined in (\ref{S3E1})-(\ref{S3E10}). Then for any $T<\infty$,
\begin{equation}
\label{54}
T_\alpha (t)\vu_0 \rightarrow T_\infty (t)\vu_0 \quad \text{ in } \quad L^2(0,T;\vH{1}) \quad \text{ as } \quad \alpha\to \infty
\end{equation}
where $ T_\infty (t)$ is the semigroup generated by the Stokes operator with Dirichlet boundary condition \cite{giga83}. Also we have
\begin{equation}
\label{50}
\int\displaylimits_{0}^{T}\int\displaylimits_{ \Gamma  }{ |T_\alpha (t)\vu_0 -T_\infty (t)\vu_0 |^2}\le \frac{C}{\alpha}.
\end{equation}

Moreover, if $\vu_0 \in \bm{H}^1_0(\Omega)$ with $\div \ \vu_0=0$ in $\Omega$, we further obtain
\begin{equation}
\label{52}
\int\displaylimits_{0}^{T}\int\displaylimits_{\Omega}{|\DT(T_\alpha (t)\vu_0 -T_\infty (t)\vu_0)|^2} + \int\displaylimits_{0}^{T}\int\displaylimits_{ \Gamma  }{ |T_\alpha (t)\vu_0 -T_\infty (t)\vu_0 |^2}\le \frac{C}{\alpha}.
\end{equation}
\end{theorem}

\begin{proof}
Let us denote $\vu_\alpha := T_\alpha (t)\vu_0$. Then $ \vu_{\alpha}$ is the solution of the Stokes problem with Navier boundary condition (\ref{evolutionary Stokes}), given by Theorem \ref{35}
where $ \pi_\alpha$ is the associated pressure.	So $\vu_\alpha$ satisfies the following energy equality
	\begin{equation*}
	\frac{1}{2}\int\displaylimits_{\Omega}{|\vu_\alpha(T)|^2} + 2 \int\displaylimits_{0}^{T} \int\displaylimits_{\Omega}{|\DT\vu_\alpha|^2} + \alpha \int\displaylimits_{0}^{T}\int\displaylimits_{\Gamma} {|\vu_{\alpha\vt}|^2} = \frac{1}{2}\int\displaylimits_{\Omega}{|\vu_0|^2}
	\end{equation*}
	which shows that as $\alpha \rightarrow \infty$,
	\begin{equation*}
	\DT\vu_\alpha \text{ is bounded in } L^2(0,T;\vL{2})
	\end{equation*}
	and
	\begin{equation*}
	\vu_{\alpha\vt} \text{ is bounded in } L^2(0,T;\vLb{2}).
	\end{equation*}
	Therefore, $\vu_\alpha$ is bounded in $ L^2(0,T;\vH{1})$. Hence $ \pi_\alpha$ is also bounded in $ L^2(0,T;\vL{2})$ since $ \|B_2\bm{v}\|_{[\bm{H}_0^{2}(\div ,\Omega)]^\prime}\simeq \|\bm{v}\|_{\vH{1}}$ for all $\bm{v}\in \mathbf{D}(B_2)$. This deduces that $ \frac{\partial \vu_\alpha}{\partial t}$ is as well bounded in $ L^2(0,T;\vH{-1})$. So there exists $ (\vu_\infty, \pi_\infty)\in L^2(0,T;\vH{1})\times L^2(0,T;\L{2})$ with $\frac{\partial \vu_{\infty}}{\partial t}\in L^2(0,T;\bm{H}^{-1}(\Omega))$ such that up to a subsequence,
	\begin{equation*}
	(\vu_\alpha, \pi_\alpha) \rightharpoonup (\vu_\infty, \pi_\infty) \text{ weakly in } L^2(0,T;\vH{1})\times L^2(0,T;\L{2}) \quad \text{ as } \alpha \rightarrow \infty
	\end{equation*}
	and
	\begin{equation*}
	\frac{\partial \vu_\alpha}{\partial t}\rightharpoonup \frac{\partial \vu_{\infty}}{\partial t} \text{ weakly in } L^2(0,T;\bm{H}^{-1}(\Omega)). 
	\end{equation*}
	Also, by Aubin-Lions Lemma, we have
	\begin{equation*}
	\vu_\alpha\to \vu_\infty \quad \text{ in } \quad L^2(0,T;\vL{6-\varepsilon}) \text{ for any } \varepsilon >0.
	\end{equation*}
	
Next we claim that $(\vu_\infty, \pi_\infty)$ satisfies the following Dirichlet problem
\begin{equation}
\label{19}
\left\{
\begin{aligned}
&\frac{\partial \vu_\infty}{\partial t} -\Delta \vu_\infty + \nabla \pi_\infty = \bm{0} , \quad \div \ \vu_\infty = 0 \quad && \text{ in } \Omega\times (0,T) \\
&\vu_\infty = \bm{0} \quad && \text{ on } \Gamma\times (0,T) \\
&\vu_\infty(0) = \vu_0 \quad && \text{ in } \Omega. 
\end{aligned}
\right.
\end{equation}
Indeed, for any $\bm{v}\in C^1([0,T];\bm{H}^1_{0,\sigma}(\Omega))$ where we denote $\bm{H}^1_{0,\sigma}(\Omega) = \bm{H}^1_0(\Omega)\cap \bm{L}^2_{\sigma,\vt}(\Omega)$, the weak formulation satisfied by $ (\vu_\alpha, \pi_\alpha)$ is,
\begin{equation}
\label{4.}
\int\displaylimits_0^T\left\langle \frac{\partial \vu_{\alpha}}{\partial t}, \bm{v}\right\rangle _{\bm{H}^{-1}(\Omega)\times \bm{H}^1_0(\Omega)} + 2\int\displaylimits_0^T\int\displaylimits_{ \Omega} {\DT\vu_{\alpha}:\DT\bm{v}}=0.
\end{equation}
Then passing limit as $\alpha \to \infty$, we obtain
\begin{equation}
\label{5.}
\int\displaylimits_0^T\left\langle \frac{\partial \vu_{\infty}}{\partial t}, \bm{v}\right\rangle _{\bm{H}^{-1}(\Omega)\times \bm{H}^1_0(\Omega)} + 2\int\displaylimits_0^T\int\displaylimits_{ \Omega} {\DT\vu_{\infty}:\DT\bm{v}} = 0.
\end{equation} 
Hence,
\begin{equation*}
\left\langle \frac{\partial \vu_{\alpha}}{\partial t}, \bm{v}\right\rangle _{\bm{H}^{-1}(\Omega)\times \bm{H}^1_0(\Omega)} + 2\int\displaylimits_{ \Omega} {\DT\vu_{\infty}:\DT\bm{v}}= 0
\end{equation*}
for any $\bm{v}\in\bm{H}^1_{0,\sigma}(\Omega)$ and a.e. $0\le t\le T$. Also we have, $\vu_{\infty}\in C([0,T];\vL{2})$.

In order to show that $\vu_\infty(0) =\vu_0$, we can write from (\ref{4.}), for any $\bm{v}\in C^1([0,T];\bm{H}^1_{0,\sigma}(\Omega))$ with $\bm{v}(T)=0$,
\begin{equation*}
-\int\displaylimits_0^T\int\displaylimits_{ \Omega  }{\vu_{\alpha}\cdot\frac{\partial \bm{v}}{\partial t}} + 2\int\displaylimits_0^T\int\displaylimits_{ \Omega} {\DT\vu_{\alpha}:\DT\bm{v}} = \int\displaylimits_{ \Omega}{\vu_0\cdot \bm{v}(0)}
\end{equation*}
and similarly from (\ref{5.}),
\begin{equation*}
-\int\displaylimits_0^T\int\displaylimits_{ \Omega  }{\vu_{\infty}\cdot\frac{\partial \bm{v}}{\partial t}} + 2\int\displaylimits_0^T\int\displaylimits_{ \Omega} {\DT\vu_{\infty}:\DT\bm{v}} = \int\displaylimits_{ \Omega}{\vu_\infty(0)\cdot \bm{v}(0)}.
\end{equation*}
As $\bm{v}(0)$ is arbitrary, we thus conclude that $\vu_\infty(0) =\vu_0$.

{\bf ii)} It remains to prove the strong convergence of $(\vu_{\alpha},\pi_\alpha)$ to $(\vu_{\infty},\pi_\infty)$. Note that $(\bm{v}_\alpha,p_\alpha ) := (\vu_{\alpha} - \vu_{\infty},\pi_\alpha - \pi_\infty )$ satisfies the following problem
\begin{equation}
\label{20}
\left\{
\begin{aligned}
&\frac{\partial \bm{v}_\alpha}{\partial t} -\Delta \bm{v}_\alpha + \nabla p_\alpha = \bm{0} , \quad \div \ \bm{v}_\alpha = 0 \quad && \text{ in } \Omega\times (0,T) \\
&\bm{v}_\alpha\cdot \vn = 0, \quad 2[(\DT\bm{v}_{\alpha})\vn]_{\vt}+ \alpha\bm{v}_{\alpha\vt}= -2 [(\DT\vu_{\infty})\vn]_{\vt}
\quad && \text{ on } \Gamma\times (0,T) \\
&\bm{v}_\alpha(0) = \bm{0} \quad && \text{ in } \Omega .
\end{aligned}
\right.
\end{equation}
Multiplying the above system by $\bm{v}_\alpha$, we obtain the following energy estimate
\begin{equation}
\label{55}
\frac{1}{2}\int\displaylimits_{ \Omega  }{|\bm{v}_\alpha(t)|^2} + 2\int\displaylimits_{0}^{T}\int\displaylimits_{ \Omega  }{|\DT \bm{v}_\alpha|^2} + \int\displaylimits_{0}^{T}\int\displaylimits_{ \Gamma  }{\alpha |\bm{v}_{\alpha\vt}|^2} = 2 \int\displaylimits_{0}^{T}\left\langle [(\DT\vu_{\infty})\vn]_{\vt}, \bm{v}_{\alpha\vt}\right\rangle _{\Gamma}
\end{equation}
which shows that $ \bm{v}_\alpha \to \bm{0}$ in $L^2(0,T;\vLb{2})$ as $\alpha \to \infty$ and thus (\ref{50}) follows. Also since $\bm{v}_{\alpha \vt} \rightharpoonup \bm{0} $ weakly in $ L^2(0,T;\vHfracb{1}{2})$ and $ [(\DT\vu_{\infty})\vn]_{\vt}\in L^2(0,T;\vHfracbd{1}{2})$, the right hand side in the above relation goes to $0$ as $\alpha\rightarrow \infty$. So $\DT\bm{v}_\alpha\to 0$ in $L^2(0,T;\vL{2})$. This proves (\ref{54}).

{\bf iii)} If we assume $\vu_0\in \bm{H}^1_0(\Omega)\cap \bm{L}^2_{\sigma,\vt}(\Omega)$, then $\vu_\infty\in L^2(0,T;\vH{2})$ (cf. \cite[Proposition 1.2, Chapter III]{Temam}) and thus the energy equality (\ref{55}) can be estimated as
\begin{equation*}
\int\displaylimits_{0}^{T}\int\displaylimits_{ \Omega  }{|\DT \bm{v}_\alpha|^2} \le \int\displaylimits_{0}^{T}\int\displaylimits_{\Gamma}{[(\DT\vu_{\infty})\vn]_{\vt}\cdot \bm{v}_{\alpha\vt}} \le  \|\vu_{\infty}\|_{L^2(0,T;\vH{2})} \|\bm{v}_\alpha\|_{L^2(0,T;\vLb{2})}. 
\end{equation*}
This along with the estimate (\ref{50}) gives (\ref{52}).
\hfill
\end{proof}

In order to prove the convergence result for any $p\in (1,\infty)$, we need to use some compactness argument. But due to unavailability of the energy estimate for general $p\ne 2$, some more regularity of the solution of (\ref{evolutionary Stokes}), hence more regular initial data is required.

\begin{theorem}
\label{S9T1}
Let $ p\in (1,\infty)$, $T_\alpha(t), T_\infty (t) $ be defined as in Theorem \ref{S9T0} and $\vu_0\in \bm{W}^{2,p}_0(\Omega)\cap \bm{L}^p_{\sigma,\vt}(\Omega)$. Then for any $T<\infty$, as $\alpha \rightarrow \infty$,
\begin{equation}
\label{16.}
T_\alpha (t)\vu_0 \rightarrow T_\infty (t)\vu_0 \quad \text{ in } \quad C([0,T];\mathbf{D}(A^{1/2}_{p, \alpha}))
\end{equation}
and
\begin{equation}
\label{16}
T_\alpha (t)\vu_0 \rightarrow T_\infty (t)\vu_0 \quad \text{ in } \quad C^k(0,T;\mathbf{D}(A^l_{p, \alpha}))\,\,\,\ \forall k\in \mathbb{N}, \,\,\,\forall l\in \mathbb{N}\setminus \{0\}.
\end{equation}
Also if we assume $\vu_0\in \mathbf{D}(A^2_{p, \alpha })$, then the following convergence rate can be obtained, for any $1<q<\infty$, 
\begin{equation*}
\|\frac{\partial }{\partial t}\left( T_\alpha (t)\vu_0 - T_\infty (t)\vu_0\right) \|_{L^q(0,T;\vL{p})}+\|T_\alpha (t)\vu_0 - T_\infty (t)\vu_0\|_{L^q(0,T;\vW{2}{p})} \le \frac{C}{\alpha} .
\end{equation*}
\end{theorem}

\begin{proof}
	First we claim the following set theoretic equality
	\begin{equation*}
	\bm{W}^{2,p}_0(\Omega)\cap \bm{L}^p_{\sigma,\vt}(\Omega) = \underset{\alpha\ge 1}{\cap} \mathbf{D}(A_{p, \alpha }).
	\end{equation*}
	Indeed, it is obvious to see that
	$$
	\underset{\alpha\ge 1}{\cap} \mathbf{D}(A_{p, \alpha }) = \{\vu \in \bm{W}^{2,p}(\Omega)\cap \bm{W}^{1,p}_0(\Omega)\cap \bm{L}^p_{\sigma,\vt}(\Omega): [(\DT\vu)\vn]_{\vt} = \bm{0} \text{ on } \Gamma\}.
	$$
	Let us denote the set on the right hand side of the above relation by $E$. It is then enough to show $E\subseteq \bm{W}^{2,p}_0(\Omega)\cap \bm{L}^p_{\sigma,\vt}(\Omega)$. To simplify, assuming $\Omega = \R^3_+$, we get that if $\vu = \bm{0}$ on $\Gamma$, then $[(\DT\vu)\vn]_{\vt} = \bm{0}$ iff $(\frac{\partial \vu_1}{\partial x_3}, \frac{\partial \vu_2}{\partial x_3}, 0) = \bm{0}$. Taking into account the fact that $\div \ \vu =0$ in $\Omega$, this implies $\frac{\partial \vu}{\partial \vn} =\bm{0}$ on $\Gamma$. Hence the claim.
	
{\bf i)} Since $\vu_0\in \mathbf{D}(A_{p, \alpha })$ for every $\alpha $ large enough,
$$
\vu_{\alpha}\in C([0,T];\bm{W}^{2,p}_{\sigma,\vt}(\Omega))\cap C^1([0,T];\bm{L}^p_{\sigma,\vt}(\Omega)) \text{ is bounded as } \alpha\rightarrow \infty.
$$
So by Aubin-Lions compactness result, there exists $\vu_\infty\in C[0,T;\bm{W}^{1,p}_{\sigma,\vt}(\Omega))$ such that up to a subsequence,
$$
\vu_{\alpha} \rightarrow \vu_{\infty}\text{ in } C([0,T];\bm{W}^{1,p}_{\sigma,\vt}(\Omega)).
$$
Also by De Rham's theorem, there exists $\pi_\infty \in C([0,T];L^{p}(\Omega))$ such that  $\pi_\alpha \rightarrow \pi_\infty$ in $C([0,T];L^{p}(\Omega))$. In fact, as $ \vu_{\alpha}\in C^k(0,T;\mathbf{D}(A^l_p))$ is bounded uniformly for any $k\in \mathbb{N}$ and $l\in \mathbb{N}\backslash\{0\}$, we get that $\vu_\infty \in C^k(0,T;\mathbf{D}(A^l_p))$. We now claim that $(\vu_\infty,\pi_\infty)$ satisfies the Stokes equation with Dirichlet boundary condition. Indeed if we write the system (\ref{evolutionary Stokes}) as
\begin{equation}
\label{modified_S}
\left\{
\begin{aligned}
\frac{\partial \vu_\alpha}{\partial t} -\Delta \vu_\alpha + \nabla \pi_\alpha = \bm{0} , \quad \div \ \vu_\alpha = 0 \quad & \text{ in } \Omega\times (0,T) \\
\vu_\alpha\cdot \vn = 0, \quad  \vu_{\alpha\vt} = -\frac{2}{\alpha}[(\DT\vu_\alpha)\vn]_{\vt} \quad & \text{ on } \Gamma\times (0,T) \\
\vu_\alpha(0) = \vu_0 \quad & \text{ in } \Omega
\end{aligned}
\right.
\end{equation}
passing the limit $\alpha\rightarrow \infty$, we obtain that $\vu_{\infty}$ satisfies the system (\ref{19})
which is given by $T_\infty(t)\vu_0$, with associated pressure $\pi_\infty$. Note that we use $\vu_{\alpha}$ is continuous up to $t=0$ to show that $\vu_\infty$ satisfies the initial data. Thus, by the hypothesis on $\vu_0$, the convergence result (\ref{16.}) and (\ref{16}) follow.

{\bf ii)} Next to deduce the rate of convergence, consider the difference between the systems (\ref{modified_S}) and (\ref{19}). The system satisfies by $(\bm{v}_\alpha,p_\alpha):= (\vu_{\alpha} -\vu_\infty, \pi_\alpha - \pi_\infty )$ is
\begin{equation}
\label{6}
\left\{
\begin{aligned}
\frac{\partial \bm{v}_\alpha}{\partial t} -\Delta \bm{v}_\alpha + \nabla p_\alpha = \bm{0} , \quad \div \ \bm{v}_\alpha = 0 \quad & \text{ in } \Omega\times (0,T) \\
\bm{v}_\alpha\cdot \vn = 0, \quad  \bm{v}_{\alpha\vt} = -\frac{2}{\alpha}[(\DT\vu_\alpha)\vn]_{\vt} \quad & \text{ on } \Gamma\times (0,T) \\
\vu_\alpha(0) = \bm{0}\quad & \text{ in } \Omega.
\end{aligned}
\right.
\end{equation}
We then reduce the system (\ref{6}) to a problem with homogeneous boundary data to apply the known estimates for Stokes problem for example, \cite[Theorem 2.8]{Giga91}. For any $\alpha>0$ and $t\ge 0$, let $(\bm{w}_\alpha(t),z_\alpha(t))$ satisfies the system
\begin{equation*}
\begin{aligned}
-\Delta \bm{w}_\alpha(t) + \nabla z_\alpha (t)=\bm{0}, \quad \div \ \bm{w}_\alpha(t) = 0, \quad
\bm{w}_\alpha(t)\arrowvert _\Gamma = \vu_{\alpha}\arrowvert _\Gamma(t).
\end{aligned}
\end{equation*}
Since we have assumed $\vu_0\in \mathbf{D}(A^2_{p, \alpha })$, we get the improved regularity \cite[Theorem 4.4.7, Chapter 4]{kesavan} $\vu_{\alpha}\in C^1([0,\infty);\bm{W}^{2,p}_{\sigma,\vt}(\Omega))$ 
which gives $\frac{\partial \vu_{\alpha}}{\partial t}\arrowvert _\Gamma(t) \in \vW{2-\frac{1}{p}}{p} $ for all $t\ge 0$ . Thus the regularity result for Stationary Stokes system with Dirichlet boundary condition \cite{cattabriga} yields that $\frac{\partial }{\partial t}\bm{w}_\alpha(t)\in \vW{2}{p}$ for all $t\ge 0$. We also obtain the estimate
\begin{equation}
\label{48}
\|\bm{w}_\alpha\|_{C^1([0,T];\bm{W}^{2,p}(\Omega))} \le C \|\vu_{\alpha}\|_{C^1([0,T];\vWfracb{2-\frac{1}{p}}{p})}\le \frac{C}{\alpha} \| [(\DT\vu_{\alpha})\vn]_{\vt}\|_{C^1([0,T];\vWfracb{2-\frac{1}{p}}{p})}.
\end{equation}
Note that the above continuity constant $C$ is independent of $\alpha$.
Then the substitution $ V_\alpha := \bm{v}_\alpha - \bm{w}_\alpha$ reduces the system (\ref{6}) to,
\begin{equation}
\label{42}
\left\{
\begin{aligned}
&\frac{\partial V_\alpha}{\partial t} -\Delta V_\alpha + \nabla p_\alpha =-\frac{\partial \bm{w}_\alpha}{\partial t}, \quad \div \ V_\alpha = 0 \quad && \text{ in } \Omega\times (0,T) \\
&V_\alpha = \bm{0} \quad && \text{ on } \Gamma\times (0,T) \\
&V_\alpha(0) = -\bm{w}_\alpha(0) \quad && \text{ in } \Omega .
\end{aligned}
\right.
\end{equation}
Hence the maximal regularity result \cite[Theorem 2.8]{Giga91} leads us to,
\begin{equation*}
V_\alpha \in L^q(0,T; \bm{W}^{2,p}_0(\Omega) )\,\,\, \text{ with } \,\,\, \frac{\partial V_\alpha}{\partial t} \in L^q(0,T;\bm{L}^p_\sigma (\Omega)) \,\,\,\,\text{ for any } 1<q<\infty
\end{equation*}
and the estimate
\begin{equation*}
\begin{aligned}
\int\displaylimits_{0}^{T}\left\| \frac{\partial V_\alpha}{\partial t} \right\| ^q_{\bm{L}^p(\Omega)} + \int\displaylimits_{0}^{T}\|-\Delta V_\alpha + \nabla p_\alpha\|^q_{\bm{L}^p (\Omega)}\le C \left( \int\displaylimits_{0}^{T}\|\frac{\partial \bm{w}_\alpha}{\partial t}\|^q_{\bm{L}^p_\sigma (\Omega)} + \|\bm{w}_\alpha(0)\|_{\vW{2}{p}}\right) 
\end{aligned}
\end{equation*}
with $C = C(\Omega,p,q)>0$ independent of $\alpha$.
Therefore, together with (\ref{48}), we obtain
\begin{equation*}
\begin{aligned}
\int\displaylimits_{0}^{T}\left\| \frac{\partial \bm{v}_\alpha}{\partial t} \right\| ^q_{\bm{L}^p (\Omega)} + \int\displaylimits_{0}^{T}\|-\Delta \bm{v}_\alpha + \nabla p_\alpha\|^q_{\bm{L}^p(\Omega)} \le \frac{C}{\alpha} 
\end{aligned}
\end{equation*}
since $\| [(\DT\vu_{\alpha})\vn]_{\vt}\|_{C^1([0,T];\vWfracb{2-\frac{1}{p}}{p})}$ is bounded for all $\alpha$ large. This concludes the proof.
\hfill
\end{proof}

We prove now our two results on the convergence as $\alpha \to \infty$ of the solutions to the Navier Stokes equation and begin with the case where $\vu_0$ belongs to an $L^2$-type space.

\begin{proof}[\bf Proof of Theorem \ref{S9T2}]
We proceed as in the linear case.
	
{\bf i)} From Proposition \ref{energy_est_NS} we can see that as $\alpha \rightarrow \infty$,
\begin{equation*}
\DT\vu_\alpha \text{ is bounded in } L^2(0,T;\vL{2})
\end{equation*}
and
\begin{equation*}
\vu_{\alpha\vt} \text{ is bounded in } L^2(0,T;\vLb{2}).
\end{equation*}
Therefore, $\vu_\alpha$ is bounded in $ L^2(0,T;\vH{1})$. 
And since $ \|B_{2,\alpha}{\bm v}\|_{[{\bm H}_0^2(\div,\Omega)]'} \simeq \|\bm{v}\|_{\vH{1}}$ for any $ \bm{v}\in {\mathbf{D}(B_{2,\alpha})}$, $ \pi_\alpha$ is as well bounded in $ L^2(0,T;\vL{2})$. This implies $ \frac{\partial \vu_\alpha}{\partial t}$ is also bounded in $ L^2(0,T;\vH{-1})$. Hence, there exists $ (\vu_\infty, \pi_\infty)\in L^2(0,T;\vH{1})\times L^2(0,T;\L{2})$ with $\frac{\partial \vu_{\infty}}{\partial t}\in L^2(0,T;\bm{H}^{-1}(\Omega))$ such that up to a subsequence,
\begin{equation*}
(\vu_\alpha, \pi_\alpha) \rightharpoonup (\vu_\infty, \pi_\infty) \quad \text{ weakly in } \quad L^2(0,T;\vH{1})\times L^2(0,T;\L{2}) \quad \text{ as } \alpha \rightarrow \infty
\end{equation*}
and
\begin{equation*}
\frac{\partial \vu_\alpha}{\partial t}\rightharpoonup \frac{\partial \vu_{\infty}}{\partial t}  \quad\text{ weakly in } \quad L^2(0,T;\bm{H}^{-1}(\Omega)).
\end{equation*}
Also, by Aubin-Lions Lemma,
\begin{equation*}
\vu_\alpha \to \vu_\infty \quad \text{ in } \quad L^2(0,T;\vL{6-\varepsilon}) \quad \text{ for any } \quad \varepsilon >0.
\end{equation*}
	
Next we show that $(\vu_\infty, \pi_\infty) $ satisfies the Dirichlet problem (\ref{NSD}). Indeed, for any $\bm{v}\in C^1([0,T];\bm{H}^1_{0,\sigma}(\Omega))$, the weak formulation of the problem (\ref{NS}) is,
\begin{equation}
\label{4}
	\int\displaylimits_0^T\left\langle \frac{\partial \vu_{\alpha}}{\partial t}, \bm{v}\right\rangle _{\vH{-1}\times \bm{H}^1_0(\Omega)} + 2\int\displaylimits_0^T\int\displaylimits_{ \Omega} {\DT\vu_{\alpha}:\DT\bm{v}}+\int\displaylimits_0^T\int\displaylimits_{ \Omega}{(\vu_{\alpha}\cdot \nabla)\vu_\alpha\cdot \bm{v}}=0.
\end{equation}
Then passing limit as $\alpha \to \infty$, we obtain
\begin{equation}
\label{5}
\int\displaylimits_0^T\left\langle \frac{\partial \vu_{\infty}}{\partial t}, \bm{v}\right\rangle _{\vH{-1}\times \bm{H}^1_0(\Omega)} + 2\int\displaylimits_0^T\int\displaylimits_{ \Omega} {\DT\vu_{\infty}:\DT\bm{v}} +\int\displaylimits_0^T\int\displaylimits_{ \Omega}{(\vu_{\infty}\cdot \nabla)\vu_\infty\cdot \bm{v}}= 0.
\end{equation}
To pass to the limit in the non-linear term, we used the standard relation
$$
\int\displaylimits_{ \Omega}{(\vu_{\alpha}\cdot \nabla)\vu_\alpha\cdot \bm{v}} = -\int\displaylimits_{ \Omega}{(\vu_{\alpha}\cdot \nabla)\bm{v}\cdot \vu_\alpha}.
$$ 
Also we have, $\vu_{\infty}\in C([0,T];\vL{2})$.

In order to prove $\vu_\infty(0) =\vu_0$, we write from (\ref{4}), for any $\bm{v}\in C^1([0,T];\bm{H}^1_{0,\sigma}(\Omega))$ with $\bm{v}(T)=0$,
\begin{equation*}
-\int\displaylimits_0^T\int\displaylimits_{ \Omega  }{\vu_{\alpha}\cdot\frac{\partial \bm{v}}{\partial t}} + 2\int\displaylimits_0^T\int\displaylimits_{ \Omega} {\DT\vu_{\alpha}:\DT\bm{v}}+\int\displaylimits_0^T\int\displaylimits_{ \Omega}{(\vu_{\alpha}\cdot \nabla)\vu_\alpha\cdot \bm{v}} = \int\displaylimits_{ \Omega}{\vu_0\cdot \bm{v}(0)}
\end{equation*}
and similarly from (\ref{5}),
\begin{equation*}
-\int\displaylimits_0^T\int\displaylimits_{ \Omega  }{\vu_{\infty}\cdot\frac{\partial \bm{v}}{\partial t}} + 2\int\displaylimits_0^T\int\displaylimits_{ \Omega} {\DT\vu_{\infty}:\DT\bm{v}} +\int\displaylimits_0^T\int\displaylimits_{ \Omega}{(\vu_{\infty}\cdot \nabla)\vu_\infty\cdot \bm{v}}= \int\displaylimits_{ \Omega}{\vu_\infty(0)\cdot \bm{v}(0)}.
\end{equation*}
As $\bm{v}(0)$ is arbitrary, we thus conclude $\vu_\infty(0) =\vu_0$.

{\bf ii)} Finally to show the strong convergence of $(\vu_\alpha, \pi_\alpha) $ to $ (\vu_\infty, \pi_\infty)$, setting $\bm{v}_\alpha = \vu_{\alpha} -\vu_\infty$ and $ p_\alpha = \pi_\alpha - \pi_\infty$, it solve the following problem
\begin{equation*}
\left\{
\begin{aligned}
&\frac{\partial \bm{v}_\alpha}{\partial t} -\Delta \bm{v}_\alpha + (\vu_{\alpha}\cdot \nabla) \vu_{\alpha} - (\vu_{\infty}\cdot \nabla )\vu_{\infty} +\nabla p_\alpha = \bm{0} , \quad \div \ \bm{v}_\alpha = 0 \quad && \text{ in } \Omega\times (0,T) \\
&\bm{v}_\alpha\cdot \vn = 0, \quad 2[(\DT\bm{v}_\alpha)\vn]_{\vt}+\alpha \bm{v}_{\alpha \vt} = -2 [(\DT\vu_{\infty})\vn]_{\vt} \quad && \text{ on } \Gamma\times (0,T) \\
&\bm{v}_\alpha(0) = 0 \quad && \text{ in } \Omega .
\end{aligned}
\right.
\end{equation*}
Multiplying the above system by $ \bm{v}_\alpha$ and integrating by parts over $\Omega \times (0,T)$, we deduce
\begin{equation}
\label{49}
\begin{aligned}
& \quad \frac{1}{2}\|\bm{v}_\alpha(T)\|_{\vL{2}}^2 + 2 \int\displaylimits_0^T \|\DT \bm{v}_\alpha\|^2_{\mathbb{L}^2(\Omega)} + \alpha \int\displaylimits_0^T {\|\bm{v}_{\alpha\vt}\|^2_{\vLb{2}} }\\
& = 
-2\int\displaylimits_{0}^{T}{\left\langle [(\DT \bm{v}_\infty)\vn]_{\vt},\bm{v}_{\alpha\vt}\right\rangle_\Gamma } - \int\displaylimits_{0}^{T}{\left\langle (\vu_{\alpha}\cdot \nabla) \vu_{\alpha} - (\vu_{\infty}\cdot \nabla )\vu_{\infty},\bm{v}_\alpha \right\rangle_\Omega }.
\end{aligned}
\end{equation}
This shows that $\bm{v}_\alpha \to \bm{0}$ in $L^2(0,T;\vLb{2})$ proving (\ref{60}). As $\bm{v}_\alpha \to \bm{0}$ in $L^2(0,T;\vL{4})$,
\begin{equation*}
\begin{aligned}
\int\displaylimits_{0}^{T}{\left\langle (\vu_{\alpha}\cdot \nabla) \vu_{\alpha} - (\vu_{\infty}\cdot \nabla )\vu_{\infty},\bm{v}_\alpha \right\rangle_\Omega } = - \int\displaylimits_{0}^{T}\int\displaylimits_{ \Omega  }{(\bm{v}_\alpha\cdot}\nabla) \vu_{\infty}\cdot \bm{v}_\alpha \le \|\bm{v}_\alpha\|_{\vL{4}}^2 \|\nabla \vu_{\infty}\|_{\vL{2}} \rightarrow 0.
\end{aligned}
\end{equation*}
Also since $ \bm{v}_\alpha\rightharpoonup \bm{0}$ weakly in $ L^2(0,T;\vHfracb{1}{2})$ and $[(\DT\vu_\infty)\vn]_{\vt}\in L^2(0,T;\vHfracbd{1}{2})$, it implies,
\begin{equation*}
\left\langle 2[(\DT \vu_\infty) \vn]_{\vt}, \bm{v}_\alpha \right\rangle _{\vHfracbd{1}{2}\times \vHfracb{1}{2}}\rightarrow 0.
\end{equation*} 
Therefore, $ \DT\bm{v}_\alpha \rightarrow \bm{0}$ in $L^2(0,T;\vL{2})$. The strong convergence for the pressure term  follows from the equation.

{\bf iii)} Now to obtain (\ref{61}), we estimate suitably the right hand side of (\ref{49}). Because $\vu_0\in \bm{H}^1_{0,\sigma}(\Omega)$, we have $\vu_\infty\in L^2(0,T;\vH{2})$ and thus
\begin{equation*}
\int\displaylimits_{0}^{T}{\left\langle [(\DT \bm{v}_\infty)\vn]_{\vt},\bm{v}_{\alpha\vt}\right\rangle_\Gamma } \le C \|\vu_{\infty}\|_{L^2(0,T;\vH{2})}\|\bm{u}_\alpha\|_{L^2(0,T;\vLb{2})}.
\end{equation*}
Similarly $\Gamma$ is $\HC{2}{1}$ and $\vu_0\in \vH{2}\cap \bm{H}^1_{0,\sigma}(\Omega)$ implies the regularity $ \vu_{\infty}\in L^2(0,T;\vH{3})$ by the same argument as \cite[Theorem 3.10, Chapter III]{Temam}. Hence $\nabla \vu_{\infty}\in L^\infty(0,T;\vH{2}) \subset L^\infty(0,T;\vL{\infty})$ and thus
\begin{equation*}
\left| \int\displaylimits_{0}^{T}\int\displaylimits_{ \Omega  }{(\bm{v}_\alpha\cdot}\nabla) \vu_{\infty}\cdot \bm{v}_\alpha\right| \le C \|\bm{v}_\alpha\|^2_{L^2(0,T;\vL{2})}\|\nabla \vu_{\infty}\|_{L^\infty(0,T;\vL{\infty})}.
\end{equation*}
Therefore, combining these estimates, we get
\begin{equation*}
\frac{1}{2}\|\bm{v}_\alpha(T)\|_{\vL{2}}^2 + 2 \int\displaylimits_0^T \|\DT \bm{v}_\alpha\|^2_{\mathbb{L}^2(\Omega)} \le C \|\bm{u}_\alpha\|_{L^2(0,T;\vLb{2})} + C \int\displaylimits_0^T{\|\bm{v}_\alpha}\|^2_{\vL{2}}
\end{equation*}
which yields by applying Gronwall's lemma,
\begin{equation*}
\|\bm{v}_\alpha(T)\|_{\vL{2}}^2 \le C \|\bm{u}_\alpha\|_{L^2(0,T;\vLb{2})}e^{Ct} \le \frac{C}{\alpha}.
\end{equation*}
The convergence in $L^\infty(0,T;\vL{2})$ and thus also in $L^2(0,T;\vH{1})$ follow immediately.
\hfill
\end{proof}

When the initial data belongs to an $L^r$-type space with $r\ge 3$, we have the following:

\begin{theorem}
\label{S9T3}
Let $ (\vu_\alpha, \pi_\alpha)$ be the solution of the problem (\ref{NS}) with $\vu_0\in \bm{W}^{2,r}_0(\Omega)\cap \bm{L}^r_{\sigma,\vt}(\Omega)$ where $r\ge 3$. Suppose also that 
$(\vu_\infty, \pi_\infty)\in C([0,T];\vW{2}{r})\cap C^1([0,T];\vL{r})\times C([0,T];\vW{1}{r}) $  is the solution of  (\ref{NSD}) with the same initial data $\vu_0$ whose existence and uniqueness has been proved in \cite[Theorem 4]{Giga86}. 
Then as $\alpha \rightarrow \infty$, 
\begin{equation}
\label{21.}	
(\vu_\alpha, \pi_\alpha) \rightarrow (\vu_\infty, \pi_\infty) \quad \text{ in } \quad C([0,T);\vW{1}{s}) \times C([0,T);\L{s})
\end{equation}
where $s\in [1, \infty) \text{ if } r=3 \text{ and } s= \infty \text{ if } r>3$.
	
Moreover, if $\vu_0\in \mathbf{D}(A^2_{p, \alpha })$ for $\alpha $ sufficiently large, then we obtain the following rate of convergence, for any $m\in (1,\infty)$,
\begin{equation}
\label{25}
\|\frac{\partial}{\partial t}(\vu_\alpha- \vu_\infty)\|_{L^m(0,T;\vL{s})} + \| \vu_\alpha- \vu_\infty\|_{L^m(0,T;\vW{2}{s})} \le \frac{C}{\alpha}.
\end{equation}
\end{theorem}

\begin{proof}
{\bf i)} As explained in the beginning of the proof of Theorem \ref{S9T1}, $ \vu_0\in\mathbf{D}(A_{r, \alpha })$ for all $\alpha $ large enough and hence
$$
\vu_{\alpha} \;\;\; \text{ is bounded in }\;\;\; C([0,T];\vW{2}{r})\cap C^1([0,T];\vL{r}).
$$
Thus by compactness, there exists $\vu_\infty\in C([0,T];\vW{1}{s})$ with $s$ as defined in the theorem such that, up to a subsequence,
$$
\vu_{\alpha}\rightarrow \vu_\infty \text{ in } C([0,T];\vW{1}{s}).
$$
This implies by De Rham theorem, $\pi_\alpha \rightarrow \pi_\infty$ in $C([0,T];L^{s}(\Omega))$. Now to show that the limit $(\vu_\infty,\pi_\infty)$ actually satisfies the Navier-Stokes problem with homogeneous Dirichlet boundary data, we write the system (\ref{NS}) in the following form
	\begin{equation}
	\label{modifiedNS}
	\left\{
	\begin{aligned}
	\frac{\partial \vu_\alpha}{\partial t} -\Delta \vu_\alpha + (\vu_\alpha\cdot \nabla) \vu_\alpha + \nabla \pi_\alpha = \bm{0} , \quad \div \ \vu_\alpha = 0 \quad & \text{ in } \Omega\times (0,T) \\
	\vu_\alpha\cdot \vn = 0, \quad  \vu_{\alpha\vt} = -\frac{2}{\alpha}[(\DT\vu_\alpha)\vn]_{\vt} \quad & \text{ on } \Gamma\times (0,T) \\
	\vu_\alpha(0) = \vu_0 \quad & \text{ in } \Omega.
	\end{aligned}
	\right.
	\end{equation}
Since $(\vu_{\alpha}\cdot \nabla) \vu_{\alpha} - (\vu_{\infty}\cdot \nabla )\vu_{\infty} = (\vu_{\alpha}-\vu_{\infty})\cdot \nabla \vu_{\alpha} + \vu_\infty\cdot \nabla (\vu_\alpha - \vu_\infty)$, the regularity implies
\begin{equation*}
\begin{aligned}
&\quad \ \|(\vu_{\alpha}\cdot \nabla) \vu_{\alpha} - (\vu_{\infty}\cdot \nabla )\vu_{\infty}\|_{\vL{s}}\\
&\le \|\vu_{\alpha}-\vu_{\infty}\|_{\vL{\infty}}\|\nabla \vu_{\alpha}\|_{\vL{s}} + \|\vu_{\infty}\|_{\vL{\infty}}\|\nabla (\vu_{\alpha}-\vu_{\infty})\|_{\vL{s}}\\
& \le \|\vu_{\alpha}-\vu_{\infty}\|_{\vW{1}{s}} \left( \|\vu_{\alpha}\|_{\vW{1}{s}} + \|\vu_\infty\|_{\vW{1}{s}}\right) 
\end{aligned}
\end{equation*}
which shows that $(\vu_{\alpha}\cdot \nabla) \vu_{\alpha} \rightarrow (\vu_{\infty}\cdot \nabla )\vu_{\infty}$ in $ C([0,T;\vL{s})$ as $\alpha \rightarrow \infty$. Hence, passing limit in the other terms of the above system yields that indeed $\vu_\infty$ is a solution of the problem (\ref{NSD}). Note  that we use the continuity of $\vu_{\alpha}$ up to $t=0$ to obtain $\vu_\infty$ satisfies the initial data.

Next to deduce the rate of convergence, taking difference between the two systems (\ref{modifiedNS}) and (\ref{NSD}) and denoting by $\bm{v}_\alpha = \vu_{\alpha} -\vu_\infty$ and $ p_\alpha = \pi_\alpha - \pi_\infty$, we obtain
\begin{equation*}
\left\{
\begin{aligned}
&\frac{\partial \bm{v}_\alpha}{\partial t} -\Delta \bm{v}_\alpha  +\nabla p_\alpha = (\vu_{\infty}\cdot \nabla )\vu_{\infty} - (\vu_{\alpha}\cdot \nabla) \vu_{\alpha} , \quad \div \ \bm{v}_\alpha = 0 \quad && \text{ in } \Omega\times (0,T) \\
&\bm{v}_\alpha\cdot \vn = 0, \quad \bm{v}_{\alpha \vt} = -\frac{2}{\alpha} [(\DT\vu_{\alpha})\vn]_{\vt} \quad && \text{ on } \Gamma\times (0,T) \\
&\bm{v}_\alpha(0) = 0 \quad && \text{ in } \Omega .
\end{aligned}
\right.
\end{equation*}
But notice that $(\vu_{\infty}\cdot \nabla )\vu_{\infty} - (\vu_{\alpha}\cdot \nabla) \vu_{\alpha}\rightarrow \bm{0}$ in $ L^m(0,T;\vL{s})$ for any $m\in [1,\infty)$ and $ [(\DT \vu_\alpha )\vn ]_{\vt}$ is bounded in $L^m(0,T;\vWfracb{1-\frac{1}{s}}{s})$. Therefore as we have done for the linear problem, using the lift operator and then the maximal regularity for the Stokes system with non-homogeneous initial data \cite[Theorem 2.8]{Giga91} leads us to the following
\begin{equation*}
\begin{aligned}
& \quad \int\displaylimits_{0}^{T}\left\| \frac{\partial \bm{v}_\alpha}{\partial t} \right\| ^m_{\bm{L}^s (\Omega)} + \int\displaylimits_{0}^{T}\|-\Delta \bm{v}_\alpha + \nabla p_\alpha\|^m_{\bm{L}^s(\Omega)}\\
& \le C\left( \int\displaylimits_0^T \|(\vu_{\infty}\cdot \nabla )\vu_{\infty} - (\vu_{\alpha}\cdot \nabla) \vu_{\alpha}\|^m_{\vL{s}} +\frac{1}{\alpha}\| [(\DT\vu_{\alpha})\vn]_{\vt}\|^m_{\vWfracb{2-\frac{1}{s}}{s}}\right).
\end{aligned}
\end{equation*}
This finally shows that both the terms in the right hand side go to $0$ as $\alpha \rightarrow \infty$. Hence the result. 
\hfill
\end{proof}

\bibliographystyle{plain}
\bibliography{Master_bibfile}

\end{document}